\documentclass{amsart}

\usepackage[utf8]{inputenc}
\usepackage[T1]{fontenc}
\usepackage{lmodern}

\usepackage{amssymb}

\usepackage{tikz}

\theoremstyle{plain}
\newtheorem*{maintheorem}{Main Theorem}
\newtheorem{theoremL}{Theorem}
\newtheorem{corollaryL}{Corollary}[theoremL]
\newtheorem{theorem}{Theorem}
\newtheorem{lemma}[theorem]{Lemma}
\newtheorem{corollary}[theorem]{Corollary}
\newtheorem{proposition}[theorem]{Proposition}

\theoremstyle{definition}

\newtheorem{example}[theorem]{Example}

\theoremstyle{remark}
\newtheorem{remark}[theorem]{Remark}
\newtheorem{claim}[theorem]{Claim}
\newtheorem*{acknowledgments}{Acknowledgments}

\usetikzlibrary{arrows.meta}

\numberwithin{equation}{section}

\newcommand{\C}{\text{c}}
\newcommand{\mc}{\text{mc}}
\newcommand{\R}{\text{r}}

\DeclareMathOperator{\Aff}{Aff}
\DeclareMathOperator{\Mod}{mod}
\DeclareMathOperator{\Mult}{Mult}
\DeclareMathOperator{\Poly}{Poly}
\DeclareMathOperator{\res}{res}

\begin{document}

\title{Moduli spaces of polynomial maps and multipliers at small cycles}

\author{Valentin Huguin}
\address{Department of Mathematics, University of Toronto, Toronto, ON M5S 2E4, Canada}
\email{valentin.huguin@utoronto.ca}
\thanks{The research of the author was partly supported by the German Research Foundation (DFG, project number 455038303).}


\begin{abstract}
Fix an integer $d \geq 2$. The space $\mathcal{P}_{d}$ of polynomial maps of degree $d$ modulo conjugation by affine transformations is naturally an affine variety over $\mathbb{Q}$ of dimension $d -1$. For each integer $P \geq 1$, the elementary symmetric functions of the multipliers at all the cycles with period $p \in \lbrace 1, \dotsc, P \rbrace$ induce a natural morphism $\Mult_{d}^{(P)}$ defined on $\mathcal{P}_{d}$. In this article, we show that the morphism $\Mult_{d}^{(2)}$ induced by the multipliers at the cycles with periods $1$ and $2$ is both finite and birational onto its image. In the case of polynomial maps, this strengthens results by McMullen and by Ji and Xie stating that $\Mult_{d}^{(P)}$ is quasifinite and birational onto its image for all sufficiently large integers $P$. Our result arises as the combination of the following two statements:
\begin{itemize}
\item A sequence of polynomials over $\mathbb{C}$ of degree $d$ with bounded multipliers at its cycles with periods $1$ and $2$ is necessarily bounded in $\mathcal{P}_{d}(\mathbb{C})$.
\item A generic conjugacy class of polynomials over $\mathbb{C}$ of degree $d$ is uniquely determined by its multipliers at its cycles with periods $1$ and $2$.
\end{itemize}
\end{abstract}

\maketitle

\section{Introduction}

Fix any integer $d \geq 2$. We wish here to describe the space $\Poly_{d}$ of polynomials of degree $d$ from a dynamical perspective. The group $\Aff$ of affine transformations acts on $\Poly_{d}$ by conjugation, via $\phi \centerdot f = \phi \circ f \circ \phi^{-1}$. Since conjugate polynomials induce the same dynamical system, up to some change of coordinates, it is natural to consider the space $\mathcal{P}_{d} = \Poly_{d}/\Aff$ of conjugacy classes of polynomial maps of degree $d$. This quotient space $\mathcal{P}_{d}$ is naturally an affine variety over $\mathbb{Q}$ of dimension $d -1$ called the \emph{moduli space of polynomial maps} of degree $d$. As an affine variety, $\mathcal{P}_{d}$ is completely determined by its coordinate ring $\mathbb{Q}\left[ \mathcal{P}_{d} \right]$, consisting of all regular functions defined on $\mathcal{P}_{d}$. From a dynamical point of view, natural regular functions on $\mathcal{P}_{d}$ are given by the elementary symmetric functions of the multipliers at all the cycles with any given period. Thus, it is natural to ask how well the multipliers at the cycles describe the moduli space $\mathcal{P}_{d}$.

Suppose that $K$ is an arbitrary algebraically closed field of characteristic $0$ and $f \in \Poly_{d}(K)$. Recall that a point $z_{0} \in K$ is \emph{periodic} for $f$ if there is some integer $p \geq 1$ such that $f^{\circ p}\left( z_{0} \right) = z_{0}$. In this case, the smallest such integer $p$ is called the \emph{period} of $z_{0}$ and $\left\lbrace f^{\circ n}\left( z_{0} \right) : n \geq 0 \right\rbrace$ is said to be a \emph{cycle} for $f$. The \emph{multiplier} of $f$ at $z_{0}$ is the number $\lambda = \left( f^{\circ p} \right)^{\prime}\left( z_{0} \right) \in K$. The polynomial $f$ has the same multiplier at each point of the cycle. Moreover, the multiplier is invariant under conjugation: if $\phi \in \Aff(K)$, then $\phi\left( z_{0} \right)$ is periodic for $\phi \centerdot f$ with period $p$ and multiplier $\lambda$.

Given $f \in \Poly_{d}(K)$, with $K$ an algebraically closed field of characteristic $0$, for $p \geq 1$, we denote by $\Lambda_{f}^{(p)} \in K^{N_{d}^{(p)}}/\mathfrak{S}_{N_{d}^{(p)}}$ the multiset of multipliers of $f$ at all its cycles with period $p$, which depends only on the conjugacy class $[f] \in \mathcal{P}_{d}(K)$. We ask here if $f \in \Poly_{d}(K)$, with $K$ an algebraically closed field of characteristic $0$, is characterized by its multiplier spectrum $\left( \Lambda_{f}^{(p)} \right)_{p \geq 1}$. Since $\mathcal{P}_{d}$ is finite-dimensional, we can even ask if $\mathcal{P}_{d}$ is well described by the multipliers at the cycles with period at most $P$ for some integer $P \geq 1$.

For $n \geq 1$, we denote by $\mathbb{A}^{n}$ the affine space of dimension $n$ over $\mathbb{Q}$. For $N \geq 1$ and every algebraically closed field $K$, we have a natural bijection $K^{N}/\mathfrak{S}_{N} \cong K^{N}$ given by the elementary symmetric functions. Via these identifications, for $P \geq 1$, we consider the \emph{multiplier spectrum morphism} $\Mult_{d}^{(P)} \colon \mathcal{P}_{d} \rightarrow \prod\limits_{p = 1}^{P} \mathbb{A}^{N_{d}^{(p)}}$ given by \[ \Mult_{d}^{(P)}\left( [f] \right) = \left( \Lambda_{f}^{(1)}, \dotsc, \Lambda_{f}^{(P)} \right) \, \text{.} \]

Here, we show that the multipliers at the cycles with periods $1$ and $2$ provide a good description of the space $\mathcal{P}_{d}$. More precisely, our main result is the following:

\begin{maintheorem}
Assume that $d \geq 2$ is an integer. Then $\Mult_{d}^{(2)}$ induces a finite birational morphism from $\mathcal{P}_{d}$ onto its image $\Sigma_{d}^{(2)}$. Furthermore, if $d \in \lbrace 2, 3 \rbrace$, then $\Mult_{d}^{(1)}$ induces an isomorphism from $\mathcal{P}_{d}$ onto its image $\Sigma_{d}^{(1)}$.
\end{maintheorem}

\begin{remark}
It follows from Main Theorem that $\Mult_{d}^{(P)}$ induces a finite birational morphism from $\mathcal{P}_{d}$ onto its image $\Sigma_{d}^{(P)}$ for all $P \geq 2$. Moreover, if $d \in \lbrace 2, 3 \rbrace$, then $\Mult_{d}^{(P)}$ induces an isomorphism from $\mathcal{P}_{d}$ onto its image $\Sigma_{d}^{(P)}$ for all $P \geq 1$.
\end{remark}

\begin{remark}
In general, $\Mult_{d}^{(P)}$ may not be an isomorphism onto its image $\Sigma_{d}^{(P)}$ for any $P \geq 1$ (see Appendix~\ref{appendix:isospec} for details).
\end{remark}

In fact, Main~Theorem is a direct combination of Theorems~\ref{theorem:degenA} and~\ref{theorem:unique}, which we now present.

\subsection{Degeneration of complex polynomial maps and multipliers at small cycles}

It is natural to investigate the behavior of multipliers under degeneration in the space $\mathcal{P}_{d}(\mathbb{C})$ of affine conjugacy classes of complex polynomials of degree $d$. This study was notably conducted by DeMarco and McMullen in~\cite{DMMM2008}. We ask here if degeneration in the space $\mathcal{P}_{d}(\mathbb{C})$ is always detected by the multipliers at the cycles with small periods.

The set $\mathcal{P}_{d}(\mathbb{C})$ of conjugacy classes of complex polynomial maps of degree $d$ is naturally a complex orbifold of dimension $d -1$. We say that a sequence $\left( f_{n} \right)_{n \geq 0}$ of elements of $\Poly_{d}(\mathbb{C})$ \emph{degenerates} in $\mathcal{P}_{d}(\mathbb{C})$ if the sequence $\left( \left[ f_{n} \right] \right)_{n \geq 0}$ eventually leaves every compact subset of $\mathcal{P}_{d}(\mathbb{C})$. We can also express degeneration in $\mathcal{P}_{d}(\mathbb{C})$ in terms of maximal escape rates.

Given an algebraically closed field $K$ of characteristic $0$ equipped with an absolute value $\lvert . \rvert$ and $f \in \Poly_{d}(K)$, the \emph{Green function} $g_{f} \colon K \rightarrow \mathbb{R}_{\geq 0}$ is given by \[ g_{f}(z) = \lim_{n \rightarrow +\infty} \frac{1}{d^{n}} \log^{+}\left\lvert f^{\circ n}(z) \right\rvert \] and the \emph{maximal escape rate} $M_{f}$ of $f$ is defined by \[ M_{f} = \max\left\lbrace g_{f}(c) : c \in K, \, f^{\prime}(c) = 0 \right\rbrace \, \text{.} \] The maximal escape rate is invariant under conjugation. Moreover, in the complex case, the maximal escape rate characterizes degeneration in $\mathcal{P}_{d}(\mathbb{C})$. More precisely, any sequence $\left( f_{n} \right)_{n \geq 0}$ of elements of $\Poly_{d}(\mathbb{C})$ degenerates in $\mathcal{P}_{d}(\mathbb{C})$ if and only if $\lim\limits_{n \rightarrow +\infty} M_{f_{n}} = +\infty$.

Now, given an algebraically closed field $K$ of characteristic $0$ equipped with an absolute value $\lvert . \rvert$ and $f \in \Poly_{d}(K)$, for $p \geq 1$, we define \[ M_{f}^{(p)} = \max_{\lambda \in \Lambda_{f}^{(p)}} \left( \frac{1}{p} \log\lvert \lambda \rvert \right) \] to be the maximal characteristic exponent of $f$ at a cycle with period $p$.

By~\cite{O2012}, for every algebraically closed valued field $K$ of characteristic $0$ and every $f \in \Poly_{d}(K)$, we have \[ \lim_{p \rightarrow +\infty} \frac{1}{N_{d}^{(p)}} \sum_{\lambda \in \Lambda_{f}^{(p)}} \left( \frac{1}{p} \log\lvert \lambda \rvert \right) = \log\lvert d \rvert +\sum_{c \in \Gamma_{f}} \rho_{c} \cdot g_{f}(c) \, \text{,} \] where $\Gamma_{f} \subseteq K$ is the set of critical points for $f$ and $\rho_{c} \geq 1$ denotes the multiplicity of $c$ as a critical point for $f$ for each $c \in \Gamma_{f}$, which yields $\sup\limits_{p \geq 1} M_{f}^{(p)} \geq M_{f} +\log\lvert d \rvert$. In particular, if $\left( f_{n} \right)_{n \geq 0}$ is any sequence of elements of $\Poly_{d}(\mathbb{C})$ that degenerates in $\mathcal{P}_{d}(\mathbb{C})$, then $\lim\limits_{n \rightarrow +\infty} \sup\limits_{p \geq 1} M_{f_{n}}^{(p)} = +\infty$. Thus, degeneration in $\mathcal{P}_{d}(\mathbb{C})$ is detected by the full multiplier spectrum.

We show here that degeneration in $\mathcal{P}_{d}(\mathbb{C})$ is already detected by the multipliers at the cycles with periods $1$ and $2$. Explicitly, we obtain the following:

\begin{theoremL}
\label{theorem:degenA}
Assume that $d \geq 2$ is an integer and $\left( f_{n} \right)_{n \geq 0}$ is any sequence of elements of $\Poly_{d}(\mathbb{C})$ that degenerates in $\mathcal{P}_{d}(\mathbb{C})$. Then \[ \lim_{n \rightarrow +\infty} \max\left\lbrace M_{f_{n}}^{(1)}, M_{f_{n}}^{(2)} \right\rbrace = +\infty \, \text{.} \] Furthermore, if $d \in \lbrace 2, 3 \rbrace$, then $\lim\limits_{n \rightarrow +\infty} M_{f_{n}}^{(1)} = +\infty$.
\end{theoremL}

As a consequence of Theorem~\ref{theorem:degenA}, the morphism $\Mult_{d}^{(2)}$ given by the multipliers at the cycles with periods $1$ and $2$ is proper, and hence finite since $\mathcal{P}_{d}$ is an affine variety. In the polynomial case, this strengthens a result established by McMullen in~\cite{MM1987} stating that $\Mult_{d}^{(P)}$ is a quasifinite morphism for some integer $P \geq 1$. By contrast, Fujimura proved in~\cite{F2007} that the morphism $\Mult_{d}^{(1)}$ induced by the multipliers at the fixed points is neither quasifinite nor surjective onto its scheme-theoretic image $\Sigma_{d}^{(1)}$ if $d \geq 4$.

Then, using the fact that $\Mult_{d}^{(2)}$ is a finite morphism, we generalize Theorem~\ref{theorem:degenA} to polynomials over any algebraically closed valued field of characteristic $0$. More precisely, we obtain the following stronger result:

\begin{corollaryL}
\label{corollary:degenLocal}
Suppose that $d \geq 2$ is an integer and $K$ is an algebraically closed valued field of characteristic $0$. Then there exist $A \in \mathbb{R}_{> 0}$ and $B \in \mathbb{R}$ such that \[ \max\left\lbrace M_{f}^{(1)}, M_{f}^{(2)} \right\rbrace \geq A \cdot M_{f} +B \] for all $f \in \Poly_{d}(K)$. Moreover, if $d \in \lbrace 2, 3 \rbrace$, then there exist $A \in \mathbb{R}_{> 0}$ and $B \in \mathbb{R}$ such that $M_{f}^{(1)} \geq A \cdot M_{f} +B$ for all $f \in \Poly_{d}(K)$.
\end{corollaryL}

We also deduce a relation between the critical height of any polynomial defined over a number field and the standard heights of its multipliers at the small cycles. Here, denote by $h \colon \overline{\mathbb{Q}} \rightarrow \mathbb{R}_{\geq 0}$ the standard height on the algebraic closure $\overline{\mathbb{Q}}$ of $\mathbb{Q}$. Given $f \in \Poly_{d}\left( \overline{\mathbb{Q}} \right)$, define $H_{f}$ to be the critical height of $f$ and, for $p \geq 1$, define $H_{f}^{(p)} = \max\limits_{\lambda \in \Lambda_{f}^{(p)}} h(\lambda)$ (see Subsection~\ref{subsection:corollaries} for details). We obtain the following:

\begin{corollaryL}
\label{corollary:degenGlobal}
Suppose that $d \geq 2$ is an integer. Then there exist $A \in \mathbb{R}_{> 0}$ and $B \in \mathbb{R}$ such that \[ \max\left\lbrace H_{f}^{(1)}, H_{f}^{(2)} \right\rbrace \geq A \cdot H_{f} +B \] for all $f \in \Poly_{d}\left( \overline{\mathbb{Q}} \right)$. Moreover, if $d \in \lbrace 2, 3 \rbrace$, then there exist $A \in \mathbb{R}_{> 0}$ and $B \in \mathbb{R}$ such that $H_{f}^{(1)} \geq A \cdot H_{f} +B$ for all $f \in \Poly_{d}\left( \overline{\mathbb{Q}} \right)$.
\end{corollaryL}

In fact, we obtain Theorem~\ref{theorem:degenA} as a direct consequence of the quantitative result below. In most cases, this is a stronger version of Corollary~\ref{corollary:degenLocal}. Furthermore, the bounds in our statement are optimal.

\begin{theoremL}
\label{theorem:degenB}
Assume that $d \geq 2$ is an integer, $K$ is an algebraically closed valued field of characteristic $0$ that is either Archimedean or non-Archimedean with residue characteristic $0$ or greater than $d$ and $f \in \Poly_{d}(K)$. If $d \geq 4$, then \[ M_{f}^{(1)} \geq \frac{d -1}{d -2} M_{f} \quad \text{or} \quad M_{f}^{(2)} \geq C_{d} \cdot M_{f} \, \text{,} \quad \text{with} \quad C_{d} = \begin{cases} \frac{2 (d -1)}{d} & \text{if } d \text{ is even}\\ \frac{2 d}{d +1} & \text{if } d \text{ is odd} \end{cases} \, \text{.} \] In addition, $M_{f}^{(1)} \geq M_{f}$ if $d = 2$, and $M_{f}^{(1)} \geq 2 M_{f}$ if $d = 3$.
\end{theoremL}

\begin{remark}
Alternatively, one can establish Theorem~\ref{theorem:degenA} by relating bounded multipliers to rescalings for sequences of polynomial maps and by counting some critical points. Suppose that $\left( f_{n} \right)_{n \geq 0}$ is a sequence of elements of $\Poly_{d}(\mathbb{C})$. We say that a sequence $\left( \phi_{n} \right)_{n \geq 0}$ of elements of $\Aff(\mathbb{C})$ is a \emph{rescaling} for $\left( f_{n} \right)_{n \geq 0}$ with period $p \geq 1$ and degree $e \geq 2$ if $\left( \phi_{n} \circ f_{n}^{\circ p} \circ \phi_{n}^{-1} \right)_{n \geq 0}$ converges locally uniformly on $\mathbb{C}$ to some $g \in \Poly_{e}(\mathbb{C})$. We say that two rescalings $\left( \phi_{n} \right)_{n \geq 0}$ and $\left( \psi_{n} \right)_{n \geq 0}$ are \emph{independent} if, for each bounded subset $D$ of $\mathbb{C}$, we have $\phi_{n}^{-1}(D) \cap \psi_{n}^{-1}(D) = \varnothing$ for all sufficiently large $n$. More precisely, to prove Theorem~\ref{theorem:degenA}, one can proceed as follows: Assume that both $\sup\limits_{n \geq 0} M_{f_{n}}^{(1)} < +\infty$ and $\sup\limits_{n \geq 0} M_{f_{n}}^{(2)} < +\infty$. Then, according to the discussion in~\cite[Section~2]{FT2008}, possibly passing to some subsequence, $\left( f_{n} \right)_{n \geq 0}$ has pairwise independent rescalings $\left( \phi_{j, n} \right)_{n \geq 0}$ with period $1$ and respective degrees $d_{j} \geq 2$, with $r \geq 1$ and $j \in \lbrace 1, \dotsc, r \rbrace$, such that $\sum\limits_{j = 1}^{r} d_{j} = d$. If $d \in \lbrace 2, 3 \rbrace$, then $r = 1$, and hence $\left( f_{n} \right)_{n \geq 0}$ does not degenerate in $\mathcal{P}_{d}(\mathbb{C})$. Thus, assume that $d \geq 4$. Using again the arguments of~\cite[Section~2]{FT2008} and passing to a subsequence if necessary, $\left( f_{n} \right)_{n \geq 0}$ has pairwise independent rescalings $\left( \psi_{k, n} \right)_{n \geq 0}$ with period $2$ and respective degrees $e_{k} \geq 2$, with $s \geq 0$ and $k \in \lbrace 1, \dotsc, s \rbrace$, that are independent from all the $\left( \phi_{j, n} \right)_{n \geq 0}$, with $j \in \lbrace 1, \dotsc, r \rbrace$, and such that $\sum\limits_{j = 1}^{r} d_{j}^{2} +\sum\limits_{k = 1}^{s} e_{k} = d^{2}$. Finally, given a sufficiently large bounded subset $D$ of $\mathbb{C}$, define $N \geq 0$ to be the number of critical points for $f_{n}^{\circ 2}$ in $\bigcup\limits_{k = 1}^{s} \psi_{k, n}^{-1}(D)$ for all sufficiently large $n$, counting multiplicities. Then one can show that $N \geq 2 (d -1) (r -1)$ and $N \leq (d +1) (r -1)$, which yields $r = 1$. Thus, the sequence $\left( f_{n} \right)_{n \geq 0}$ does not degenerate in $\mathcal{P}_{d}(\mathbb{C})$. We refer the reader to~\cite{F2024} for a concise and complete proof of Theorem~\ref{theorem:degenA} using analogous arguments, which considers the dynamical system induced by a meromorphic family of rational maps on a certain Berkovich projective line (see~\cite{K2015} for the correspondence between rescalings and periodic points of type~II in this Berkovich space).
\end{remark}

\begin{remark}
While we provide lower bounds on $\max\left\lbrace M_{f}^{(1)}, M_{f}^{(2)} \right\rbrace$ in terms of $M_{f}$, with $K$ an algebraically closed valued field of characteristic $0$ and $f \in \Poly_{d}(K)$, it is not difficult to prove reverse inequalities. In fact, for every $p \geq 1$, one can easily establish an upper bound on $M_{f}^{(p)}$ in terms of $M_{f}$, with $K$ an algebraically closed valued field of characteristic $0$ and $f \in \Poly_{d}(K)$ (see Appendix~\ref{appendix:bounds} for details).
\end{remark}

\begin{remark}
Actually, one can prove that, if $\left( f_{n} \right)_{n \geq 0}$ is any sequence of elements of $\Poly_{d}(\mathbb{C})$ that degenerates in $\mathcal{P}_{d}(\mathbb{C})$, then $\lim\limits_{n \rightarrow +\infty} M_{f_{n}}^{(2)} = +\infty$. Furthermore, under the hypotheses of Theorem~\ref{theorem:degenB}, one can show that $M_{f}^{(2)} \geq M_{f}$ if $M_{f} > 0$.
\end{remark}

\subsection{Determination of generic conjugacy classes of polynomial maps by their multipliers at their small cycles}

It is natural to ask how many polynomial maps have the same multipliers, up to conjugation.

Fujimura showed in~\cite{F2007} that the induced morphism $\Mult_{d}^{(1)} \colon \mathcal{P}_{d} \rightarrow \Sigma_{d}^{(1)}$ has degree $(d -2)!$ (see also~\cite{S2017} and~\cite{S2023}). Thus, generically, there are exactly $(d -2)!$ elements of $\mathcal{P}_{d}(\mathbb{C})$ that have the same multiset of multipliers at their fixed points. Here, we prove that a generic element of $\mathcal{P}_{d}(\mathbb{C})$ is uniquely determined by its multipliers at its cycles with periods $1$ and $2$.

\begin{theoremL}
\label{theorem:unique}
Assume that $d \geq 2$ is an integer. Then there is a nonempty Zariski-open subset $U$ of $\mathcal{P}_{d}$ such that each $[f] \in U(\mathbb{C})$ is the unique $[g] \in \mathcal{P}_{d}(\mathbb{C})$ such that $\Lambda_{g}^{(1)} = \Lambda_{f}^{(1)}$ and $\Lambda_{g}^{(2)} = \Lambda_{f}^{(2)}$. Moreover, if $d \in \lbrace 2, 3 \rbrace$, then each $[f] \in \mathcal{P}_{d}(\mathbb{C})$ is the unique $[g] \in \mathcal{P}_{d}(\mathbb{C})$ such that $\Lambda_{g}^{(1)} = \Lambda_{f}^{(1)}$.
\end{theoremL}

In other words, Theorem~\ref{theorem:unique} states that the morphism $\Mult_{d}^{(2)}$ is birational onto its image $\Sigma_{d}^{(2)}$. This proves a conjecture made by Hutz and Tepper in~\cite{HT2013}, who had checked it when $d \in \lbrace 2, 3, 4, 5 \rbrace$. This also strengthens a recent result obtained by Ji and Xie in~\cite{JX2024} asserting that $\Mult_{d}^{(P)}$ is birational onto its image $\Sigma_{d}^{(P)}$ for some integer $P \geq 1$.

\begin{remark}
In general, there may exist distinct elements $[f], [g] \in \mathcal{P}_{d}(\mathbb{C})$ such that $\Lambda_{f}^{(1)} = \Lambda_{g}^{(1)}$ and $\Lambda_{f}^{(2)} = \Lambda_{g}^{(2)}$. For $d = 4$, we can describe precisely when this occurs (see Appendix~\ref{appendix:isospec}).
\end{remark}

\subsection{Known results in the case of rational maps}

The algebraic group $\operatorname{PSL}_{2}$ of M\"{o}bius transformations acts on the space $\operatorname{Rat}_{d}$ of rational maps of degree $d$ by conjugation. The quotient $\mathcal{M}_{d}$ forms an affine variety over $\mathbb{Q}$ of dimension $2 d -2$ called the \emph{moduli space of rational maps} of degree $d$. As in the polynomial setting, for each $P \geq 1$, the elementary symmetric functions of the multipliers at the cycles with period $p \in \lbrace 1, \dotsc, P \rbrace$ define a morphism $\widehat{\Mult}_{d}^{(P)} \colon \mathcal{M}_{d} \rightarrow \prod\limits_{p = 1}^{P} \mathbb{A}^{\widehat{N}_{d}^{(p)}}$.

Milnor showed in~\cite{M1993} that the morphism $\widehat{\Mult}_{2}^{(1)}$ given by the multipliers of quadratic rational maps at the fixed points is an isomorphism onto its image $\widehat{\Sigma}_{2}^{(1)}$.

Unlike the case of polynomial maps, if $d \geq 4$, then the morphism $\widehat{\Mult}_{d}^{(P)}$ is not proper for any $P \geq 1$. For example, if $d$ is a perfect square, taking flexible Latt\`{e}s maps, one can find degenerating sequences in $\mathcal{M}_{d}(\mathbb{C})$ whose elements all have the same multisets of multipliers for each period. As another example, first examined by McMullen in~\cite{MM1988}, the sequence $\left( f_{n} \right)_{n \geq 1}$ of elements of $\operatorname{Rat}_{5}(\mathbb{C})$ defined by $f_{n}(z) = z^{2} +\frac{1}{n z^{3}}$ degenerates in $\mathcal{M}_{5}(\mathbb{C})$ and has uniformly bounded multipliers for each period. We refer to~\cite{L2022} for a description of the hyperbolic components of $\mathcal{M}_{d}(\mathbb{C})$ containing degenerating sequences with uniformly bounded multipliers for each period. We also refer to~\cite{F2024} and~\cite{G2024} for very recent results about the behavior of multipliers under degeneration in $\mathcal{M}_{d}(\mathbb{C})$.

Now, denote by $\mathcal{L}_{d} \subseteq \mathcal{M}_{d}$ the locus of conjugacy classes of flexible Latt\`{e}s maps of degree $d$. Note that $\mathcal{L}_{d}$ is empty if $d$ is not a perfect square, $\mathcal{L}_{d}$ is an irreducible curve if $d$ is an even square and $\mathcal{L}_{d}$ is the union of two irreducible curves if $d$ is an odd square (see~\cite{M2006}). McMullen showed that, aside from flexible Latt\`{e}s maps, any conjugacy class of rational maps of degree $d$ is determined up to finitely many choices by its multiplier spectrum. Thus, McMullen established the following:

\begin{theorem}[{\cite[Corollary~2.3]{MM1987}}]
\label{theorem:quasifinite}
There exists an integer $P \geq 1$ such that the restriction of $\widehat{\Mult}_{d}^{(P)}$ to $\mathcal{M}_{d} \setminus \mathcal{L}_{d}$ is a quasifinite morphism.
\end{theorem}

In~\cite{GOV2020}, Gauthier, Okuyama and Vigny exhibited some integer $P \geq 1$ as in Theorem~\ref{theorem:quasifinite} that can be explicitly computed from quantities related to bifurcations. In addition, Ji and Xie established in~\cite{JX2023} that, aside from flexible Latt\`{e}s maps, any conjugacy class of rational maps of degree $d$ is already determined up to only finitely many choices by its multisets of moduli of multipliers for each period.

Very recently, Ji and Xie also showed that a generic conjugacy class of rational maps of degree $d$ is uniquely determined by its multiplier spectrum.

\begin{theorem}[{\cite[Theorem~1.3]{JX2024}}]
There exists an integer $P \geq 1$ such that the morphism $\widehat{\Mult}_{d}^{(P)}$ is birational onto its scheme-theoretic image $\widehat{\Sigma}_{d}^{(P)}$.
\end{theorem}

Moreover, Ji and Xie conjectured that Theorem~\ref{theorem:unique} also holds for rational maps or, equivalently, that the morphism $\widehat{\Mult}_{d}^{(2)}$ is birational onto its image $\widehat{\Sigma}_{d}^{(2)}$. This conjecture was proved for cubic rational maps by Gotou in~\cite{G2023}.

\subsection{Outline of the paper}

For the reader's convenience, the sections are mostly independent from each other.

In Section~\ref{section:prelim}, we show that the moduli space $\mathcal{P}_{d}$ exists as a geometric quotient, we provide a precise definition of the multiplier spectrum morphisms $\Mult_{d}^{(P)}$, with $P \geq 1$, and we examine the cases of quadratic and cubic polynomials. In Section~\ref{section:degenArch}, we prove Theorem~\ref{theorem:degenB} in the complex setting and we derive Corollaries~\ref{corollary:degenLocal} and~\ref{corollary:degenGlobal} from Theorem~\ref{theorem:degenA}. In Section~\ref{section:unique}, we prove Theorem~\ref{theorem:unique}. In Appendix~\ref{appendix:bounds}, we obtain a few additional estimates on absolute values of multipliers of polynomial maps. In Appendix~\ref{appendix:isospec}, we discuss isospectral polynomial maps and we describe the pairs of quartic polynomials that have the same multipliers for periods $1$ and $2$.

\begin{acknowledgments}
The author would like to thank Xavier Buff for helpful discussions, and in particular for stating a two-islands lemma. The author would like to thank Charles Favre for inquiring about Theorem~\ref{theorem:degenA}, which was the initial motivation for this paper, and for suggesting to investigate the non-Archimedean setting. The author would also like to thank Igors Gorbovickis for helpful discussions.
\end{acknowledgments}

\section{Moduli spaces of polynomial maps and multiplier spectrum morphisms}
\label{section:prelim}

Throughout this section, we fix an integer $d \geq 2$.

\subsection{Moduli spaces of polynomial maps}

First, let us present the space $\mathcal{P}_{d}$ of polynomial maps of degree $d$ modulo conjugation by an affine transformation, and specifically let us describe its algebraic and complex analytic structures. Here, we shall elaborate on the similar discussion in~\cite[Section~2.1]{FG2022}.

We shall start by showing that this space $\mathcal{P}_{d}$ exists as a geometric quotient and affine variety over $\mathbb{Q}$. The analogous statement about the space $\mathcal{M}_{d}$ of all rational maps of degree $d$ modulo conjugation by a M\"{o}bius transformation was obtained by Silverman in~\cite{S1998} by using the geometric invariant theory developed by Mumford in~\cite{MFK1994}.

As we work over a non-algebraically closed field, we shall first briefly recall some material from algebraic geometry. For more details, we refer the reader to~\cite{M2017}, \cite{MFK1994} and~\cite{P2017}.

Here, we call \emph{variety} over $\mathbb{Q}$ any scheme $X$ that is geometrically reduced, separated and of finite type over $\mathbb{Q}$. For each variety $X$ over $\mathbb{Q}$ and each commutative $\mathbb{Q}$\nobreakdash-algebra $R$, we denote by $X(R)$ the set of $R$\nobreakdash-valued points of $X$. For each $n \geq 1$, we denote by $\mathbb{A}^{n}$ the affine space of dimension $n$ over $\mathbb{Q}$.

Here, we call \emph{algebraic group} over $\mathbb{Q}$ any variety $G$ over $\mathbb{Q}$ together with a point $e \in G(\mathbb{Q})$ and morphisms $m \colon G \times G \rightarrow G$ and $\operatorname{inv} \colon G \rightarrow G$ of varieties over $\mathbb{Q}$ that satisfy the usual group axioms, where $e$, $m$ and $\operatorname{inv}$ represent the identity element, the group law and the inversion, respectively. Given any algebraic group $G$ over $\mathbb{Q}$ and any commutative $\mathbb{Q}$\nobreakdash-algebra $R$, the set $G(R)$ has a natural group structure.

Given an algebraic group $G$ over $\mathbb{Q}$ and a variety $X$ over $\mathbb{Q}$, we call \emph{action} of $G$ on $X$ any morphism $\theta \colon G \times X \rightarrow X$ of varieties over $\mathbb{Q}$ such that the induced map $\theta \colon G(R) \times X(R) \rightarrow X(R)$ is a group action for each commutative $\mathbb{Q}$\nobreakdash-algebra $R$.

Suppose that $G$ is an algebraic group over $\mathbb{Q}$, $X$ is a variety over $\mathbb{Q}$ and $G$ acts on $X$ via a morphism $\theta \colon G \times X \rightarrow X$. For each commutative $\mathbb{Q}$\nobreakdash-algebra $R$, denote here by $(g, x) \mapsto g \centerdot x$ the induced group action $\theta \colon G(R) \times X(R) \rightarrow X(R)$. Given a variety $Z$ over $\mathbb{Q}$, we say that a morphism $\Psi \colon X \rightarrow Z$ is \emph{invariant} under the action $\theta$ if, for every commutative $\mathbb{Q}$\nobreakdash-algebra $R$, we have $\Psi(g \centerdot x) = \Psi(x)$ for all $g \in G(R)$ and all $x \in X(R)$. In fact, given an algebraically closed field $K$ of characteristic $0$, any morphism $\Psi \colon X \rightarrow Z$, with $Z$ a variety over $\mathbb{Q}$, is invariant under $\theta$ if and only if $\Psi(g \centerdot x) = \Psi(x)$ for all $g \in G(K)$ and all $x \in X(K)$. If $X$ is affine, we denote by $\mathbb{Q}[X]^{G}$ the commutative $\mathbb{Q}$\nobreakdash-algebra of all regular functions on $X$ that are invariant under $\theta$ when viewed as morphisms from $X$ to $\mathbb{A}^{1}$.

Suppose again that some algebraic group $G$ over $\mathbb{Q}$ acts on a variety $X$ over $\mathbb{Q}$. Here, we call \emph{geometric quotient} of $X$ by $G$ any variety $X/G$ over $\mathbb{Q}$ together with a morphism $\pi \colon X \rightarrow X/G$ that satisfies the following conditions:
\begin{enumerate}
\item for any algebraically closed field $K$ of characteristic $0$, we have \[ (X/G)(K) = X(K)/G(K) \, \text{,} \] in the sense that the induced map $\pi \colon X(K) \rightarrow (X/G)(K)$ is surjective and its fibers are precisely the orbits $\left\lbrace g \centerdot x : g \in G(K) \right\rbrace$, with $x \in X(K)$;
\item the space $X/G$ has the quotient topology: any subset $U$ of $X/G$ is open if and only if $\pi^{-1}(U)$ is an open subset of $X$;
\item for every open subset $U$ of $X/G$, each regular function $\psi$ on $\pi^{-1}(U)$ that is invariant under the induced action of $G$ on $\pi^{-1}(U)$ factors as $\psi = \overline{\psi} \circ \pi$, with $\overline{\psi}$ a regular function on $U$.
\end{enumerate}
If a variety $X/G$ over $\mathbb{Q}$ with a morphism $\pi \colon X \rightarrow X/G$ is a geometric quotient of $X$ by $G$, then it is also a \emph{categorical quotient}: each invariant morphism $\Psi \colon X \rightarrow Z$, with $Z$ a variety over $\mathbb{Q}$, factors as $\Psi = \overline{\Psi} \circ \pi$ in a unique way, with $\overline{\Psi} \colon X/G \rightarrow Z$ a morphism. In particular, a geometric quotient of $X$ by $G$ (if it exists) is unique, up to isomorphism.

We shall now state general results about existence of geometric quotients. First, we have the well-known result below about geometric quotients of affine varieties by finite algebraic groups. We omit here the proof and refer to~\cite[Expos\'{e}~V, Th\'{e}or\`{e}me~4.1]{SGA1970} for a more general statement.

\begin{lemma}
\label{lemma:quotient1}
Suppose that $G$ is some finite algebraic group over $\mathbb{Q}$ that acts on an affine variety $X$ over $\mathbb{Q}$. Then $\mathbb{Q}[X]^{G}$ is a finitely generated commutative $\mathbb{Q}$\nobreakdash-algebra. Moreover, the affine variety $X/G$ over $\mathbb{Q}$ such that $\mathbb{Q}[X/G] = \mathbb{Q}[X]^{G}$ together with the morphism $\pi \colon X \rightarrow X/G$ induced by the inclusion $\mathbb{Q}[X]^{G} \subseteq \mathbb{Q}[X]$ is a geometric quotient of $X$ by $G$. Furthermore, $\pi$ is a finite morphism.
\end{lemma}

Now, suppose that an algebraic group $G$ over $\mathbb{Q}$ acts on a variety $X$ over $\mathbb{Q}$ and $Y$ is a closed subvariety of $X$. We call \emph{stabilizer} of $Y$ the algebraic subgroup $H$ of $G$ such that $H(K) = \left\lbrace g \in G(K) : g \centerdot Y(K) \subseteq Y(K) \right\rbrace$, where $K$ is any algebraically closed field of characteristic $0$. Note that the action of $G$ on $X$ yields an action of $H$ on $Y$. If $X$ is affine and $H$ is finite, then a geometric quotient $Y/H$ of $Y$ by $H$ exists by Lemma~\ref{lemma:quotient1}. Under some additional assumption implying that every orbit $\left\lbrace g \centerdot x : g \in G(K) \right\rbrace$, with $x \in X(K)$, has a nonempty intersection with $Y(K)$, where $K$ is any algebraically closed field of characteristic $0$, a geometric quotient $X/G$ of $X$ by $G$ also exists and $X/G \cong Y/H$. More precisely, we have the following result. We omit the proof and refer to~\cite[Expos\'{e}~V, Lemme~6.1]{SGA1970} for a more general statement expressed in terms of groupoids.

\begin{lemma}
\label{lemma:quotient2}
Suppose that $G$ is some algebraic group over $\mathbb{Q}$ that acts on an affine variety $X$ over $\mathbb{Q}$ via a morphism $\theta \colon G \times X \rightarrow X$. Also assume that there exists a closed subvariety $Y$ of $X$ such that the induced morphism $\theta \colon G \times Y \rightarrow X$ is finite, flat and surjective. Then the stabilizer $H$ of $Y$ is finite and the closed immersion $\imath \colon Y \hookrightarrow X$ induces an isomorphism $\imath^{*} \colon \mathbb{Q}[X]^{G} \rightarrow \mathbb{Q}[Y]^{H}$ of $\mathbb{Q}$\nobreakdash-algebras. Therefore, a geometric quotient $Y/H$ of $Y$ by $H$ exists, $\mathbb{Q}[X]^{G}$ is a finitely generated commutative $\mathbb{Q}$\nobreakdash-algebra and the affine variety $X/G$ over $\mathbb{Q}$ such that $\mathbb{Q}[X/G] = \mathbb{Q}[X]^{G}$ is isomorphic to $Y/H$. Furthermore, $X/G$ together with the morphism $\pi \colon X \rightarrow X/G$ induced by the inclusion $\mathbb{Q}[X]^{G} \subseteq \mathbb{Q}[X]$ is a geometric quotient of $X$ by $G$.
\end{lemma}

We now turn to the construction of the moduli space $\mathcal{P}_{d}$ of polynomial maps of degree $d$. Consider the space \[ \Poly_{d} = \left\lbrace a_{d} z^{d} +\dotsb +a_{1} z +a_{0} : a_{d} \neq 0 \right\rbrace \] of all polynomial maps of degree $d$. Identifying a polynomial with its coefficients, $\Poly_{d}$ is naturally an affine variety over $\mathbb{Q}$ such that \[ \mathbb{Q}\left[ \Poly_{d} \right] = \mathbb{Q}\left[ a_{0}, a_{1}, \dotsc, a_{d}, a_{d}^{-1} \right] \, \text{.} \] Also, the space $\Aff = \lbrace \alpha z +\beta : \alpha \neq 0 \rbrace$ of all affine transformations is an algebraic group over $\mathbb{Q}$ under composition. Moreover, $\Aff$ acts on $\Poly_{d}$ by conjugation, via $\phi \centerdot f = \phi \circ f \circ \phi^{-1}$. We shall prove the existence of some geometric quotient $\mathcal{P}_{d}$ of $\Poly_{d}$ by $\Aff$, which is necessarily unique up to isomorphism.

\begin{remark}
As the algebraic group $\Aff$ is not reductive, one cannot directly apply the geometric invariant theory developed in~\cite{MFK1994} to prove that there exists a geometric quotient $\mathcal{P}_{d}$ of $\Poly_{d}$ by $\Aff$.
\end{remark}

Now, consider the space \[ \Poly_{d}^{\mc} = \left\lbrace z^{d} +b_{d -2} z^{d -2} +\dotsb +b_{1} z +b_{0} \right\rbrace \] of all monic centered polynomial maps of degree $d$. Note that $\Poly_{d}^{\mc}$ is naturally a closed subvariety of $\Poly_{d}$. Moreover, its stabilizer for the action of $\Aff$ on $\Poly_{d}$ by conjugation is the algebraic subgroup $\mu_{d -1} = \left\lbrace \omega : \omega^{d -1} = 1 \right\rbrace$ of $\Aff$, under the identification of $\omega \in \mu_{d -1}$ with $\omega z \in \Aff$. Thus, we have an induced action of $\mu_{d -1}$ on $\Poly_{d}^{\mc}$ by conjugation, which is given by \[ \omega \centerdot \left( z^{d} +\sum_{j = 0}^{d -2} b_{j} z^{j} \right) = z^{d} +\sum_{j = 0}^{d -2} \omega^{1 -j} b_{j} z^{j} \, \text{.} \] Moreover, since $\mu_{d -1}$ is finite and $\Poly_{d}^{\mc}$ is affine, the $\mathbb{Q}$\nobreakdash-algebra $\mathbb{Q}\left[ \Poly_{d}^{\mc} \right]^{\mu_{d -1}}$ is finitely generated and the affine variety $\mathcal{P}_{d}^{\mc}$ over $\mathbb{Q}$ such that \[ \mathbb{Q}\left[ \mathcal{P}_{d}^{\mc} \right] = \mathbb{Q}\left[ \Poly_{d}^{\mc} \right]^{\mu_{d -1}} \] together with the natural morphism $\pi_{d}^{\mc} \colon \Poly_{d}^{\mc} \rightarrow \mathcal{P}_{d}^{\mc}$ is a geometric quotient of $\Poly_{d}^{\mc}$ by $\mu_{d -1}$, according to Lemma~\ref{lemma:quotient1}.

Finally, we shall apply Lemma~\ref{lemma:quotient2} to prove the existence of a geometric quotient $\mathcal{P}_{d} \cong \mathcal{P}_{d}^{\mc}$ of $\Poly_{d}$ by $\Aff$. To do so, let us first prove the statement below, which implies the well-known fact that each polynomial of degree $d$ over an algebraically closed field of characteristic $0$ is conjugate to a monic centered polynomial. Here, we denote by $\theta \colon \Aff \times \Poly_{d}^{\mc} \rightarrow \Poly_{d}$ the morphism induced by the action of $\Aff$ on $\Poly_{d}$ by conjugation.

\begin{claim}
\label{claim:quotient}
The morphism $\theta \colon \Aff \times \Poly_{d}^{\mc} \rightarrow \Poly_{d}$ is finite, flat and surjective.
\end{claim}

\begin{proof}
For simplicity, write \[ R = \mathbb{Q}\left[ \Poly_{d} \right] = \mathbb{Q}\left[ a_{0}, a_{1}, \dotsc, a_{d}, a_{d}^{-1} \right] \] and \[ S = \mathbb{Q}\left[ \Aff \times \Poly_{d}^{\mc} \right] = \mathbb{Q}\left[ \alpha, \alpha^{-1}, \beta, b_{0}, \dotsc, b_{d -2} \right] \, \text{.} \] Also denote by $K$ an algebraic closure of the field of fractions $\mathbb{Q}\left( \Poly_{d} \right)$ of $R$. For each $\xi \in K$ such that $\xi^{d -1} = a_{d}$, we have \[ \left( \xi z +\frac{a_{d -1} \xi}{d \cdot a_{d}} \right) \centerdot \sum_{j = 0}^{d} a_{j} z^{j} = z^{d} +\sum_{j = 0}^{d -2} B_{j}^{(\xi)} z^{j} \in \Poly_{d}^{\mc}\left( R_{\xi} \right) \, \text{,} \] with \[ R_{\xi} = \mathbb{Q}\left[ a_{0}, \dotsc, a_{d -1}, \xi, \xi^{-1} \right] \quad \text{and} \quad B_{0}^{(\xi)}, \dotsc, B_{d -2}^{(\xi)} \in R_{\xi} \, \text{,} \] and hence \[ \left( \frac{1}{\xi} z -\frac{a_{d -1}}{d \cdot a_{d}} \right) \centerdot \left( z^{d} +\sum_{j = 0}^{d -2} B_{j}^{(\xi)} z^{j} \right) = \sum_{j = 0}^{d} a_{j} z^{j} \, \text{.} \] In addition, for each $j \in \lbrace 0, \dotsc, d -2 \rbrace$ and all $\xi_{1}, \xi_{2} \in K$ such that $\xi_{1}^{d -1} = a_{d}$ and $\xi_{2}^{d -1} = a_{d}$, we have $H_{\xi_{1}, \xi_{2}}\left( B_{j}^{\left( \xi_{1} \right)} \right) = B_{j}^{\left( \xi_{2} \right)}$, where $H_{\xi_{1}, \xi_{2}} \colon R_{\xi_{1}} \rightarrow R_{\xi_{2}}$ is the unique $\mathbb{Q}$\nobreakdash-algebra homomorphism such that $H_{\xi_{1}, \xi_{2}}\left( a_{k} \right) = a_{k}$ for all $k \in \lbrace 0, \dotsc, d -1 \rbrace$ and $H_{\xi_{1}, \xi_{2}}\left( \xi_{1} \right) = \xi_{2}$. Therefore, for each $j \in \lbrace 0, \dotsc, d -2 \rbrace$, the polynomial \[ Q_{j}(T) = \prod_{\xi \in K, \, \xi^{d -1} = a_{d}} \left( T -B_{j}^{(\xi)} \right) \] lies in $R[T]$. For $j \in \lbrace 0, \dotsc, d -2 \rbrace$, define $P_{j} = \theta^{*}\left( Q_{j} \right) \in S[T]$, where $\theta^{*} \colon R \rightarrow S$ denotes the $\mathbb{Q}$\nobreakdash-algebra homomorphism induced by $\theta \colon \Aff \times \Poly_{d}^{\mc} \rightarrow \Poly_{d}$. Now, consider the closed subvariety $Z$ of $\Aff \times \Poly_{d}^{\mc}$ given by \[ Z = \left\lbrace \alpha^{d -1} = \theta^{*}\left( \frac{1}{a_{d}} \right) \right\rbrace \cap \left\lbrace \beta = \theta^{*}\left( \frac{-a_{d -1}}{d \cdot a_{d}} \right) \right\rbrace \cap \bigcap_{j = 0}^{d -2} \left\lbrace P_{j}\left( b_{j} \right) = 0 \right\rbrace \, \text{.} \] Then the morphism $\theta \colon Z \rightarrow \Poly_{d}$ is finite and surjective. In particular, we have \[ \dim(Z) \geq \dim\left( \Poly_{d} \right) = d +1 = \dim\left( \Aff \times \Poly_{d}^{\mc} \right) \, \text{,} \] and hence $Z = \Aff \times \Poly_{d}^{\mc}$ since $\Aff \times \Poly_{d}^{\mc}$ is irreducible. Thus, the morphism $\theta \colon \Aff \times \Poly_{d}^{\mc} \rightarrow \Poly_{d}$ is both finite and surjective. Finally, as $\Aff \times \Poly_{d}^{\mc}$ and $\Poly_{d}$ are both smooth over $\mathbb{Q}$, the morphism $\theta$ is also flat by the miracle flatness theorem (see~\cite[Theorem~23.1]{M1986}). Thus, the claim is proved.
\end{proof}

By Lemma~\ref{lemma:quotient2} and Claim~\ref{claim:quotient}, the closed immersion $\imath \colon \Poly_{d}^{\mc} \hookrightarrow \Poly_{d}$ induces an isomorphism $\imath^{*} \colon \mathbb{Q}\left[ \Poly_{d} \right]^{\Aff} \rightarrow \mathbb{Q}\left[ \Poly_{d}^{\mc} \right]^{\mu_{d -1}}$ of $\mathbb{Q}$\nobreakdash-algebras, and thus we have the commutative diagram below.
\begin{center}
\begin{tikzpicture}
\node (M00) at (0,0) {$\mathbb{Q}\left[ \Poly_{d} \right]^{\Aff}$};
\node (M01) at (4,0) {$\mathbb{Q}\left[ \Poly_{d}^{\mc} \right]^{\mu_{d -1}}$};
\node (M10) at (0,-2) {$\mathbb{Q}\left[ \Poly_{d} \right]$};
\node (M11) at (4,-2) {$\mathbb{Q}\left[ \Poly_{d}^{\mc} \right]$};
\draw[->,>=stealth,line width=0.5pt] (M00) to node[below]{$\sim$} node[above]{$\imath^{*}$} (M01);
\draw[->>,>=stealth,line width=0.5pt] (M10) to node[below]{$\imath^{*}$} (M11);
\draw[{Hooks[length=3pt,width=8pt,right]}->,>=stealth,line width=0.5pt] (M00) to (M10);
\draw[{Hooks[length=3pt,width=8pt,right]}->,>=stealth,line width=0.5pt] (M01) to (M11);
\end{tikzpicture}
\end{center}
Moreover, the affine variety $\mathcal{P}_{d}$ over $\mathbb{Q}$ given by $\mathbb{Q}\left[ \mathcal{P}_{d} \right] = \mathbb{Q}\left[ \Poly_{d} \right]^{\Aff}$ together with the morphism $\pi_{d} \colon \Poly_{d} \rightarrow \mathcal{P}_{d}$ induced by the inclusion $\mathbb{Q}\left[ \Poly_{d} \right]^{\Aff} \subseteq \mathbb{Q}\left[ \Poly_{d} \right]$ is a geometric quotient of $\Poly_{d}$ by $\Aff$. This variety $\mathcal{P}_{d}$ has dimension $d -1$ and is called the \emph{moduli space of polynomial maps} of degree $d$. We have the commutative diagram below.
\begin{center}
\begin{tikzpicture}
\node (M00) at (0,0) {$\Poly_{d}^{\mc}$};
\node (M01) at (4,0) {$\Poly_{d}$};
\node (M10) at (0,-2) {$\mathcal{P}_{d}^{\mc}$};
\node (M11) at (4,-2) {$\mathcal{P}_{d}$};
\draw[{Hooks[length=3pt,width=8pt,right]}->,>=stealth,line width=0.5pt] (M00) to node[above]{$\imath$} (M01);
\draw[->,>=stealth,line width=0.5pt] (M10) to node[below]{$\sim$} (M11);
\draw[->>,>=stealth,line width=0.5pt] (M00) to node[left]{$\pi_{d}^{\mc}$} (M10);
\draw[->>,>=stealth,line width=0.5pt] (M01) to node[right]{$\pi_{d}$} (M11);
\end{tikzpicture}
\end{center}
For $f \in \Poly_{d}(R)$, with $R$ a commutative $\mathbb{Q}$\nobreakdash-algebra, we write $[f] = \pi_{d}(f) \in \mathcal{P}_{d}(R)$.

\begin{remark}
One can explicitly describe the inverse of the $\mathbb{Q}$\nobreakdash-algebra isomorphism $\imath^{*} \colon \mathbb{Q}\left[ \Poly_{d} \right]^{\Aff} \rightarrow \mathbb{Q}\left[ \Poly_{d}^{\mc} \right]^{\mu_{d -1}}$. Suppose that $\psi \in \mathbb{Q}\left[ \Poly_{d}^{\mc} \right]^{\mu_{d -1}}$. Write \[ \mathbb{Q}\left[ \Poly_{d} \right] = \mathbb{Q}\left[ a_{0}, a_{1}, \dotsc, a_{d}, a_{d}^{-1} \right] \quad \text{and} \quad \mathbb{Q}\left[ \Poly_{d}^{\mc} \right] = \mathbb{Q}\left[ b_{0}, \dotsc, b_{d -2} \right] \, \text{,} \] denote by $\overline{\mathbb{Q}\left( \Poly_{d} \right)}$ the algebraic closure of the field of fractions of $\mathbb{Q}\left[ \Poly_{d} \right]$ and choose any $\xi \in \overline{\mathbb{Q}\left( \Poly_{d} \right)}$ such that $\xi^{d -1} = a_{d}$. Now, define \[ \varphi = \psi\left( B_{0}^{(\xi)}, \dotsc, B_{d -2}^{(\xi)} \right) \, \text{,} \] with $B_{0}^{(\xi)}, \dotsc, B_{d -2}^{(\xi)} \in \overline{\mathbb{Q}\left( \Poly_{d} \right)}$ as in the proof of Claim~\ref{claim:quotient}. Then $\varphi \in \mathbb{Q}\left[ \Poly_{d} \right]$. Moreover, $\varphi$ is the unique element of $\mathbb{Q}\left[ \Poly_{d} \right]^{\Aff}$ such that $\imath^{*}(\varphi) = \psi$.
\end{remark}

\begin{remark}
For every field $K$ of characteristic $0$, the base change $\Aff_{K}$ of $\Aff$ to $K$ also acts on the base change $\left( \Poly_{d} \right)_{K}$ of $\Poly_{d}$ to $K$ by conjugation, and the base change $\left( \mathcal{P}_{d} \right)_{K}$ of $\mathcal{P}_{d}$ to $K$ is a geometric quotient of $\left( \Poly_{d} \right)_{K}$ by $\Aff_{K}$.
\end{remark}

Finally, we shall briefly describe the complex analytic structure of $\mathcal{P}_{d}(\mathbb{C})$. The set $\mathcal{P}_{d}(\mathbb{C})$ of complex polynomial maps of degree $d$ modulo conjugation by complex affine transformations is naturally a complex analytic space of dimension $d -1$. In fact, since $\Poly_{d}(\mathbb{C}) \cong \mathbb{C}^{d} \times \mathbb{C}^{*}$ is a complex manifold and the action of $\Aff(\mathbb{C})$ on $\Poly_{d}(\mathbb{C})$ by conjugation is proper, faithful and its stabilizers are all finite, $\mathcal{P}_{d}(\mathbb{C})$ is a complex orbifold. Moreover, $\mathcal{P}_{d}^{\mc}(\mathbb{C})$ is also a complex orbifold, and we have a natural biholomorphism $\mathcal{P}_{d}(\mathbb{C}) \cong \mathcal{P}_{d}^{\mc}(\mathbb{C})$. The complex topology of $\mathcal{P}_{d}(\mathbb{C})$ is the quotient topology: any subset $U$ of $\mathcal{P}_{d}(\mathbb{C})$ is open if and only if $\pi_{d}^{-1}(U)$ is an open subset of $\Poly_{d}(\mathbb{C})$. We refer to~\cite{C2022} for further information about orbifolds.

We say that a sequence $\left( f_{n} \right)_{n \geq 0}$ of elements of $\Poly_{d}(\mathbb{C})$ \emph{degenerates} in $\mathcal{P}_{d}(\mathbb{C})$ if, for every compact subset $K$ of $\mathcal{P}_{d}(\mathbb{C})$, we have $\left[ f_{n} \right] \in \mathcal{P}_{d}(\mathbb{C}) \setminus K$ for all sufficiently large $n$. Thus, any sequence $\left( f_{n} \right)_{n \geq 0}$ of elements of $\Poly_{d}(\mathbb{C})$ degenerates in $\mathcal{P}_{d}(\mathbb{C})$ if and only if there does not exist any sequence $\left( \phi_{n} \right)_{n \geq 0}$ of elements of $\Aff(\mathbb{C})$ such that $\left( \phi_{n} \centerdot f_{n} \right)_{n \geq 0}$ has a convergent subsequence in $\Poly_{d}(\mathbb{C})$. We can also express degeneration in $\mathcal{P}_{d}(\mathbb{C})$ in terms of maximal escape rates. Explicitly, any sequence $\left( f_{n} \right)_{n \geq 0}$ of elements of $\Poly_{d}(\mathbb{C})$ degenerates in the moduli space $\mathcal{P}_{d}(\mathbb{C})$ if and only if $\lim\limits_{n \rightarrow +\infty} M_{f_{n}} = +\infty$ (see~\cite[Proposition~3.6]{BH1988}).

\subsection{Multiplier spectrum morphisms}

Now, let us give a precise definition of the morphisms $\Mult_{d}^{(P)}$, with $P \geq 1$. To do this, we shall first recall the notions of dynatomic and multiplier polynomials associated with a polynomial map. We refer the reader to~\cite{MP1994} and~\cite{VH1992} for further details.

Suppose that $R$ is any $\mathbb{Q}$\nobreakdash-algebra that is an integral domain and $f \in \Poly_{d}(R)$. Then there is a unique sequence $\left( \Phi_{f}^{(p)} \right)_{p \geq 1}$ of elements of $R[z]$ such that, for each $p \geq 1$, we have \[ f^{\circ p}(z) -z = \prod_{k \mid p} \Phi_{f}^{(k)}(z) \, \text{.} \] For $p \geq 1$, the polynomial $\Phi_{f}^{(p)} \in R[z]$ is called the $p$th \emph{dynatomic polynomial} of $f$. For every $p \geq 1$, we have $\deg\left( \Phi_{f}^{(p)} \right) = \nu_{d}^{(p)}$, where \[ \nu_{d}^{(p)} = \sum_{k \mid p} \mu\left( \frac{p}{k} \right) d^{k} \] and $\mu \colon \mathbb{Z}_{\geq 1} \rightarrow \lbrace -1, 0, 1 \rbrace$ denotes the M\"{o}bius function.

The result below gives the relation between the periodic points of a polynomial map and its dynatomic polynomials.

\begin{proposition}[{\cite[Proposition~3.2]{MS1995}}]
\label{proposition:dynatomic}
Assume that $R$ is any $\mathbb{Q}$\nobreakdash-algebra that is an integral domain, $f \in \Poly_{d}(R)$ and $p \geq 1$. Then $z_{0} \in R$ is a root of $\Phi_{f}^{(p)}$ if and only if either $z_{0}$ is a periodic point for $f$ with period $p$ or $z_{0}$ is a periodic point for $f$ with period a proper divisor $k$ of $p$ and multiplier a primitive $\frac{p}{k}$th root of unity.
\end{proposition}

Suppose that $R$ is any $\mathbb{Q}$\nobreakdash-algebra that is an integral domain and $f \in \Poly_{d}(R)$. For every $p \geq 1$, there exists a unique monic polynomial $\chi_{f}^{(p)} \in R[\lambda]$ such that \[ \chi_{f}^{(p)}(\lambda)^{p} = a_{d}^{-m_{d}^{(p)}} \res_{z}\left( \Phi_{f}^{(p)}(z), \lambda -\left( f^{\circ p} \right)^{\prime}(z) \right) \, \text{,} \] where $a_{d} \in R^{*}$ denotes the leading coefficient of $f$, $\res_{z}$ denotes the resultant with respect to $z$ and \[ m_{d}^{(p)} = \begin{cases} d -1 & \text{if } p = 1\\ \frac{\nu_{d}^{(p)} \left( d^{p} -1 \right)}{d -1} & \text{if } p \geq 2 \end{cases} \, \text{.} \] For $p \geq 1$, the polynomial $\chi_{f}^{(p)} \in R[\lambda]$ is called the $p$th \emph{multiplier polynomial} of $f$. For every $p \geq 1$, we have $\deg\left( \chi_{f}^{(p)} \right) = N_{d}^{(p)}$, where $N_{d}^{(p)} = \frac{\nu_{d}^{(p)}}{p}$.

For an algebraically closed field $K$ of characteristic $0$, $f \in \Poly_{d}(K)$ and $p \geq 1$, we denote by $\Lambda_{f}^{(p)} \in K^{N_{d}^{(p)}}/\mathfrak{S}_{N_{d}^{(p)}}$ the multiset of roots of $\chi_{f}^{(p)}$.

Using Proposition~\ref{proposition:dynatomic}, we immediately obtain the result below, which relates the multiplier polynomials of a polynomial map to its multipliers.

\begin{proposition}
Assume that $K$ is an algebraically closed field of characteristic $0$, $f \in \Poly_{d}(K)$ and $p \geq 1$. Then $\lambda \in K$ lies in $\Lambda_{f}^{(p)}$ if and only if at least one of the following two conditions is satisfied:
\begin{itemize}
\item $\lambda$ is a multiplier of $f$ at a cycle with period $p$;
\item $\lambda = 1$ and $f$ has a cycle with period a proper divisor $k$ of $p$ and multiplier a primitive $\frac{p}{k}$th root of unity.
\end{itemize}
In particular, if $f$ has no parabolic cycle with period dividing $p$, then $\Lambda_{f}^{(p)}$ consists precisely of the multipliers of $f$ at its cycles with period $p$.
\end{proposition}

Now, consider the generic polynomial \[ \boldsymbol{f}(z) = \sum_{j = 0}^{d} a_{j} z^{j} \in \Poly_{d}\left( \mathbb{Q}\left[ \Poly_{d} \right] \right) \, \text{.} \] For $p \geq 1$, write \[ \chi_{\boldsymbol{f}}^{(p)}(\lambda) = \lambda^{N_{d}^{(p)}} +\sum_{j = 1}^{N_{d}^{(p)}} (-1)^{j} \boldsymbol{\sigma}_{d, j}^{(p)} \lambda^{N_{d}^{(p)} -j} \in \mathbb{Q}\left[ \Poly_{d} \right][\lambda] \, \text{.} \] Specializing, for every $\mathbb{Q}$\nobreakdash-algebra $R$ that is an integral domain, every $f \in \Poly_{d}(R)$ and every $p \geq 1$, we have \[ \chi_{f}^{(p)}(\lambda) = \lambda^{N_{d}^{(p)}} +\sum_{j = 1}^{N_{d}^{(p)}} (-1)^{j} \boldsymbol{\sigma}_{d, j}^{(p)}(f) \lambda^{N_{d}^{(p)} -j} \, \text{.} \] Thus, for every algebraically closed field $K$ of characteristic $0$, every $f \in \Poly_{d}(K)$ and every $p \geq 1$, the $\sigma_{d, j}^{(p)}(f)$, with $j \in \left\lbrace 1, \dotsc, N_{d}^{(p)} \right\rbrace$, are the elementary symmetric functions of the elements of $\Lambda_{f}^{(p)}$. As the multiplier is invariant under conjugation, it follows that the regular function $\boldsymbol{\sigma}_{d, j}^{(p)} \in \mathbb{Q}\left[ \Poly_{d} \right]$ is invariant under the action of $\Aff$ on $\Poly_{d}$ by conjugation for each $p \geq 1$ and each $j \in \left\lbrace 1, \dotsc, N_{d}^{(p)} \right\rbrace$. Therefore, for every $p \geq 1$ and every $j \in \left\lbrace 1, \dotsc, N_{d}^{(p)} \right\rbrace$, there exists a unique regular function $\sigma_{d, j}^{(p)} \in \mathbb{Q}\left[ \mathcal{P}_{d} \right]$ such that $\sigma_{d, j}^{(p)}\left( [f] \right) = \boldsymbol{\sigma}_{d, j}^{(p)}(f)$ for each commutative $\mathbb{Q}$\nobreakdash-algebra $R$ and each $f \in \Poly_{d}(R)$. For $P \geq 1$, we define the \emph{multiplier spectrum morphism} \[ \Mult_{d}^{(P)} = \left( \left( \sigma_{d, j}^{(1)} \right)_{1 \leq j \leq N_{d}^{(1)}}, \dotsc, \left( \sigma_{d, j}^{(P)} \right)_{1 \leq j \leq N_{d}^{(P)}} \right) \colon \mathcal{P}_{d} \rightarrow \prod_{p = 1}^{P} \mathbb{A}^{N_{d}^{(p)}} \, \text{.} \] For $P \geq 1$, we denote by $\Sigma_{d}^{(P)}$ the scheme-theoretic image of $\Mult_{d}^{(P)}$, which equals the Zariski-closure of $\Mult_{d}^{(P)}\left( \mathcal{P}_{d}(\mathbb{Q}) \right)$ in $\prod\limits_{p = 1}^{P} \mathbb{A}^{N_{d}^{(p)}}$.

Finally, we shall recall a few facts about the multipliers at the fixed points. For every algebraically closed field $K$ of characteristic $0$ and every $f \in \Poly_{d}(K)$ such that $\lambda \neq 1$ for all $\lambda \in \Lambda_{f}^{(1)}$, we have \[ \sum_{\lambda \in \Lambda_{f}^{(1)}} \frac{1}{1 -\lambda} = 0 \, \text{.} \] The relation above is known as the holomorphic fixed-point formula. Therefore, in $\mathbb{Q}\left[ \mathcal{P}_{d} \right]$, we have \[ d +\sum_{j = 1}^{d} (-1)^{j} (d -j) \sigma_{d, j}^{(1)} = 0 \, \text{.} \] In fact, denoting by $s_{1}, \dotsc, s_{d}$ the standard coordinates on $\mathbb{A}^{d}$, we have \[ \Sigma_{d}^{(1)} = \left\lbrace d +\sum_{j = 1}^{d} (-1)^{j} (d -j) s_{j} = 0 \right\rbrace \subseteq \mathbb{A}^{d} \, \text{,} \] and Fujimura showed in~\cite{F2007} that the morphism $\Mult_{d}^{(1)} \colon \mathcal{P}_{d} \rightarrow \Sigma_{d}^{(1)}$ has degree $(d -2)!$. In addition, Fujimura also proved that the morphism $\Mult_{d}^{(1)} \colon \mathcal{P}_{d} \rightarrow \Sigma_{d}^{(1)}$ is neither surjective nor quasifinite when $d \geq 4$. We refer the reader to Sugiyama's articles~\cite{S2017} and~\cite{S2023} for the exact number of conjugacy classes $[f] \in \mathcal{P}_{d}(\mathbb{C})$ that satisfy $\Lambda_{f}^{(1)} = \Lambda$, for each $\Lambda \in \mathbb{C}^{d}/\mathfrak{S}_{d}$ such that $\lambda \neq 1$ for all $\lambda \in \Lambda$.

\subsection{The cases of quadratic and cubic polynomial maps}

To conclude this section, let us briefly describe the moduli spaces $\mathcal{P}_{2}$ and $\mathcal{P}_{3}$ of quadratic and cubic polynomial maps, respectively, and show that the morphisms $\Mult_{2}^{(1)}$ and $\Mult_{3}^{(1)}$ induced by the multipliers at the fixed points are isomorphisms onto their images. The results presented here are well known.

\begin{example}
We first study the case of quadratic polynomial maps. Since we have a natural isomorphism $\mathcal{P}_{2} \cong \mathcal{P}_{2}^{\mc}$, we may restrict our attention to monic centered quadratic polynomials. As $\mu_{1} = \lbrace 1 \rbrace$ is the trivial algebraic group, we have \[ \mathcal{P}_{2}^{\mc} \cong \Poly_{2}^{\mc} = \left\lbrace z^{2} +a_{0} \right\rbrace \quad \text{and} \quad \mathbb{Q}\left[ \mathcal{P}_{2}^{\mc} \right] = \mathbb{Q}\left[ \Poly_{2}^{\mc} \right] = \mathbb{Q}\left[ a_{0} \right] \, \text{.} \] Now, computing the polynomial $\chi_{f}^{(1)} \in \mathbb{Q}\left[ \mathcal{P}_{2}^{\mc} \right][\lambda]$ for $f(z) = z^{2} +a_{0} \in \Poly_{2}^{\mc}$, we obtain \[ \sigma_{2, 1}^{(1)} = 2 \in \mathbb{Q}\left[ \mathcal{P}_{2}^{\mc} \right] \quad \text{and} \quad \sigma_{2, 2}^{(1)} = 4 a_{0} \in \mathbb{Q}\left[ \mathcal{P}_{2}^{\mc} \right] \] via the natural isomorphism $\mathcal{P}_{2} \cong \mathcal{P}_{2}^{\mc}$. Therefore, we have \[ \mathbb{Q}\left[ \Sigma_{2}^{(1)} \right] = \mathbb{Q}\left[ \sigma_{2, 1}^{(1)}, \sigma_{2, 2}^{(1)} \right] = \mathbb{Q}\left[ a_{0} \right] = \mathbb{Q}\left[ \mathcal{P}_{2}^{\mc} \right] \, \text{,} \] where $\Sigma_{2}^{(1)}$ denotes the image of the morphism $\Mult_{2}^{(1)} = \left( \sigma_{2, 1}^{(1)}, \sigma_{2, 2}^{(1)} \right) \colon \mathcal{P}_{2}^{\mc} \rightarrow \mathbb{A}^{2}$. Thus, $\Mult_{2}^{(1)}$ induces an isomorphism from $\mathcal{P}_{2}$ onto its image $\Sigma_{2}^{(1)}$.
\end{example}

\begin{example}
We now turn to the case of cubic polynomial maps. A similar discussion can be found in~\cite[Appendix~A]{M1992}. As $\mathcal{P}_{3} \cong \mathcal{P}_{3}^{\mc}$, we restrict our attention to monic centered cubic polynomials. Recall that \[ \Poly_{3}^{\mc} = \left\lbrace z^{3} +a_{1} z +a_{0} \right\rbrace \] and that the algebraic group $\mu_{2} = \lbrace \pm 1 \rbrace$ acts on $\Poly_{3}^{\mc}$ by \[ \omega \centerdot \left( z^{3} +a_{1} z +a_{0} \right) = z^{3} +a_{1} z +\omega a_{0} \, \text{.} \] Therefore, we have \[ \mathbb{Q}\left[ \mathcal{P}_{3}^{\mc} \right] = \mathbb{Q}\left[ \Poly_{3}^{\mc} \right]^{\mu_{2}} = \mathbb{Q}[\alpha, \beta] \, \text{,} \quad \text{with} \quad \alpha = a_{1} \quad \text{and} \quad \beta = a_{0}^{2} \, \text{.} \] Now, for simplicity, write $s_{j} = \sigma_{3, j}^{(1)}$ for $j \in \lbrace 1, 2, 3 \rbrace$. Computing $\chi_{f}^{(1)} \in \mathbb{Q}\left[ \mathcal{P}_{3}^{\mc} \right][\lambda]$ for $f(z) = z^{3} +a_{1} z +a_{0} \in \Poly_{3}^{\mc}$, we obtain \[ s_{1} = -3 \alpha +6 \, \text{,} \quad s_{2} = -6 \alpha +9 \quad \text{and} \quad s_{3} = 4 \alpha^{3} -12 \alpha^{2} +9 \alpha +27 \beta \] via the natural isomorphism $\mathcal{P}_{3} \cong \mathcal{P}_{3}^{\mc}$, which yields \[ \alpha = \frac{-1}{3} s_{1} +2 \quad \text{and} \quad \beta = \frac{4}{729} s_{1}^{3} -\frac{4}{81} s_{1}^{2} +\frac{1}{9} s_{1} +\frac{1}{27} s_{3} -\frac{2}{27} \, \text{.} \] Therefore, we have \[ \mathbb{Q}\left[ \Sigma_{3}^{(1)} \right] = \mathbb{Q}\left[ s_{1}, s_{2}, s_{3} \right] = \mathbb{Q}[\alpha, \beta] = \mathbb{Q}\left[ \mathcal{P}_{3}^{\mc} \right] \, \text{,} \] where $\Sigma_{3}^{(1)}$ denotes the image of the morphism $\Mult_{3}^{(1)} = \left( s_{1}, s_{2}, s_{3} \right) \colon \mathcal{P}_{3}^{\mc} \rightarrow \mathbb{A}^{3}$. Thus, $\Mult_{3}^{(1)}$ induces an isomorphism from $\mathcal{P}_{3}$ onto its image $\Sigma_{3}^{(1)}$.
\end{example}

\section{Multipliers at small cycles and maximal escape rates for complex polynomial maps}
\label{section:degenArch}

In this section, we shall first prove Theorem~\ref{theorem:degenB} in the complex case. Our proof is inspired by the article~\cite{DMMM2008} by DeMarco and McMullen. It relies on a combinatorial argument and on an inequality relating the modulus of the multiplier at a repelling periodic point and the modulus of some annulus. As the latter is already known, the main novelty here is a combinatorial result concerning sublevel sets of the Green function of a polynomial map with disconnected Julia set. Nonetheless, we provide a detailed proof of Theorem~\ref{theorem:degenB} in the complex setting for completeness and to exhibit the similarity with our proof in the non-Archimedean case. Finally, we shall close this section by deriving Corollaries~\ref{corollary:degenLocal} and~\ref{corollary:degenGlobal} from Theorem~\ref{theorem:degenA}.

We fix here an integer $d \geq 2$. In this section and in the next one, we sometimes use the letter $e$ to denote the degree of certain maps, but we never write $e$ for the exponential function, which we denote by $\exp$.

\subsection{The Green function of a complex polynomial map}

First, let us recall some well-known facts regarding the Green function of a complex polynomial map. We refer to~\cite[Chapter~III, Section~4]{CG1993} and~\cite[Expos\'{e}~VIII, Section~I]{DH1984} for further information.

Suppose that $f \in \Poly_{d}(\mathbb{C})$. Recall that the \emph{filled Julia set} $\mathcal{K}_{f}$ of $f$ is given by \[ \mathcal{K}_{f} = \left\lbrace z \in \mathbb{C} : \sup_{n \geq 0} \left\lvert f^{\circ n}(z) \right\rvert < +\infty \right\rbrace \, \text{.} \] Also recall that the \emph{Green function} $g_{f} \colon \mathbb{C} \rightarrow \mathbb{R}_{\geq 0}$ of $f$ is given by \[ g_{f}(z) = \lim_{n \rightarrow +\infty} \frac{1}{d^{n}} \log^{+}\left\lvert f^{\circ n}(z) \right\rvert \, \text{.} \] This map $g_{f}$ is well defined, continuous and subharmonic on $\mathbb{C}$ and it is harmonic on $\mathbb{C} \setminus \mathcal{K}_{f}$. Moreover, we have $g_{f} \circ f = d \cdot g_{f}$ and $g_{f}(z) = \log\lvert z \rvert +O(1)$ as $z \rightarrow \infty$, and in particular $\left\lbrace g_{f} = 0 \right\rbrace = \mathcal{K}_{f}$. Define the \emph{maximal escape rate} $M_{f}$ of $f$ by \[ M_{f} = \max\left\lbrace g_{f}(c) : c \in \mathbb{C}, \, f^{\prime}(c) = 0 \right\rbrace \, \text{.} \] It follows from the Riemann--Hurwitz formula that the set $\mathcal{K}_{f}$ is connected if and only if $M_{f} = 0$ or, equivalently, if and only if the critical points for $f$ all lie in $\mathcal{K}_{f}$ (see~\cite[Theorem~9.5.1]{B1991}).

Note that $f$ is conjugate to $z \mapsto z^{d}$ near infinity by B\"{o}ttcher's theorem since $\infty$ is a fixed point for $f$ with local degree $d$, viewing $f$ as a rational map $f \colon \widehat{\mathbb{C}} \rightarrow \widehat{\mathbb{C}}$. In fact, there exists a biholomorphism \[ \phi_{f} \colon \left\lbrace g_{f} > M_{f} \right\rbrace \rightarrow \mathbb{C} \setminus \overline{D\left( 0, \exp\left( M_{f} \right) \right)} \] such that $\lim\limits_{z \rightarrow \infty} \phi_{f}(z) = \infty$ and $\phi_{f} \circ f = \phi_{f}^{d}$. This biholomorphism $\phi_{f}$ is unique up to multiplication by a $(d -1)$th root of unity and is called a \emph{B\"{o}ttcher coordinate} of $f$ at infinity. Moreover, we have $g_{f} = \log\left\lvert \phi_{f} \right\rvert$.

Now, we say that a compact subset $K$ of $\mathbb{C}$ is \emph{full} if the set $\mathbb{C} \setminus K$ is connected. It follows from the maximum principle that $\left\lbrace g_{f} \leq \eta \right\rbrace$ is a full compact subset of $\mathbb{C}$ for all $\eta \in \mathbb{R}_{\geq 0}$. In addition, $\left\lbrace g_{f} < \eta \right\rbrace$ is the interior of $\left\lbrace g_{f} \leq \eta \right\rbrace$ for all $\eta \in \mathbb{R}_{> 0}$ as $g_{f}$ has no local maximum on $\mathbb{C} \setminus \mathcal{K}_{f}$. Therefore, for every $\eta \in \mathbb{R}_{> 0}$, the connected components of $\left\lbrace g_{f} < \eta \right\rbrace$ are all bounded simply connected subsets of $\mathbb{C}$. Moreover, for every $\eta \in \mathbb{R}_{> 0}$, the connected components of $\left\lbrace g_{f} < \eta \right\rbrace$ all intersect $\mathcal{K}_{f}$ since $g_{f}$ has no local minimum on $\mathbb{C} \setminus \mathcal{K}_{f}$.

For each $\eta \in \left[ M_{f}, +\infty \right)$, the set $\left\lbrace g_{f} \leq \eta \right\rbrace$ is connected as $\widehat{\mathbb{C}} \setminus \left\lbrace g_{f} \leq \eta \right\rbrace$ is biholomorphic to $\widehat{\mathbb{C}} \setminus \overline{D\left( 0, \exp(\eta) \right)}$ under any B\"{o}ttcher coordinate $\phi_{f}$ of $f$ at infinity. As a result, $\left\lbrace g_{f} < \eta \right\rbrace$ is also connected for all $\eta \in \left( M_{f}, +\infty \right)$. In contrast, we have the following:

\begin{lemma}
\label{lemma:disconnected}
Suppose that $f \in \Poly_{d}(\mathbb{C})$ has a disconnected filled Julia set $\mathcal{K}_{f}$, and define $C \geq 1$ to be the number of critical points $c \in \mathbb{C}$ for $f$ such that $g_{f}(c) = M_{f}$, counting multiplicities. Then $\left\lbrace g_{f} < M_{f} \right\rbrace$ has exactly $C +1$ connected components.
\end{lemma}

\begin{proof}
Denote here by $U_{1}, \dotsc, U_{N}$ the connected components of $\left\lbrace g_{f} < M_{f} \right\rbrace$, with $N \geq 1$. For $j \in \lbrace 1, \dotsc, N \rbrace$, denote by $C_{j} \geq 0$ the number of critical points for $f$ in $U_{j}$, counting multiplicities. Note that $\sum\limits_{j = 1}^{N} C_{j} = d -1 -C$. Now, as \[ \left\lbrace g_{f} < M_{f} \right\rbrace = f^{-1}\left( \left\lbrace g_{f} < d \cdot M_{f} \right\rbrace \right) \, \text{,} \] the map $f \colon U_{j} \rightarrow \left\lbrace g_{f} < d \cdot M_{f} \right\rbrace$ is proper of degree $d_{j} \geq 1$ for each $j \in \lbrace 1, \dotsc, N \rbrace$, and we have $d = \sum\limits_{j = 1}^{N} d_{j}$. In addition, for each $j \in \lbrace 1, \dotsc, N \rbrace$, we have $d_{j} = C_{j} +1$ by the Riemann--Hurwitz formula because $U_{j}$ is simply connected by the previous discussion. Therefore, we have \[ d = \sum_{j = 1}^{N} \left( C_{j} +1 \right) = d -1 -C +N \, \text{,} \] and hence $N = C +1$. Thus, the lemma is proved.
\end{proof}

\begin{remark}
In fact, we shall only use the well-known fact that, if $f \in \Poly_{d}(\mathbb{C})$ has a disconnected filled Julia set $\mathcal{K}_{f}$, then $\left\lbrace g_{f} < M_{f} \right\rbrace$ is disconnected.
\end{remark}

\subsection{A combinatorial argument}

Now, let us count the critical points in certain sublevel sets of the Green function to obtain a result implying that the Julia set of any polynomial map either is connected or has a connected component consisting only of a periodic point with period $1$ or $2$. We shall present this as a consequence of a general two-islands lemma.

To obtain our two-islands lemma, we shall prove a result regarding preimages of simply connected domains under holomorphic maps. To do so, we shall first prove the general fact below.

\begin{lemma}
\label{lemma:connected}
Suppose that $X$ is a topological space that is both connected and locally connected, $A$ is a connected subset of $X$ and $B$ is a clopen subset of $X \setminus A$. Then $A \cup B$ is connected.
\end{lemma}

\begin{proof}
Note that the desired result is immediate if $A = \varnothing$. From now on, suppose that $A \neq \varnothing$. Then it suffices to prove that $A \cup C$ is connected for each connected component $C$ of $B$. Thus, assume that $C$ is a connected component of $B$. Denote by $D$ the connected component of $X \setminus A$ containing $C$. Then $B \cap D = D$ because $B \cap D$ is a nonempty clopen subset of $D$ and $D$ is connected, which yields $D \subseteq B$, and hence $D = C$. Thus, $C$ is a connected component of $X \setminus A$, and in particular \[ \partial C \subseteq \partial (X \setminus A) = \partial A \] since $X$ is locally connected. Furthermore, $\partial C \neq \varnothing$ since, otherwise, $C$ would be a nonempty clopen subset of $X$ contained in $X \setminus A$. Choose $x \in \partial C$. If $x \in A$, then $x \in A \cap \overline{C}$. Now, suppose that $x \in X \setminus A$. Denote by $C^{\prime}$ the connected component of $X \setminus A$ containing $x$. Then $C \cup C^{\prime}$ is connected because $C$ and $C^{\prime}$ are connected and $x \in \overline{C} \cap C^{\prime}$, and hence $C = C^{\prime}$. As a result, $x \in \overline{A} \cap C$. Thus, we have proved that $A \cap \overline{C} \neq \varnothing$ or $\overline{A} \cap C \neq \varnothing$. As $A$ and $C$ are connected, it follows that $A \cup C$ is connected. This completes the proof of the lemma.
\end{proof}

Using the previous lemma, we obtain the general result below, which the author was unable to find in the literature.

\begin{lemma}
\label{lemma:simply}
Suppose that $U, V$ are nonempty simply connected open subsets of $\mathbb{C}$ and $f \colon U \rightarrow \mathbb{C}$ is a holomorphic map. Then every connected component of $f^{-1}(V)$ is simply connected.
\end{lemma}

\begin{proof}
Note that the desired result is immediate if $f$ is constant. Assume now that $f$ is not constant. Recall that a connected open subset $D$ of $\mathbb{C}$ is simply connected if and only if $\widehat{\mathbb{C}} \setminus D$ is connected, where $\widehat{\mathbb{C}}$ is the Riemann sphere. Suppose that $U_{0}$ is a connected component of $f^{-1}(V)$, and let us show that $U_{0}$ is simply connected. Thus, assume that $A$ is a clopen subset of $\widehat{\mathbb{C}} \setminus U_{0}$ that does not contain $\infty$, and let us prove that $A = \varnothing$. Note that $A \setminus U$ is a clopen subset of $\widehat{\mathbb{C}} \setminus U$, which does not contain $\infty$, and hence $A \subseteq U$ since $U \subseteq \mathbb{C}$ is simply connected. Now, define \[ B = f\left( A \cup U_{0} \right) \setminus V \, \text{.} \] Then $B$ is an open subset of $\widehat{\mathbb{C}} \setminus V$ since $A \cup U_{0} = \widehat{\mathbb{C}} \setminus \left( \left( \widehat{\mathbb{C}} \setminus U_{0} \right) \setminus A \right)$ is an open subset of $U$ and the map $f$ is open. Moreover, we have \[ B = f\left( A \setminus f^{-1}(V) \right) \] and $A \setminus f^{-1}(V)$ is a compact subset of $U$ since it is closed in $\widehat{\mathbb{C}}$. It follows that $B$ is compact, and in particular it is also closed in $\widehat{\mathbb{C}} \setminus V$. Therefore, $B = \varnothing$ because $V \subseteq \mathbb{C}$ is simply connected and $B \subseteq \mathbb{C}$, and hence $A \subseteq f^{-1}(V)$. Moreover, $A \cup U_{0}$ is connected by Lemma~\ref{lemma:connected}. Therefore, $A \subseteq U_{0}$ since $U_{0}$ is a connected component of $f^{-1}(V)$, and hence $A = \varnothing$. This completes the proof of the lemma.
\end{proof}

Finally, counting critical points, we deduce the two-islands lemma below. This statement, which greatly simplified the author's exposition, was communicated to him by Buff. For comparison, we refer the reader to~\cite{B2000} for information about the classical Ahlfors five-islands theorem.

\begin{lemma}
\label{lemma:islandsArch}
Suppose that $U, V$ are nonempty simply connected open subsets of $\mathbb{C}$, $f \colon U \rightarrow V$ is a proper holomorphic map and $V_{1}, V_{2}$ are disjoint nonempty simply connected open subsets of $V$. Then there exist an index $j \in \lbrace 1, 2 \rbrace$ and a connected component $U_{j}$ of $f^{-1}\left( V_{j} \right)$ such that $f$ induces a biholomorphism from $U_{j}$ to $V_{j}$.
\end{lemma}

\begin{proof}
Denote by $e \geq 1$ the degree of $f \colon U \rightarrow V$. For $j \in \lbrace 1, 2 \rbrace$, denote by $C_{j} \geq 0$ the number of critical points for $f$ in $f^{-1}\left( V_{j} \right)$, counting multiplicities. Then \[ e = C +1 \geq C_{1} +C_{2} +1 \geq 2 \min\left\lbrace C_{1}, C_{2} \right\rbrace +1 \] by the Riemann--Hurwitz formula, where $C \geq 0$ is the number of critical points for $f$ in $U$, counting multiplicities. Now, for $j \in \lbrace 1, 2 \rbrace$, denote by $U_{j}^{(1)}, \dotsc, U_{j}^{\left( N_{j} \right)}$, with $N_{j} \geq 1$, the connected components of $f^{-1}\left( V_{j} \right)$. Then, for each $j \in \lbrace 1, 2 \rbrace$ and each $\ell \in \left\lbrace 1, \dotsc, N_{j} \right\rbrace$, the map $f \colon U_{j}^{(\ell)} \rightarrow V_{j}$ is proper of degree $e_{j}^{(\ell)} \geq 1$. For $j \in \lbrace 1, 2 \rbrace$, we have $e = \sum\limits_{\ell = 1}^{N_{j}} e_{j}^{(\ell)}$. Moreover, $U_{j}^{(\ell)}$ is a simply connected open subset of $\mathbb{C}$ for all $j \in \lbrace 1, 2 \rbrace$ and all $\ell \in \left\lbrace 1, \dotsc, N_{j} \right\rbrace$ by Lemma~\ref{lemma:simply}. By the Riemann--Hurwitz formula, it follows that $e_{j}^{(\ell)} = C_{j}^{(\ell)} +1$, where $C_{j}^{(\ell)} \geq 0$ is the number of critical points for $f$ in $U_{j}^{(\ell)}$, counting multiplicities, for all $j \in \lbrace 1, 2 \rbrace$ and all $\ell \in \left\lbrace 1, \dotsc, N_{j} \right\rbrace$. Thus, \[ \forall j \in \lbrace 1, 2 \rbrace, \, e = \sum_{\ell = 1}^{N_{j}} \left( C_{j}^{(\ell)} +1 \right) = C_{j} +N_{j} \, \text{.} \] Therefore, there exists $j \in \lbrace 1, 2 \rbrace$ such that $C_{j} < N_{j}$ as, otherwise, we would have $e \leq 2 \min\left\lbrace C_{1}, C_{2} \right\rbrace$. Then $C_{j}^{(\ell)} = 0$ for some $\ell \in \left\lbrace 1, \dotsc, N_{j} \right\rbrace$, and we have $e_{j}^{(\ell)} = 1$. Thus, $f$ induces a biholomorphism from $U_{j}^{(\ell)}$ to $V_{j}$, and the lemma is proved.
\end{proof}

\begin{remark}
\label{remark:islands}
In fact, under the hypotheses of Lemma~\ref{lemma:islandsArch}, one of the following two conditions is satisfied:
\begin{enumerate}
\item\label{item:islandsArch1} there exist an index $j \in \lbrace 1, 2 \rbrace$ and distinct connected components $U_{j}^{(1)}$ and $U_{j}^{(2)}$ of $f^{-1}\left( V_{j} \right)$ that are both mapped biholomorphically onto $V_{j}$ by $f$;
\item\label{item:islandsArch2} there exist connected components $U_{1}$ and $U_{2}$ of $f^{-1}\left( V_{1} \right)$ and $f^{-1}\left( V_{2} \right)$ that are mapped biholomorphically onto $V_{1}$ and $V_{2}$ by $f$, respectively.
\end{enumerate}
Returning to our proof of Lemma~\ref{lemma:islandsArch}, the condition~\eqref{item:islandsArch1} is satisfied if $C_{1} \neq C_{2}$ and the condition~\eqref{item:islandsArch2} is satisfied if $C_{1} = C_{2}$.
\end{remark}

Applying the previous lemma to a dynamical setting, we easily obtain the result below, which is a key ingredient in our proof of Theorem~\ref{theorem:degenB} in the complex case.

\begin{lemma}
\label{lemma:combinatArch}
Suppose that $f \in \Poly_{d}(\mathbb{C})$ has a disconnected filled Julia set $\mathcal{K}_{f}$. Then one of the following two conditions is satisfied:
\begin{enumerate}
\item\label{item:combinatArch1} there exists a connected component $U$ of $\left\lbrace g_{f} < \frac{M_{f}}{d} \right\rbrace$ such that $f$ induces a biholomorphism from $U$ onto the connected component $V$ of $\left\lbrace g_{f} < M_{f} \right\rbrace$ containing $U$;
\item\label{item:combinatArch2} for all distinct connected components $V, V^{\prime}$ of $\left\lbrace g_{f} < M_{f} \right\rbrace$, there exists a connected component $U$ of $\left\lbrace g_{f} < \frac{M_{f}}{d} \right\rbrace$ contained in $V$ such that $f$ induces a biholomorphism from $U$ onto $V^{\prime}$.
\end{enumerate}
In addition, if $d \in \lbrace 2, 3 \rbrace$, then there exists a connected component $V$ of $\left\lbrace g_{f} < M_{f} \right\rbrace$ such that $f$ induces a biholomorphism from $V$ onto $\left\lbrace g_{f} < d \cdot M_{f} \right\rbrace$.
\end{lemma}

\begin{proof}
If the condition~\eqref{item:combinatArch1} is satisfied, we are done. Now, suppose that this is not the case, and let us show that the condition~\eqref{item:combinatArch2} is satisfied. Assume that $V, V^{\prime}$ are two distinct connected components of $\left\lbrace g_{f} < M_{f} \right\rbrace$. Then $f \colon V \rightarrow \left\lbrace g_{f} < d \cdot M_{f} \right\rbrace$ is a proper map and there is no connected component of $\left\lbrace g_{f} < \frac{M_{f}}{d} \right\rbrace$ contained in $V$ that is mapped biholomorphically onto $V$ by $f$. As a result, by Lemma~\ref{lemma:islandsArch}, $f$ maps a connected component of $\left\lbrace g_{f} < \frac{M_{f}}{d} \right\rbrace$ contained in $V$ biholomorphically onto $V^{\prime}$. Thus, the desired result is proved.

Now, assume that $d \in \lbrace 2, 3 \rbrace$. Denote by $V_{1}, \dotsc, V_{N}$, with $N \geq 2$, the connected components of $\left\lbrace g_{f} < M_{f} \right\rbrace$. For $j \in \lbrace 1, \dotsc, N \rbrace$, the map $f \colon V_{j} \rightarrow \left\lbrace g_{f} < d \cdot M_{f} \right\rbrace$ is proper of degree $d_{j} \geq 1$. We have $d = \sum\limits_{j = 1}^{N} d_{j}$. Therefore, as $N \geq 2$ and $d \leq 3$, there exists $j \in \lbrace 1, \dotsc, N \rbrace$ such that $d_{j} = 1$. Thus, $f$ induces a biholomorphism from $V_{j}$ onto $\left\lbrace g_{f} < d \cdot M_{f} \right\rbrace$, and the lemma is proved.
\end{proof}

\begin{remark}
As the connected components of $\left\lbrace g_{f} < \eta \right\rbrace$, with $\eta \in \mathbb{R}_{> 0}$, are known to be simply connected, our proof of Lemma~\ref{lemma:combinatArch} does not actually require Lemma~\ref{lemma:simply} and one could have simply added as an assumption in the statement of Lemma~\ref{lemma:islandsArch} the fact that the connected components of $f^{-1}\left( V_{j} \right)$, with $j \in \lbrace 1, 2 \rbrace$, are all simply connected. Nonetheless, Lemmas~\ref{lemma:simply} and~\ref{lemma:islandsArch} are of interest in their own right.
\end{remark}

\begin{remark}
\label{remark:combinatArch}
Using Remark~\ref{remark:islands}, one can replace the condition~\eqref{item:combinatArch1} in the statement of Lemma~\ref{lemma:combinatArch} by the following one:
\begin{enumerate}
\item[(1')] there exist a connected component $V$ of $\left\lbrace g_{f} < M_{f} \right\rbrace$ and distinct connected components $U, U^{\prime}$ of $\left\lbrace g_{f} < \frac{M_{f}}{d} \right\rbrace$ contained in $V$ such that $f$ maps both $U$ and $U^{\prime}$ biholomorphically onto $V$.
\end{enumerate}
\end{remark}

As a consequence of Lemma~\ref{lemma:combinatArch}, we easily obtain the result below, which is not used in our proof of Theorem~\ref{theorem:degenB} but may be of independent interest.

\begin{proposition}
\label{proposition:singleton}
Assume that $f \in \Poly_{d}(\mathbb{C})$ has a disconnected filled Julia set $\mathcal{K}_{f}$. Then $\mathcal{K}_{f}$ has a connected component that consists only of a periodic point for $f$ with period $1$ or $2$. Furthermore, if $d \in \lbrace 2, 3 \rbrace$, then $\mathcal{K}_{f}$ has a connected component that consists only of a fixed point for $f$.
\end{proposition}

\begin{proof}
Assume for a moment that the condition~\eqref{item:combinatArch2} of Lemma~\ref{lemma:combinatArch} holds. Denote by $V_{1}, \dotsc, V_{N}$, with $N \geq 2$, the connected components of $\left\lbrace g_{f} < M_{f} \right\rbrace$. Then there are connected components $U_{1}, U_{2}$ of $\left\lbrace g_{f} < \frac{M_{f}}{d} \right\rbrace$ contained in $V_{1}, V_{2}$, respectively, such that $f$ induces biholomorphisms from $U_{1}$ to $V_{2}$ and from $U_{2}$ to $V_{1}$. Define $g_{1}$ to be the inverse of $f \colon U_{1} \rightarrow V_{2}$. Then $g_{1}\left( U_{1} \right)$ is a connected component of $\left\lbrace g_{f} < \frac{M_{f}}{d^{2}} \right\rbrace$ contained in $U_{1}$ and the map $f^{\circ 2} \colon g_{1}\left( U_{1} \right) \rightarrow V_{1}$ is a biholomorphism.

Thus, by Lemma~\ref{lemma:combinatArch}, there exist $\eta \in \mathbb{R}_{> 0}$, $p \in \lbrace 1, 2 \rbrace$ and a connect component $U$ of $\left\lbrace g_{f} < \frac{\eta}{d^{p}} \right\rbrace$ such that $f^{\circ p}$ induces a biholomorphism from U onto the connected component $V$ of $\left\lbrace g_{f} < \eta \right\rbrace$ containing $U$, and we can take $p = 1$ if $d \in \lbrace 2, 3 \rbrace$. Now, denote by $g$ the inverse of $f^{\circ p} \colon U \rightarrow V$, and define \[ S = \bigcap_{n \geq 0} g^{\circ n}(U) \, \text{.} \] Let us show that $S$ is a connected component of $\mathcal{K}_{f}$, which consists only of a fixed point for $f^{\circ p}$. Note that $S \subseteq \mathcal{K}_{f}$ because $U$ is bounded. Moreover, $S$ is nonempty and connected since $S = \bigcap\limits_{n \geq 0} g^{\circ n}\left( \overline{U} \right)$ is the intersection of a decreasing sequence of nonempty, connected and compact subsets of $\mathbb{C}$. Now, denote by $C$ the connected component of $\mathcal{K}_{f}$ containing $S$. Then, for every $n \geq 0$, we have \[ f^{\circ p n}(C) \subseteq \mathcal{K}_{f} \subseteq \left\lbrace g_{f} < \frac{\eta}{d^{p}} \right\rbrace \quad \text{and} \quad \varnothing \neq f^{\circ p n}(S) \subseteq f^{\circ p n}(C) \cap U \, \text{,} \] and hence $f^{\circ p n}(C) \subseteq U$ as $f^{\circ p n}(C)$ is connected and $U$ is a connected component of $\left\lbrace g_{f} < \frac{\eta}{d^{p}} \right\rbrace$. It follows by induction that \[ \forall n \geq 0, \, C = g^{\circ n}\left( f^{\circ p n}(C) \right) \subseteq g^{\circ n}(U) \, \text{.} \] Thus, $C = S$. Now, note that $g$ is contracting with respect to the Poincar\'{e} metric on $V$ by the Schwarz lemma, since $U \Subset V$. Therefore, $S$ consists of a single point, which is necessarily fixed for $f^{\circ p}$ because $S$ is invariant under $f^{\circ p}$. This completes the proof of the proposition.
\end{proof}

\begin{remark}
Note that, if $f \in \Poly_{d}(\mathbb{C})$ has a disconnected filled Julia set $\mathcal{K}_{f}$, then $\mathcal{K}_{f}$ has uncountably many connected components. In~\cite{QY2009}, Qiu and Yin proved that, for every $f \in \Poly_{d}(\mathbb{C})$, all but countably many connected components of $\mathcal{K}_{f}$ consist of a single point. In contrast, as shown by McMullen in~\cite{MM1988}, there are rational maps $f \colon \widehat{\mathbb{C}} \rightarrow \widehat{\mathbb{C}}$ whose Julia set $\mathcal{J}_{f}$ is disconnected but has no connected component that consists of a single point. The existence of connected components of the Julia set consisting of a single point for transcendental maps was studied by Dom\'{\i}nguez in~\cite{D1997} and~\cite{D1998}.
\end{remark}

\begin{remark}
In fact, using Remark~\ref{remark:combinatArch}, one can show that, for every $f \in \Poly_{d}(\mathbb{C})$, either $\mathcal{K}_{f}$ is connected or $\mathcal{K}_{f}$ has a connected component consisting only of a periodic point for $f$ with period $2$.
\end{remark}

\subsection{Multipliers and maximal escape rates}

Here, let us obtain lower bounds on moduli on multipliers in terms of maximal escape rates. These bounds depend on combinatorial information regarding sublevel sets of Green functions. A similar discussion can be found in~\cite[Section~4]{DMMM2008}, where the corresponding results are stated in terms of trees. Thus, we give details for the reader's convenience.

First, we shall briefly recall a few necessary elements from conformal geometry. We say that a Riemann surface $A$ is an \emph{annulus} if its fundamental group $\pi_{1}(A)$ is isomorphic to $\mathbb{Z}$. For every $z_{0} \in \mathbb{C}$ and all $r, R \in \mathbb{R}_{> 0}$, with $r < R$, \[ \mathcal{A}_{z_{0}}(r, R) = \left\lbrace z \in \mathbb{C} : r < \left| z -z_{0} \right| < R \right\rbrace \] is an annulus. In fact, every annulus $A$ is biholomorphic to the punctured disk $\mathbb{D}^{*}$, the punctured plane $\mathbb{C}^{*}$ or the round annulus $\mathcal{A}_{0}(1, R)$ for a unique $R \in (1, +\infty)$. We call \emph{modulus} of an annulus $A$ the number \[ \Mod(A) = \begin{cases} \frac{1}{2 \pi} \log(R) & \text{if } A \cong \mathcal{A}_{0}(1, R) \text{, with } R \in (1, +\infty)\\ +\infty & \text{if } A \cong \mathbb{D}^{*} \text{ or } \mathbb{C}^{*} \end{cases} \, \text{.} \] We shall only use the following two facts about moduli of annuli:
\begin{itemize}
\item If $A$ is a Riemann surface, $B$ is an annulus and $f \colon A \rightarrow B$ is a holomorphic covering map of degree $e \geq 1$, then $A$ is an annulus and \[ \Mod(B) = e \cdot \Mod(A) \, \text{.} \]
\item Given an annulus $B$, we say that a nonempty connected open subset $A$ of $B$ is a \emph{subannulus} of $B$ if the inclusion $\imath \colon A \hookrightarrow B$ induces an isomorphism $\imath_{*} \colon \pi_{1}(A) \rightarrow \pi_{1}(B)$. If $B$ is an annulus and $A_{1}, \dotsc, A_{r}$ are pairwise disjoint subannuli of $B$, with $r \geq 0$, then \[ \sum_{j = 1}^{r} \Mod\left( A_{j} \right) \leq \Mod(B) \, \text{.} \] This statement is known as Gr\"{o}tzsch's inequality.
\end{itemize}
We refer the reader to~\cite[Section~3.2]{H2006} for further details about annuli.

Recall that a compact subset $K$ of $\mathbb{C}$ is full if the set $\mathbb{C} \setminus K$ is connected. Note that, if $V$ is a simply connected open subset of $\mathbb{C}$ and $K$ is a nonempty, connected and full compact subset of $V$, then $V \setminus K$ is an annulus. Also note that, if $W \subseteq V$ are simply connected open subsets of $\mathbb{C}$ and $K \subseteq L$ are nonempty, connected and full compact subsets of $W$, then $W \setminus L$ is a subannulus of $V \setminus K$.

We shall crucially use the inequality below, which relates moduli of multipliers at repelling fixed points to moduli of certain annuli.

\begin{lemma}
\label{lemma:modulusArch}
Suppose that $U, V$ are nonempty simply connected open subsets of $\mathbb{C}$ such that $\overline{U}$ is a full compact subset of $V$ and $f \colon U \rightarrow V$ is a biholomorphism. Then $f$ has a unique fixed point $z_{0} \in U$ and we have \[ \frac{1}{2 \pi} \log\left\lvert f^{\prime}\left( z_{0} \right) \right\rvert \geq \Mod\left( V \setminus \overline{U} \right) \, \text{.} \]
\end{lemma}

\begin{proof}
Set $A = V \setminus \overline{U}$. Note that $U, V$ are biholomorphic to the unit disk $\mathbb{D}$. Now, denote by $g \colon V \rightarrow U$ the inverse of $f$. By the Schwarz lemma, as $U \Subset V$, the map $g$ is a contraction with respect to the Poincar\'{e} metric on $V$. In particular, $g$ has a unique fixed point $z_{0} \in U$, which is necessarily attracting for $g$. Then $z_{0}$ is also the unique fixed point for $f$ and $z_{0}$ is repelling for $f$. Now, suppose that $\alpha > \left\lvert f^{\prime}\left( z_{0} \right) \right\rvert$. Then there exists $R \in \mathbb{R}_{> 0}$ such that $\left\lvert f(z) -z_{0} \right\rvert \leq \alpha \left\lvert z -z_{0} \right\rvert$ for all $z \in D\left( z_{0}, R \right)$. Since $g$ is contracting with respect to the Poincar\'{e} metric on $V$, there exists $N \geq 0$ such that $g^{\circ N}(V) \subseteq D\left( z_{0}, R \right)$. Take $r \in (0, R)$ such that $D\left( z_{0}, r \right) \subseteq g^{\circ N}(U)$. Note that $g^{\circ N}(A), \dotsc, g^{\circ n}(A)$ are pairwise disjoint subannuli of $g^{\circ N}(V) \setminus g^{\circ n}\left( \overline{U} \right)$ with modulus $\Mod(A)$ for all $n \geq N$. Moreover, for every $n \geq N$, we have \[ D\left( z_{0}, \frac{r}{\alpha^{n -N}} \right) \subseteq g^{\circ (n -N)}\left( D\left( z_{0}, r \right) \right) \subseteq g^{\circ n}(U) \] since $\alpha > 1$, and hence $g^{\circ N}(V) \setminus g^{\circ n}\left( \overline{U} \right)$ is a subannulus of $\mathcal{A}_{z_{0}}\left( \frac{r}{\alpha^{n -N}}, R \right)$. Thus, by Gr\"{o}tzsch's inequality, we have \[ (n -N +1) \Mod(A) \leq \Mod\left( \mathcal{A}_{z_{0}}\left( \frac{r}{\alpha^{n -N}}, R \right) \right) = \frac{1}{2 \pi} \log\left( \frac{R}{r} \right) +\frac{n -N}{2 \pi} \log(\alpha) \] for all $n \geq N$, which yields $\Mod(A) \leq \frac{1}{2 \pi} \log(\alpha)$ by dividing by $n -N$ and letting $n \rightarrow +\infty$. Letting $\alpha \rightarrow \left\lvert f^{\prime}\left( z_{0} \right) \right\rvert$, we obtain the desired inequality. Thus, the lemma is proved.
\end{proof}

Now, we shall prove the result below (compare~\cite[Lemma~4.6]{DMMM2008}). It is a key point in our proof of Theorem~\ref{theorem:degenB} in the complex case.

\begin{lemma}
\label{lemma:ineqArch}
Suppose that $f \in \Poly_{d}(\mathbb{C})$ has a disconnected filled Julia set $\mathcal{K}_{f}$, $\eta \in \left[ M_{f}, +\infty \right)$, $U_{0}, \dotsc, U_{p -1}$ are (not necessarily distinct) connected components of $\left\lbrace g_{f} < \frac{\eta}{d^{k}} \right\rbrace$, with $k \geq 0$ and $p \geq 1$, $V_{0}, \dotsc, V_{p -1}$ are the connected components of $\left\lbrace g_{f} < \frac{\eta}{d^{k -1}} \right\rbrace$ containing $U_{0}, \dotsc, U_{p -1}$, respectively, and $f$ induces a biholomorphism from $U_{j}$ to $V_{j +1 \pmod{p}}$ for all $j \in \lbrace 0, \dotsc, p -1 \rbrace$. Then $f^{\circ p}$ has a unique fixed point $z_{0} \in \mathbb{C}$ such that $f^{\circ j}\left( z_{0} \right) \in U_{j}$ for all $j \in \lbrace 0, \dotsc, p -1 \rbrace$. Furthermore, we have \[ \log\left\lvert \left( f^{\circ p} \right)^{\prime}\left( z_{0} \right) \right\rvert \geq (d -1) \left( \sum_{j = 0}^{p -1} \frac{1}{d_{j}} \right) \eta \, \text{,} \] where $d_{j}$ denotes the degree of $f^{\circ k} \colon V_{j} \rightarrow \left\lbrace g_{f} < d \cdot \eta \right\rbrace$ for all $j \in \lbrace 0, \dotsc, p -1 \rbrace$.
\end{lemma}

\begin{proof}
Denote by $g_{j}$ the inverse of $f \colon U_{j} \rightarrow V_{j +1 \pmod{p}}$ for $j \in \lbrace 0, \dotsc, p -1 \rbrace$, and define \[ W = g_{0} \circ \dotsb \circ g_{p -1}\left( V_{0} \right) \, \text{.} \] Then $W, V_{0}$ are simply connected open subsets of $\mathbb{C}$ such that $\overline{W}$ is a full compact subset of $V_{0}$ and $f^{\circ p} \colon W \rightarrow V_{0}$ is a biholomorphism. Therefore, by Lemma~\ref{lemma:modulusArch}, the map $f^{\circ p}$ has a unique fixed point $z_{0} \in W$ and we have \[ \log\left\lvert \left( f^{\circ p} \right)^{\prime}\left( z_{0} \right) \right\rvert \geq 2 \pi \Mod\left( V_{0} \setminus \overline{W} \right) \, \text{.} \] Observe that $z_{0}$ is the unique fixed point for $f^{\circ p}$ that satisfies $f^{\circ j}\left( z_{0} \right) \in U_{j}$ for all $j \in \lbrace 0, \dotsc, p -1 \rbrace$. Thus, it remains to prove the desired inequality. Note that the $g_{0} \circ \dotsb \circ g_{j -1}\left( V_{j} \setminus \overline{U_{j}} \right)$, with $j \in \lbrace 0, \dotsc, p -1 \rbrace$, are pairwise disjoint subannuli of $V_{0} \setminus \overline{W}$. Therefore, by Gr\"{o}tzsch's inequality, we have \[ \log\left\lvert \left( f^{\circ p} \right)^{\prime}\left( z_{0} \right) \right\rvert \geq 2 \pi \sum_{j = 0}^{p -1} \Mod\left( V_{j} \setminus \overline{U_{j}} \right) \, \text{.} \] Suppose that $j \in \lbrace 0, \dotsc, p -1 \rbrace$, and let us show that $\Mod\left( V_{j} \setminus \overline{U_{j}} \right) \geq \frac{1}{2 \pi} \left( \frac{d -1}{d_{j}} \right) \eta$. Define $\eta_{j}^{(0)} < \dotsb < \eta_{j}^{\left( N_{j} \right)}$, with $N_{j} \geq 1$, by $\eta_{j}^{(0)} = \frac{\eta}{d^{k}}$ and $\eta_{j}^{(N_{j})} = \frac{\eta}{d^{k -1}}$ and \[ \left\lbrace \eta_{j}^{(1)}, \dotsc, \eta_{j}^{\left( N_{j} -1 \right)} \right\rbrace = \left\lbrace g_{f}(\gamma) : \gamma \in V_{j}, \, g_{f}(\gamma) > \frac{\eta}{d^{k}}, \, \left( f^{\circ k} \right)^{\prime}(\gamma) = 0 \right\rbrace \, \text{.} \] For $\ell \in \left\lbrace 1, \dotsc, N_{j} \right\rbrace$, denote by $D_{j}^{(\ell)}$ the connected component of $\left\lbrace g_{f} < \eta_{j}^{(\ell)} \right\rbrace$ that contains $U_{j}$ and define \[ A_{j}^{(\ell)} = D_{j}^{(\ell)} \setminus \left\lbrace g_{f} \leq \eta_{j}^{(\ell -1)} \right\rbrace \, \text{.} \] Then $f^{\circ k}$ induces a covering map from $A_{j}^{(\ell)}$ to $\left\lbrace d^{k} \eta_{j}^{(\ell -1)} < g_{f} < d^{k} \eta_{j}^{(\ell)} \right\rbrace$ of degree at most $d_{j}$ for each $\ell \in \left\lbrace 1, \dotsc, N_{j} \right\rbrace$. Moreover, $A_j^{(1)}, \dotsc, A_{j}^{\left( N_{j} \right)}$ are pairwise disjoint subannuli of $V_{j} \setminus \overline{U_{j}}$. Therefore, we have \[ \Mod\left( V_{j} \setminus \overline{U_{j}} \right) \geq \frac{1}{d_{j}} \sum_{\ell = 1}^{N_{j}} \Mod\left( \left\lbrace d^{k} \eta_{j}^{(\ell -1)} < g_{f} < d^{k} \eta_{j}^{(\ell)} \right\rbrace \right) \] Finally, denote by $\phi_{f} \colon \left\lbrace g_{f} > M_{f} \right\rbrace \rightarrow \mathbb{C} \setminus \overline{D\left( 0, \exp\left( M_{f} \right) \right)}$ a B\"{o}ttcher coordinate of $f$ at infinity, which satisfies $g_{f} = \log\left\lvert \phi_{f} \right\rvert$. Then $\phi_{f}$ induces a biholomorphism from $\left\lbrace d^{k} \eta_{j}^{(\ell -1)} < g_{f} < \eta_{j}^{(\ell)} \right\rbrace$ onto the round annulus $\mathcal{A}_{0}\left( \exp\left( d^{k} \eta_{j}^{(\ell -1)} \right), \exp\left( d^{k} \eta_{j}^{(\ell)} \right) \right)$ for each $\ell \in \left\lbrace 1, \dotsc, N_{j} \right\rbrace$. Therefore, we have \[ \Mod\left( \left\lbrace d^{k} \eta_{j}^{(\ell -1)} < g_{f} < d^{k} \eta_{j}^{(\ell)} \right\rbrace \right) = \frac{d^{k}}{2 \pi} \left( \eta_{j}^{(\ell)} -\eta_{j}^{(\ell -1)} \right) \] for all $\ell \in \left\lbrace 1, \dotsc, N_{j} \right\rbrace$, and hence \[ \Mod\left( V_{j} \setminus \overline{U_{j}} \right) \geq \frac{d^{k}}{2 \pi d_{j}} \left( \eta_{j}^{\left( N_{j} \right)} -\eta_{j}^{(0)} \right) = \frac{1}{2 \pi} \left( \frac{d -1}{d_{j}} \right) \eta \, \text{.} \] This completes the proof of the lemma.
\end{proof}

\begin{remark}
In order to prove Theorem~\ref{theorem:degenB} in the Archimedean case, we shall only apply Lemma~\ref{lemma:ineqArch} with $\eta = M_{f}$, $k \in \lbrace 0, 1 \rbrace$ and $p \in \lbrace 1, 2 \rbrace$. Nonetheless, our general statement of Lemma~\ref{lemma:ineqArch} also allows us to prove a slightly weaker version of~\cite[Theorem~1.6]{EL1992}, which provides a lower bound on the characteristic exponent at any periodic point in terms of the minimum of the Green function on the set of critical points (see Proposition~\ref{proposition:ineqMin}).
\end{remark}

\subsection{Proof of Theorem~\ref{theorem:degenB} in the Archimedean case}

Finally, let us combine Lemmas~\ref{lemma:combinatArch} and~\ref{lemma:ineqArch} in order to prove Theorem~\ref{theorem:degenB} for polynomial maps defined over an algebraically closed Archimedean valued field.

\begin{proof}[Proof of Theorem~\ref{theorem:degenB} in the Archimedean case]
Suppose that $K$ is an algebraically closed field equipped with an Archimedean absolute value $\lvert . \rvert_{\infty}$ and $f \in \Poly_{d}(K)$. Define $\widehat{K}$ to be the completion of $K$. Then, by Ostrowski's theorem, there exist an embedding $\sigma \colon \widehat{K} \hookrightarrow \mathbb{C}$ and $s \in (0, 1]$ such that $\lvert z \rvert_{\infty} = \left\lvert \sigma(z) \right\rvert^{s}$ for all $z \in \widehat{K}$, where $\lvert . \rvert$ denotes the usual absolute value on $\mathbb{C}$ (see~\cite[Chapter~II, (4.2)]{N1999}). Since $K$ is algebraically closed, the critical points and the periodic points for $\sigma(f)$ in $\mathbb{C}$ all lie in $\sigma(K)$. Therefore, $M_{f} = s \cdot M_{\sigma(f)}$ and $M_{f}^{(p)} = s \cdot M_{\sigma(f)}^{(p)}$ for each $p \geq 1$. Thus, replacing $f$ by $\sigma(f)$ if necessary, we may assume that $f \in \Poly_{d}(\mathbb{C})$.

First, assume that $\mathcal{K}_{f}$ is connected. Then $M_{f} = 0$. Let us prove that $M_{f}^{(1)} \geq 0$. Denote by $\lambda_{1}, \dotsc, \lambda_{d}$ the multipliers of $f$ at its fixed points repeated according to their multiplicities. Then either $\lambda_{j} = 1$ for some $j \in \lbrace 1, \dotsc, d \rbrace$ or $\sum\limits_{j = 1}^{d} \frac{1}{1 -\lambda_{j}} = 0$ by the holomorphic fixed-point formula. Note that, if $\lambda \in D(0, 1)$, then $\Re\left( \frac{1}{1 -\lambda} \right) > \frac{1}{2}$. Therefore, there exists $j \in \lbrace 1, \dotsc, d \rbrace$ such that $\left\lvert \lambda_{j} \right\rvert \geq 1$, and thus $M_{f}^{(1)} \geq 0$.

Thus, assume now that $\mathcal{K}_{f}$ is disconnected. Denote by $V_{1}, \dotsc, V_{N}$, with $N \geq 2$, the connected components of $\left\lbrace g_{f} < M_{f} \right\rbrace$. For $j \in \lbrace 1, \dotsc, N \rbrace$, denote by $d_{j} \geq 1$ the degree of $f \colon V_{j} \rightarrow \left\lbrace g_{f} < d \cdot M_{f} \right\rbrace$. Then $\sum\limits_{j = 1}^{N} d_{j} = d$. Let us consider three cases.

Suppose that $d_{j} = 1$ for some $j \in \lbrace 1, \dotsc, N \rbrace$. Note that this holds if $d \in \lbrace 2, 3 \rbrace$. Then $f$ induces a biholomorphism from $V_{j}$ to $\left\lbrace g_{f} < d \cdot M_{f} \right\rbrace$. Thus, by Lemma~\ref{lemma:ineqArch}, the map $f$ has a unique fixed point $z_{0} \in V_{j}$ and we have $\log\left\lvert f^{\prime}\left( z_{0} \right) \right\rvert \geq (d -1) M_{f}$. In particular, $M_{f}^{(1)} \geq (d -1) M_{f}$.

Now, suppose that the condition~\eqref{item:combinatArch1} of Lemma~\ref{lemma:combinatArch} is satisfied and $d_{j} \geq 2$ for all $j \in \lbrace 1, \dotsc, N \rbrace$. Then there exist $j \in \lbrace 1, \dotsc, N \rbrace$ and a connected component $U_{j}$ of $\left\lbrace g_{f} < \frac{M_{f}}{d} \right\rbrace$ contained in $V_{j}$ such that $f$ induces a biholomorphism from $U_{j}$ to $V_{j}$. Therefore, by Lemma~\ref{lemma:ineqArch}, the map $f$ has a unique fixed point $z_{0} \in U_{j}$ and we have $\log\left\lvert f^{\prime}\left( z_{0} \right) \right\rvert \geq \frac{d -1}{d_{j}} M_{f}$. Moreover, $d_{j} = d -\sum\limits_{k \neq j} d_{k} \leq d -2$. Thus, $M_{f}^{(1)} \geq \frac{d -1}{d -2} M_{f}$.

Finally, suppose that the condition~\eqref{item:combinatArch2} of Lemma~\ref{lemma:combinatArch} is satisfied. Then there are connected components $U_{1}, U_{2}$ of $\left\lbrace g_{f} < \frac{M_{f}}{d} \right\rbrace$ contained in $V_{1}, V_{2}$, respectively, such that $f$ induces biholomorphisms from $U_{1}$ to $V_{2}$ and from $U_{2}$ to $V_{1}$. By Lemma~\ref{lemma:ineqArch}, it follows that the map $f^{\circ 2}$ has a unique fixed point $z_{0} \in \mathbb{C}$ such that $z_{0} \in U_{1}$ and $f\left( z_{0} \right) \in U_{2}$ and we have \[ \log\left\lvert \left( f^{\circ 2} \right)^{\prime}\left( z_{0} \right) \right\rvert \geq (d -1) \left( \frac{1}{d_{1}} +\frac{1}{d_{2}} \right) M_{f} \geq (d -1) \left( \frac{1}{d_{1}} +\frac{1}{d -d_{1}} \right) M_{f} \, \text{.} \] Therefore, $M_{f}^{(2)} \geq C_{d} \cdot M_{f}$, where \[ C_{d} = \min_{j \in \lbrace 1, \dotsc, d -1 \rbrace} \frac{d -1}{2} \left( \frac{1}{j} +\frac{1}{d -j} \right) = \begin{cases} \frac{2 (d -1)}{d} & \text{if } d \text{ is even}\\ \frac{2 d}{d +1} & \text{if } d \text{ is odd} \end{cases} \, \text{.} \] This completes the proof of the theorem.
\end{proof}

\begin{remark}
Using Remark~\ref{remark:combinatArch}, one can easily adapt the proof above to show that $M_{f}^{(2)} \geq M_{f}$ for all $f \in \Poly_{d}(\mathbb{C})$ with disconnected filled Julia set $\mathcal{K}_{f}$.
\end{remark}

Finally, note that Theorem~\ref{theorem:degenA} is proved since it is an immediate consequence of Theorem~\ref{theorem:degenB} in the Archimedean case.

\subsection{Consequences of Theorem~\ref{theorem:degenA}}
\label{subsection:corollaries}

To conclude this section, let us apply here Theorem~\ref{theorem:degenA} to establish results about the morphism $\Mult_{d}^{(2)}$. We shall also deduce Corollaries~\ref{corollary:degenLocal} and~\ref{corollary:degenGlobal}. For brevity, we omit the analogous results concerning the multipliers at the fixed points alone when $d \in \lbrace 2, 3 \rbrace$, as these can be obtained in a completely similar way.

To conclude this section, let us use here Theorem~1.2 to obtain results about the morphism $\Mult_{d}^{(2)}$. We shall also deduce Corollaries~\ref{corollary:degenLocal} and~\ref{corollary:degenGlobal}. For brevity, we omit the analogous results concerning the multipliers at the fixed points alone when $d \in \lbrace 2, 3 \rbrace$, as these can be obtained in a completely similar way.

First, we have the following direct consequence of Theorem~\ref{theorem:degenA}:

\begin{corollary}
\label{corollary:proper}
The holomorphic map $\Mult_{d}^{(2)} \colon \mathcal{P}_{d}(\mathbb{C}) \rightarrow \mathbb{C}^{d} \times \mathbb{C}^{\frac{d (d -1)}{2}}$ is proper.
\end{corollary}

\begin{proof}
Note that, for every $f \in \Poly_{d}(\mathbb{C})$ and every $p \geq 1$, we have \[ \exp\left( p \cdot M_{f}^{(p)} \right) = \max_{\lambda \in \Lambda_{f}^{(p)}} \lvert \lambda \rvert \leq 1 +\max_{j \in \left\lbrace 1, \dotsc, N_{d}^{(p)} \right\rbrace} \left\lvert \sigma_{d, j}^{(p)}\left( [f] \right) \right\rvert \] by Cauchy's bound on roots of a complex polynomial. As a result, by Theorem~\ref{theorem:degenA}, if $\left( f_{n} \right)_{n \geq 0}$ is any sequence of elements of $\Poly_{d}(\mathbb{C})$ that degenerates in $\mathcal{P}_{d}(\mathbb{C})$, then $\lim\limits_{n \rightarrow +\infty} \Mult_{d}^{(2)}\left( \left[ f_{n} \right] \right) = \infty$ in $\mathbb{C}^{d} \times \mathbb{C}^{\frac{d (d -1)}{2}}$. Thus, the corollary is proved.
\end{proof}

To derive further consequences of Theorem~\ref{theorem:degenA}, we shall apply the general result below from algebraic geometry.

\begin{lemma}
\label{lemma:finite}
Suppose that $X, Y$ are two affine varieties over $\mathbb{Q}$ and $\Psi \colon X \rightarrow Y$ is a morphism such that the induced holomorphic map $\Psi \colon X(\mathbb{C}) \rightarrow Y(\mathbb{C})$ is proper. Then $\Psi$ is a finite morphism.
\end{lemma}

\begin{proof}
Denote here by $X_{\mathbb{C}}$ and $Y_{\mathbb{C}}$ the base changes of $X$ and $Y$ to $\mathbb{C}$, respectively. As the holomorphic map $\Psi \colon X(\mathbb{C}) \rightarrow Y(\mathbb{C})$ is proper, the morphism $\Psi \colon X_{\mathbb{C}} \rightarrow Y_{\mathbb{C}}$ is proper (see~\cite[Expos\'{e}~XII, Proposition~3.2]{SGA1971}). Therefore, since $X_{\mathbb{C}}$ and $Y_{\mathbb{C}}$ are affine, $\Psi \colon X_{\mathbb{C}} \rightarrow Y_{\mathbb{C}}$ is a finite morphism (see~\cite[Chapter~3, Lemma~3.17]{Li2002}). As a result, the morphism $\Psi \colon X \rightarrow Y$ is finite (see~\cite[(2.7.1)]{G1965}).
\end{proof}

Combining Corollary~\ref{corollary:proper} and Lemma~\ref{lemma:finite}, we immediately obtain the following:

\begin{corollary}
The morphism $\Mult_{d}^{(2)} \colon \mathcal{P}_{d} \rightarrow \mathbb{A}^{d} \times \mathbb{A}^{\frac{d (d -1)}{2}}$ is finite.
\end{corollary}

We shall now prove Corollaries~\ref{corollary:degenLocal} and~\ref{corollary:degenGlobal}. To do this, we first present general results about polynomial maps with rational coefficients that allow one to deduce facts over arbitrary valued fields from analogous facts in the complex setting.

Given a valued field $K$, we denote here by $\lvert . \rvert_{K}$ its absolute value and, for $n \geq 1$, we denote by $\lVert . \rVert_{K^{n}}$ the norm on $K^{n}$ defined by \[ \lVert \boldsymbol{t} \rVert_{K^{n}} = \max_{j \in \lbrace 1, \dotsc, n \rbrace} \left\lvert t_{j} \right\rvert_{K} \quad \text{for} \quad \boldsymbol{t} = \left( t_{1}, \dotsc, t_{n} \right) \in K^{n} \, \text{.} \] Given a valued field $K$ and an integer $n \geq 1$, we also define \[ (n)_{K} = \begin{cases} n & \text{if } K \text{ is Archimedean}\\ 1 & \text{if } K \text{ is non-Archimedean} \end{cases} \, \text{,} \] so that \[ \forall \boldsymbol{t} \in K^{n}, \, \left\lvert \sum_{j = 1}^{n} t_{j} \right\rvert_{K} \leq (n)_{K} \lVert \boldsymbol{t} \rVert_{K^{n}} \] by the triangle inequality.

\begin{lemma}
\label{lemma:minimum}
Suppose that $\boldsymbol{F} \colon \mathbb{A}^{m} \rightarrow \mathbb{A}^{n}$, with $m, n \geq 1$, is a morphism such that the induced holomorphic map $\boldsymbol{F} \colon \mathbb{C}^{m} \rightarrow \mathbb{C}^{n}$ has no zero in $\mathbb{C}^{m}$. Then, for every valued field $K$ of characteristic $0$, there exist some $\alpha_{K}^{(0)} \in \mathbb{R}_{> 0}$ depending only on the restriction of $\lvert . \rvert_{K}$ to $\mathbb{Q}$ and some $\delta \in \mathbb{R}_{\geq 0}$ not depending on $K$ such that \[ \forall \boldsymbol{t} \in K^{m}, \, \left\lVert \boldsymbol{F}(\boldsymbol{t}) \right\rVert_{K^{n}} \geq \alpha_{K}^{(0)} \max\left\lbrace 1, \lVert \boldsymbol{t} \rVert_{K^{m}} \right\rbrace^{-\delta} \, \text{.} \] Moreover, we can take $\alpha_{K}^{(0)} = 1$ for every non-Archimedean field $K$ with characteristic $0$ and residue characteristic outside a finite set $S^{(0)}$ of prime numbers.
\end{lemma}

\begin{proof}
Write $\boldsymbol{F} = \left( F_{1}, \dotsc, F_{n} \right)$, with $F_{1}, \dotsc, F_{n} \in \mathbb{Q}\left[ T_{1}, \dotsc, T_{m} \right]$. Since $F_{1}, \dotsc, F_{n}$ have no common zero in $\mathbb{C}^{m}$ by hypothesis, it follows from the Nullstellensatz that there exist $G_{1}, \dotsc, G_{n} \in \mathbb{Q}\left[ T_{1}, \dotsc, T_{m} \right]$ such that \[ \sum_{j = 1}^{n} F_{j}\left( T_{1}, \dotsc, T_{m} \right) G_{j}\left( T_{1}, \dotsc, T_{m} \right) = 1 \, \text{.} \] Define $\delta = \max\limits_{j \in \lbrace 1, \dotsc, n \rbrace} \deg\left( G_{j} \right)$. For $j \in \lbrace 1, \dotsc, n \rbrace$, write \[ G_{j}\left( T_{1}, \dotsc, T_{m} \right) = \sum_{\boldsymbol{\ell} \in \lbrace 0, \dotsc, \delta \rbrace^{m}} a_{j, \boldsymbol{\ell}} \prod_{k = 1}^{m} T_{k}^{\ell_{k}} \in \mathbb{Q}\left[ T_{1}, \dotsc, T_{m} \right] \, \text{.} \] Consider the finite set \[ S^{(0)} = \left\lbrace p \text{ prime} : \max_{\substack{j \in \lbrace 1, \dotsc, n \rbrace\\ \boldsymbol{\ell} \in \lbrace 0, \dotsc, \delta \rbrace^{m}}} \left\lvert a_{j, \boldsymbol{\ell}} \right\rvert_{p} \neq 1 \right\rbrace \, \text{,} \] where $\lvert . \rvert_{p}$ denotes the $p$\nobreakdash-adic absolute value on $\mathbb{Q}$ for each prime number $p$. Now, suppose that $K$ is a valued field of characteristic $0$. Then we have \[ 1 \leq (n)_{K} \left\lVert \boldsymbol{F}(\boldsymbol{t}) \right\rVert_{K^{n}} \left( \max_{j \in \lbrace 1, \dotsc, n \rbrace} \left\lvert G_{j}(\boldsymbol{t}) \right\rvert_{K} \right) \leq A_{K}^{(0)} \left\lVert \boldsymbol{F}(\boldsymbol{t}) \right\rVert_{K^{n}} \max\left\lbrace 1, \lVert \boldsymbol{t} \rVert_{K^{m}} \right\rbrace^{\delta} \] for all $\boldsymbol{t} \in K^{m}$ by the triangle inequality, where \[ A_{K}^{(0)} = (n)_{K} \left( (\delta +1)^{m} \right)_{K} \left( \max_{\substack{j \in \lbrace 1, \dotsc, n \rbrace\\ \boldsymbol{\ell} \in \lbrace 0, \dotsc, \delta \rbrace^{m}}} \left\lvert a_{j, \boldsymbol{\ell}} \right\rvert_{K} \right) \in \mathbb{R}_{> 0} \, \text{.} \] Therefore, setting $\alpha_{K}^{(0)} = \left( A_{K}^{(0)} \right)^{-1}$, we have \[ \forall \boldsymbol{t} \in K^{m}, \, \left\lVert \boldsymbol{F}(\boldsymbol{t}) \right\rVert_{K^{n}} \geq \alpha_{K}^{(0)} \max\left\lbrace 1, \lVert \boldsymbol{t} \rVert_{K^{m}} \right\rbrace^{-\delta} \, \text{.} \] Finally, note that $\alpha_{K}^{(0)}$ depends only on the restriction of $\lvert . \rvert_{K}$ to $\mathbb{Q}$. Furthermore, we have $\alpha_{K}^{(0)} = 1$ if $K$ is non-Archimedean with residue characteristic not in $S^{(0)}$. This completes the proof of the lemma.
\end{proof}

Also, we have the following result about the growth of proper polynomial maps at infinity:

\begin{lemma}
\label{lemma:lojasiewcz}
Suppose that $\boldsymbol{F} \colon \mathbb{A}^{m} \rightarrow \mathbb{A}^{n}$, with $m, n \geq 1$, is a morphism such that the induced holomorphic map $\boldsymbol{F} \colon \mathbb{C}^{m} \rightarrow \mathbb{C}^{n}$ is proper. Then, for every valued field $K$ of characteristic $0$, there exist some $\alpha_{K}^{(\infty)}, R_{K}^{(\infty)} \in \mathbb{R}_{> 0}$ depending only on the restriction of $\lvert . \rvert_{K}$ to $\mathbb{Q}$ and some $\beta \in \mathbb{R}_{> 0}$ not depending on $K$ such that \[ \forall \boldsymbol{t} \in K^{m}, \, \lVert \boldsymbol{t} \rVert_{K^{m}} > R_{K}^{(\infty)} \Longrightarrow \left\lVert \boldsymbol{F}(\boldsymbol{t}) \right\rVert_{K^{n}} \geq \alpha_{K}^{(\infty)} \lVert \boldsymbol{t} \rVert_{K^{m}}^{\beta} \, \text{.} \] Moreover, we can take $\alpha_{K}^{(\infty)} = 1$ and $R_{K}^{(\infty)} = 1$ for every non-Archimedean field $K$ with characteristic $0$ and residue characteristic outside a finite set $S^{(\infty)}$ of prime numbers.
\end{lemma}

\begin{proof}
Write $\boldsymbol{F} = \left( F_{1}, \dotsc, F_{n} \right)$, with $F_{1}, \dotsc, F_{n} \in \mathbb{Q}\left[ T_{1}, \dotsc, T_{m} \right]$. As the holomorphic map $\boldsymbol{F} \colon \mathbb{C}^{m} \rightarrow \mathbb{C}^{n}$ is proper by hypothesis, $\boldsymbol{F} \colon \mathbb{A}^{m} \rightarrow \mathbb{A}^{n}$ is a finite morphism by Lemma~\ref{lemma:finite}, and hence $T_{j}$ is integral over $\mathbb{Q}\left[ F_{1}, \dotsc, F_{n} \right]$ for each $j \in \lbrace 1, \dotsc, m \rbrace$. Thus, for each $j \in \lbrace 1, \dotsc, m \rbrace$, there exist $P_{j, 0}, \dotsc, P_{j, D_{j} -1} \in \mathbb{Q}\left[ X_{1}, \dotsc, X_{n} \right]$, with $D_{j} \geq 1$, such that \[ T_{j}^{D_{j}} = \sum_{k = 0}^{D_{j} -1} P_{j, k}\left( F_{1}\left( T_{1}, \dotsc, T_{m} \right), \dotsc, F_{n}\left( T_{1}, \dotsc, T_{m} \right) \right) T_{j}^{k} \, \text{.} \] Define \[ \gamma = \max_{\substack{j \in \lbrace 1, \dotsc, m \rbrace\\ k \in \left\lbrace 0, \dotsc, D_{j} -1 \right\rbrace}} \deg\left( P_{j, k} \right) \in \mathbb{Z}_{\geq 1} \quad \text{and} \quad \beta = \frac{1}{\gamma} \in \mathbb{R}_{> 0} \, \text{.} \] For $j \in \lbrace 1, \dotsc, m \rbrace$ and $k \in \left\lbrace 0, \dotsc, D_{j} -1 \right\rbrace$, write \[ P_{j, k}\left( X_{1}, \dotsc, X_{n} \right) = \sum_{\boldsymbol{\ell} \in \lbrace 0, \dotsc, \gamma \rbrace^{n}} b_{j, k, \boldsymbol{\ell}} \prod_{r = 1}^{n} X_{r}^{\ell_{r}} \in \mathbb{Q}\left[ X_{1}, \dotsc, X_{n} \right] \, \text{.} \] Consider the finite set \[ S^{(\infty)} = \left\lbrace p \text{ prime} : \max_{\substack{j \in \lbrace 1, \dotsc, m \rbrace\\ k \in \left\lbrace 0, \dotsc, D_{j} -1 \right\rbrace\\ \boldsymbol{\ell} \in \lbrace 0, \dotsc, \gamma \rbrace^{n}}} \left\lvert b_{j, k, \boldsymbol{\ell}} \right\rvert_{p} \neq 1 \right\rbrace \, \text{,} \] where $\lvert . \rvert_{p}$ denotes the $p$\nobreakdash-adic absolute value on $\mathbb{Q}$ for each prime number $p$. Now, suppose that $K$ is a valued field of characteristic $0$. Set \[ A_{K}^{(\infty)} = \left( \max_{j \in \lbrace 1, \dotsc, m \rbrace} \left( D_{j} \right)_{K} \right) \left( (\gamma +1)^{n} \right)_{K} \left( \max_{\substack{j \in \lbrace 1, \dotsc, m \rbrace\\ k \in \left\lbrace 0, \dotsc, D_{j} -1 \right\rbrace\\ \boldsymbol{\ell} \in \lbrace 0, \dotsc, \gamma \rbrace^{n}}} \left\lvert b_{j, k, \boldsymbol{\ell}} \right\rvert_{K} \right) \in \mathbb{R}_{> 0} \, \text{.} \] For every $\boldsymbol{t} \in K^{m}$, as $t_{j}^{D_{j}} = \sum\limits_{k = 0}^{D_{j} -1} P_{j, k}\left( \boldsymbol{F}(\boldsymbol{t}) \right) t_{j}^{k}$ for all $j \in \lbrace 1, \dotsc, m \rbrace$, we have \[ \forall j \in \lbrace 1, \dotsc, m \rbrace, \, \left\lvert t_{j} \right\rvert_{K}^{D_{j}} \leq A_{K}^{(\infty)} \max\left\lbrace 1, \left\lVert \boldsymbol{F}(\boldsymbol{t}) \right\rVert_{K^{n}} \right\rbrace^{\gamma} \max\left\lbrace 1, \left\lvert t_{j} \right\rvert_{K} \right\rbrace^{D_{j} -1} \] by the triangle inequality. Therefore, we have \[ \forall \boldsymbol{t} \in K^{m}, \, \lVert \boldsymbol{t} \rVert_{K^{m}} \geq 1 \Longrightarrow \lVert \boldsymbol{t} \rVert_{K^{m}} \leq A_{K}^{(\infty)} \max\left\lbrace 1, \left\lVert \boldsymbol{F}(\boldsymbol{t}) \right\rVert_{K^{n}} \right\rbrace^{\gamma} \, \text{.} \] As a result, setting $\alpha_{K}^{(\infty)} = \left( A_{K}^{(\infty)} \right)^{-\beta}$ and $R_{K}^{(\infty)} = \max\left\lbrace 1, A_{K}^{(\infty)} \right\rbrace$, we have \[ \forall \boldsymbol{t} \in K^{m}, \, \lVert \boldsymbol{t} \rVert_{K^{m}} > R_{K}^{(\infty)} \Longrightarrow \left\lVert \boldsymbol{F}(\boldsymbol{t}) \right\rVert_{K^{n}} \geq \alpha_{K}^{(\infty)} \lVert \boldsymbol{t} \rVert_{K^{m}}^{\beta} \, \text{.} \] Finally, note that both $\alpha_{K}^{(\infty)}$ and $R_{K}^{(\infty)}$ depend only on the restriction of $\lvert . \rvert_{K}$ to $\mathbb{Q}$. Moreover, we have $\alpha_{K}^{(\infty)} = 1$ and $R_{K}^{(\infty)} = 1$ if $K$ is non-Archimedean with residue characteristic not in $S^{(\infty)}$. This completes the proof of the lemma.
\end{proof}

\begin{remark}
Given a polynomial map $\boldsymbol{F} \colon \mathbb{C}^{m} \rightarrow \mathbb{C}^{n}$, with $m, n \geq 1$, the supremum of all $\beta \in \mathbb{R}$ for which there exist $\alpha, R \in \mathbb{R}_{> 0}$ such that $\left\lVert \boldsymbol{F}(\boldsymbol{t}) \right\rVert_{\mathbb{C}^{n}} \geq \alpha \lVert \boldsymbol{t} \rVert_{\mathbb{C}^{m}}^{\beta}$ for all $\boldsymbol{t} \in \mathbb{C}^{m}$ such that $\lVert \boldsymbol{t} \rVert_{\mathbb{C}^{m}} > R$ is known as the \L{}ojasiewicz exponent of $\boldsymbol{F}$ at infinity. We refer the reader to~\cite[Corollary~2]{CK1997} for an analytic proof that every proper complex polynomial map has positive \L{}ojasiewicz exponent at infinity.
\end{remark}

Combining Lemmas~\ref{lemma:minimum} and~\ref{lemma:lojasiewcz}, we easily obtain the general result below.

\begin{lemma}
\label{lemma:everywhere}
Suppose that $\boldsymbol{F} \colon \mathbb{A}^{m} \rightarrow \mathbb{A}^{n}$, with $m, n \geq 1$, is a morphism such that the induced holomorphic map $\boldsymbol{F} \colon \mathbb{C}^{m} \rightarrow \mathbb{C}^{n}$ is proper and has no zero in $\mathbb{C}^{m}$. Then, for every valued field $K$ of characteristic $0$, there exist some $\alpha_{K} \in \mathbb{R}_{> 0}$ depending only on the restriction of $\lvert . \rvert_{K}$ to $\mathbb{Q}$ and some $\beta \in \mathbb{R}_{> 0}$ not depending on $K$ such that \[ \forall \boldsymbol{t} \in K^{m}, \, \left\lVert \boldsymbol{F}(\boldsymbol{t}) \right\rVert_{K^{n}} \geq \alpha_{K} \max\left\lbrace 1, \lVert \boldsymbol{t} \rVert_{K^{m}} \right\rbrace^{\beta} \, \text{.} \] Moreover, we can take $\alpha_{K} = 1$ for every non-Archimedean field $K$ with characteristic $0$ and residue characteristic outside a finite set $S$ of prime numbers.
\end{lemma}

\begin{proof}
Suppose that $K$ is any valued field of characteristic $0$. Define \[ \alpha_{K} = \min\left\lbrace \alpha_{K}^{(0)} \max\left\lbrace 1, R_{K}^{(\infty)} \right\rbrace^{-\beta -\delta}, \alpha_{K}^{(\infty)} \right\rbrace \in \mathbb{R}_{> 0} \quad \text{and} \quad S = S^{(0)} \cup S^{(\infty)} \, \text{,} \] with $\alpha_{K}^{(0)} \in \mathbb{R}_{> 0}$, $\delta \in \mathbb{R}_{\geq 0}$, $S^{(0)}$ as in Lemma~\ref{lemma:minimum} and $\alpha_{K}^{(\infty)}, R_{K}^{(\infty)}, \beta \in \mathbb{R}_{> 0}$, $S^{(\infty)}$ as in Lemma~\ref{lemma:lojasiewcz}. Then $\alpha_{K}$ depends only on the restriction of $\lvert . \rvert_{K}$ to $\mathbb{Q}$, and we have $\alpha_{K} = 1$ if $K$ is non-Archimedean with residue characteristic outside $S$. Moreover, for every $\boldsymbol{t} \in K^{m}$ such that $\lVert \boldsymbol{t} \rVert_{K^{m}} \leq \max\left\lbrace 1, R_{K}^{(\infty)} \right\rbrace$, we have \[ \left\lVert \boldsymbol{F}(\boldsymbol{t}) \right\rVert_{K^{n}} \geq \alpha_{K}^{(0)} \max\left\lbrace 1, R_{K}^{(\infty)} \right\rbrace^{-\delta} \geq \alpha_{K} \max\left\lbrace 1, \lVert \boldsymbol{t} \rVert_{K^{m}} \right\rbrace^{\beta} \, \text{.} \] Also, for every $\boldsymbol{t} \in K^{m}$ such that $\lVert \boldsymbol{t} \rVert_{K^{m}} > \max\left\lbrace 1, R_{K}^{(\infty)} \right\rbrace$, we clearly have \[ \left\lVert \boldsymbol{F}(\boldsymbol{t}) \right\rVert_{K^{n}} \geq \alpha_{K} \max\left\lbrace 1, \lVert \boldsymbol{t} \rVert_{K^{m}} \right\rbrace^{\beta} \, \text{.} \] Thus, the lemma is proved.
\end{proof}

We shall now apply the lemma above to obtain results about the multipliers of polynomials over arbitrary algebraically closed valued fields of characteristic $0$. To do this, we work with polynomial maps in a particular form, which was introduced by Ingram in~\cite{I2012}. Although the two claims below concerning this normal form are already known, we include proofs for the reader's convenience.

Given any field $K$ of characteristic $0$ and $\boldsymbol{c} = \left( c_{1}, \dotsc, c_{d -1} \right) \in K^{d -1}$, define \[ f_{\boldsymbol{c}}(z) = \frac{1}{d} z^{d} +\sum_{j = 1}^{d -1} \frac{(-1)^{j} \tau_{j}(\boldsymbol{c})}{d -j} z^{d -j} \in \Poly_{d}(K) \, \text{,} \] where $\tau_{1}(\boldsymbol{c}), \dotsc, \tau_{d -1}(\boldsymbol{c})$ denote the elementary symmetric functions of $c_{1}, \dotsc, c_{d -1}$, so that \[ f_{\boldsymbol{c}}(0) = 0 \quad \text{and} \quad f_{\boldsymbol{c}}^{\prime}(z) = \prod_{j = 1}^{d -1} \left( z -c_{j} \right) \, \text{.} \] Consider the morphism $F \colon \mathbb{A}^{d -1} \rightarrow \mathcal{P}_{d}$ defined by $F(\boldsymbol{c}) = f_{\boldsymbol{c}}$.

\begin{claim}
\label{claim:moduli}
The holomorphic map $F \colon \mathbb{C}^{d -1} \rightarrow \mathcal{P}_{d}(\mathbb{C})$ is proper.
\end{claim}

\begin{proof}
Assume that $\left( \boldsymbol{c}_{n} \right)_{n \geq 0}$ is a sequence of elements of $\mathbb{C}^{d -1}$ such that $\left( \left[ f_{\boldsymbol{c}_{n}} \right] \right)_{n \geq 0}$ converges in $\mathcal{P}_{d}(\mathbb{C})$. We shall show that $\left( \boldsymbol{c}_{n} \right)_{n \geq 0}$ is bounded in $\mathbb{C}^{d -1}$. There exists a sequence $\left( \phi_{n} \right)_{n \geq 0}$ of elements of $\Aff(\mathbb{C})$ such that the sequence $\left( f_{n} \right)_{n \geq 0}$ given by $f_{n} = \phi_{n} \centerdot f_{\boldsymbol{c}_{n}}$ converges to some $g \in \Poly_{d}(\mathbb{C})$. The multiset of all the fixed points for $f_{n}$ tends to the multiset of all the fixed points for $g$ in $\mathbb{C}^{d}/\mathfrak{S}_{d}$ as $n \rightarrow +\infty$. In particular, $\left( \phi_{n}(0) \right)_{n \geq 0}$ is bounded in $\mathbb{C}$. In addition, the multiset of all the critical points for $f_{n}$ tends to the multiset of all the critical points for $g$ in $\mathbb{C}^{d -1}/\mathfrak{S}_{d -1}$ as $n \rightarrow +\infty$. As a result, writing $\boldsymbol{c}_{n} = \left( c_{n}^{(1)}, \dotsc, c_{n}^{(d -1)} \right)$ for each $n \geq 0$, the sequence $\left( \phi_{n}\left( c_{n}^{(j)} \right) \right)_{n \geq 0}$ is bounded in $\mathbb{C}$ for each $j \in \lbrace 1, \dotsc, d -1 \rbrace$. Furthermore, writing $\phi_{n}(z) = \alpha_{n} z +\beta_{n}$ for all $n \geq 0$, the polynomial $f_{n}$ has leading coefficient $\frac{\alpha_{n}^{1 -d}}{d}$ for all $n \geq 0$, which yields $\lim\limits_{n \rightarrow +\infty} \left\lvert \alpha_{n} \right\rvert = \left\lvert d \cdot a_{d} \right\rvert^{\frac{-1}{d -1}} \in \mathbb{R}_{> 0}$, where $a_{d} \in \mathbb{C}^{*}$ denotes the leading coefficient of $g$. Therefore, as $c_{n}^{(j)} = \frac{\phi_{n}\left( c_{n}^{(j)} \right) -\phi_{n}(0)}{\alpha_{n}}$ for all $j \in \lbrace 1, \dotsc, d -1 \rbrace$ and all $n \geq 0$, the sequence $\left( \boldsymbol{c}_{n} \right)_{n \geq 0}$ is bounded in $\mathbb{C}^{d -1}$. This completes the proof of the claim.
\end{proof}

Moreover, using the triangle inequality, we easily obtain an upper bound on the Green functions of the polynomials $f_{\boldsymbol{c}}$, with $K$ an algebraically closed valued field of characteristic $0$ and $\boldsymbol{c} \in K^{d -1}$.

\begin{claim}
\label{claim:estimates1}
For each algebraically closed valued field $K$ of characteristic $0$, we have \[ g_{f_{\boldsymbol{c}}}(z) \leq \log^{+}\left( \max\left\lbrace \lVert \boldsymbol{c} \rVert_{K^{d -1}}, \lvert z \rvert_{K} \right\rbrace \right) +\Delta_{K} \] for all $\boldsymbol{c} \in K^{d -1}$ and all $z \in K$, and in particular $M_{f_{\boldsymbol{c}}} \leq \log^{+}\lVert \boldsymbol{c} \rVert_{K^{d -1}} +\Delta_{K}$ for all $\boldsymbol{c} \in K^{d -1}$, where \[ \Delta_{K} = \frac{1}{d -1} \log(d)_{K} +\frac{1}{d -1} \log\left( \max_{j \in \lbrace 0, \dotsc, d -1 \rbrace} \frac{1}{\lvert d -j \rvert_{K}} \left( \binom{d -1}{j} \right)_{K} \right) \in \mathbb{R}_{\geq 0} \, \text{.} \]
\end{claim}

\begin{proof}
For every $\boldsymbol{c} \in K^{d -1}$ and every $z \in K$, we have \[ \left\lvert f_{\boldsymbol{c}}(z) \right\rvert_{K} \leq (d)_{K} \left( \max_{j \in \lbrace 0, \dotsc, d -1 \rbrace} \left\lvert \frac{(-1)^{j} \tau_{j}(\boldsymbol{c})}{d -j} z^{d -j} \right\rvert_{K} \right) \leq \delta_{K} \max\left\lbrace \lVert \boldsymbol{c} \rVert_{K^{d -1}}, \lvert z \rvert_{K} \right\rbrace^{d} \] by the triangle inequality, where $\tau_{0}(\boldsymbol{c}) = 1$ by convention and \[ \delta_{K} = (d)_{K} \left( \max_{j \in \lbrace 0, \dotsc, d -1 \rbrace} \frac{1}{\lvert d -j \rvert_{K}} \left( \binom{d -1}{j} \right)_{K} \right) \in \mathbb{R}_{\geq 1} \, \text{.} \] It follows by induction that \[ \left\lvert f_{\boldsymbol{c}}^{\circ n}(z) \right\rvert_{K} \leq \delta_{K}^{\frac{d^{n} -1}{d -1}} \max\left\lbrace 1, \lVert \boldsymbol{c} \rVert_{K^{d -1}}, \lvert z \rvert_{K} \right\rbrace^{d^{n}} \] for all $\boldsymbol{c} \in K^{d -1}$, all $z \in K$ and all $n \geq 0$. Therefore, for every $\boldsymbol{c} \in K^{d -1}$ and every $z \in K$, we have \[ \frac{1}{d^{n}} \log^{+}\left\lvert f_{\boldsymbol{c}}^{\circ n}(z) \right\rvert_{K} \leq \frac{d^{n} -1}{d^{n} (d -1)} \log\left( \delta_{K} \right) +\log^{+}\left( \max\left\lbrace \lVert \boldsymbol{c} \rVert_{K^{d -1}}, \lvert z \rvert_{K} \right\rbrace \right) \] for all $n \geq 0$, which yields the desired result by letting $n \rightarrow +\infty$. Thus, the claim is proved.
\end{proof}

Finally, we obtain the following result regarding polynomial maps over arbitrary algebraically closed valued fields of characteristic $0$:

\begin{proposition}
\label{proposition:corollary}
Assume that $K$ is an algebraically closed valued field of characteristic $0$. Then there exist some $A \in \mathbb{R}_{> 0}$ not depending on $K$ and some $B_{K} \in \mathbb{R}$ depending only on the restriction of $\lvert . \rvert_{K}$ to $\mathbb{Q}$ such that \[ \max\left\lbrace M_{f}^{(1)}, M_{f}^{(2)} \right\rbrace \geq A \cdot M_{f} +B_{K} \] for all $f \in \Poly_{d}(K)$. Furthermore, we can take $B_{K} = 0$ if $K$ is non-Archimedean with residue characteristic outside some finite set $S$ of prime numbers.
\end{proposition}

\begin{proof}
Consider the morphism $\boldsymbol{G} \colon \mathbb{A}^{d -1} \rightarrow \mathbb{A}^{d} \times \mathbb{A}^{\frac{d (d -1)}{2}}$ defined by \[ \boldsymbol{G}(\boldsymbol{c}) = \Mult_{d}^{(2)}\left( \left[ f_{\boldsymbol{c}} \right] \right) = \left( \left( \sigma_{d, j}^{(1)}\left( \left[ f_{\boldsymbol{c}} \right] \right) \right)_{1 \leq j \leq d}, \left( \sigma_{d, j}^{(2)}\left( \left[ f_{\boldsymbol{c}} \right] \right) \right)_{1 \leq j \leq \frac{d (d -1)}{2}} \right) \, \text{.} \] Then the holomorphic map $\boldsymbol{G} \colon \mathbb{C}^{d -1} \rightarrow \mathbb{C}^{d} \times \mathbb{C}^{\frac{d (d -1)}{2}}$ has no zero in $\mathbb{C}^{d -1}$ because $d +\sum\limits_{j = 1}^{d} (-1)^{j} (d -j) \sigma_{d, j}^{(1)} = 0$ by the holomorphic fixed-point formula. In addition, the map $\boldsymbol{G} \colon \mathbb{C}^{d -1} \rightarrow \mathbb{C}^{d} \times \mathbb{C}^{\frac{d (d -1)}{2}}$ is proper by Corollary~\ref{corollary:proper} and Claim~\ref{claim:moduli}. Thus, by Lemma~\ref{lemma:everywhere}, there exist some $A^{\prime} \in \mathbb{R}_{> 0}$ not depending on $K$ and some $B_{K}^{\prime} \in \mathbb{R}$ depending only on the restriction of $\lvert . \rvert_{K}$ to $\mathbb{Q}$ such that \[ \forall \boldsymbol{c} \in K^{d -1}, \, \log\left\lVert \boldsymbol{G}(\boldsymbol{c}) \right\rVert_{K^{d} \times K^{\frac{d (d -1)}{2}}} \geq A^{\prime} \cdot \log^{+}\lVert \boldsymbol{c} \rVert_{K^{d -1}} +B_{K}^{\prime} \, \text{.} \] Moreover, we can take $B_{K}^{\prime} = 0$ if $K$ is non-Archimedean with residue characteristic outside some finite set $S^{\prime}$ of prime numbers. By Claim~\ref{claim:estimates1}, it follows that \[ \log\left\lVert \boldsymbol{G}(\boldsymbol{c}) \right\rVert_{K^{d} \times K^{\frac{d (d -1)}{2}}} \geq A^{\prime} \cdot M_{f_{\boldsymbol{c}}} +B_{K}^{\prime} -A^{\prime} \cdot \Delta_{K} \] for all $\boldsymbol{c} \in K^{d -1}$. Now, for every $\boldsymbol{c} \in K^{d -1}$, we have \[ \left\lvert \sigma_{d, j}^{(p)}\left( \left[ f_{\boldsymbol{c}} \right] \right) \right\rvert_{K} \leq \left( \binom{N_{d}^{(p)}}{j} \right)_{K} \left( \max_{\lambda \in \Lambda_{f_{\boldsymbol{c}}}^{(p)}} \lvert \lambda \rvert \right)^{j} = \left( \binom{N_{d}^{(p)}}{j} \right)_{K} \exp\left( j p \cdot M_{f_{\boldsymbol{c}}}^{(p)} \right) \] for all $p \geq 1$ and all $j \in \left\lbrace 1, \dotsc, N_{d}^{(p)} \right\rbrace$ by the triangle inequality, and hence \[ \log\left\lVert \boldsymbol{G}(\boldsymbol{c}) \right\rVert_{K^{d} \times K^{\frac{d (d -1)}{2}}} \leq \max_{\substack{p \in \lbrace 1, 2 \rbrace\\ j \in \left\lbrace 1, \dotsc, N_{d}^{(p)} \right\rbrace}} \left( \log\left( \binom{N_{d}^{(p)}}{j} \right)_{K} +j p \cdot M_{f_{\boldsymbol{c}}}^{(p)} \right) \, \text{.} \] Therefore, setting \[ A = \frac{A^{\prime}}{d (d -1)} \quad \text{and} \quad B_{K} = \min_{\substack{p \in \lbrace 1, 2 \rbrace\\ j \in \left\lbrace 1, \dotsc, N_{d}^{(p)} \right\rbrace}} \left( \frac{B_{K}^{\prime} -A^{\prime} \cdot \Delta_{K} -\log\left( \binom{N_{d}^{(p)}}{j} \right)_{K}}{j p} \right) \, \text{,} \] we have \[ \forall \boldsymbol{c} \in K^{d -1}, \, \max\left\lbrace M_{f_{\boldsymbol{c}}}^{(1)}, M_{f_{\boldsymbol{c}}}^{(2)} \right\rbrace \geq A \cdot M_{f_{\boldsymbol{c}}} +B_{K} \, \text{.} \] Now, note that $A \in \mathbb{R}_{> 0}$ does not depend on $K$ and $B_{K} \in \mathbb{R}$ depends only on the restriction of $\lvert . \rvert_{K}$ to $\mathbb{Q}$. Furthermore, setting \[ S = S^{\prime} \cup \left\lbrace q \text{ prime} : 2 \leq q \leq \max\left\lbrace d, \frac{d (d -1)}{2} \right\rbrace \right\rbrace \, \text{,} \] we have $B_{K} = 0$ if $K$ is non-Archimedean with residue characteristic not in $S$.

To conclude, assume that $f \in \Poly_{d}(K)$ has leading coefficient $a_{d} \in K^{*}$. Choose any $(d -1)$th root $\alpha \in K^{*}$ of $d \cdot a_{d}$ and any fixed point $w \in K$ for $f$, and consider $\phi(z) = \alpha (z -w) \in \Aff(K)$. Then $\phi \centerdot f \in \Poly_{d}(K)$ has leading coefficient $\frac{1}{d}$ and it satisfies $\phi \centerdot f(0) = 0$, and hence $\phi \centerdot f = f_{\boldsymbol{c}}$, where $c_{1}, \dotsc, c_{d -1} \in K$ are the critical points for $\phi \centerdot f$ and $\boldsymbol{c} = \left( c_{1}, \dotsc, c_{d -1} \right)$. Since $M_{f} = M_{f_{\boldsymbol{c}}}$ and $M_{f}^{(p)} = M_{f_{\boldsymbol{c}}}^{(p)}$ for each integer $p \geq 1$, we have $\max\left\lbrace M_{f}^{(1)}, M_{f}^{(2)} \right\rbrace \geq A \cdot M_{f} +B_{K}$ by the discussion above. Thus, the proposition is proved.
\end{proof}

\begin{remark}
It follows immediately from Theorem~\ref{theorem:degenB} that we can take $S$ to be the set of all primes less than or equal to $d$ in the statement of Proposition~\ref{proposition:corollary}.
\end{remark}

Note that Corollary~\ref{corollary:degenLocal} is simply a weaker version of Proposition~\ref{proposition:corollary}. We shall now show that Corollary~\ref{corollary:degenGlobal} also follows directly from Proposition~\ref{proposition:corollary}. To do this, let us first briefly recall various notions of height. We refer to~\cite[Chapter~3]{L1983}, \cite[Chapter~3]{S2007} and~\cite[Chapter~5]{S2012} for further details.

We denote here by $\mathbb{P}$ the set of all prime numbers. For $p \in \mathbb{P}$, denote by $\overline{\mathbb{Q}}_{p}$ the algebraic closure of the field $\mathbb{Q}_{p}$ of $p$\nobreakdash-adic numbers and by $\lvert . \rvert_{p}$ the natural absolute value on $\overline{\mathbb{Q}}_{p}$. Also set $\overline{\mathbb{Q}}_{\infty} = \mathbb{C}$ and denote by $\lvert . \rvert_{\infty}$ the usual absolute value on $\overline{\mathbb{Q}}_{\infty}$. The \emph{standard height} $h \colon \overline{\mathbb{Q}} \rightarrow \mathbb{R}_{\geq 0}$ is given by \[ h(t) = \frac{1}{[\mathbb{K} \colon \mathbb{Q}]} \sum_{p \in \mathbb{P} \cup \lbrace \infty \rbrace} \sum_{\sigma \mathpunct{:} \mathbb{K} \hookrightarrow \overline{\mathbb{Q}}_{p}} \log^{+}\left\lvert \sigma(t) \right\rvert_{p} \, \text{,} \quad \text{with} \quad t \in \mathbb{K} \quad \text{and} \quad [\mathbb{K} \colon \mathbb{Q}] < +\infty \, \text{.} \]

Suppose that $f \in \Poly_{d}\left( \overline{\mathbb{Q}} \right)$. The \emph{canonical height} $\widehat{h}_{f} \colon \overline{\mathbb{Q}} \rightarrow \mathbb{R}_{\geq 0}$ relative to $f$ is defined by \[ \widehat{h}_{f}(z) = \lim_{n \rightarrow +\infty} \frac{1}{d^{n}} h\left( f^{\circ n}(z) \right) \, \text{.} \] For every number field $\mathbb{K}$ such that $f \in \Poly_{d}(\mathbb{K})$ and every $z \in \mathbb{K}$, we have \[ \widehat{h}_{f}(z) = \frac{1}{[\mathbb{K} \colon \mathbb{Q}]} \sum_{p \in \mathbb{P} \cup \lbrace \infty \rbrace} \sum_{\sigma \mathpunct{:} \mathbb{K} \hookrightarrow \overline{\mathbb{Q}}_{p}} g_{\sigma(f)}\left( \sigma(z) \right) \, \text{.} \] Now, denote by $\Gamma_{f} \subseteq \overline{\mathbb{Q}}$ the set of critical points for $f$ and, for $\gamma \in \Gamma_{f}$, define $\rho_{\gamma}$ to be its multiplicity as a critical point for $f$. The \emph{critical height} $H_{f}$ of $f$ is given by \[ H_{f} = \sum_{\gamma \in \Gamma_{f}} \rho_{\gamma} \cdot \widehat{h}_{f}(\gamma) \, \text{.} \]

\begin{proof}[Proof of Corollary~\ref{corollary:degenGlobal}]
By Proposition~\ref{proposition:corollary}, there exist $A^{\prime} \in \mathbb{R}_{> 0}$ and $B_{p}^{\prime} \in \mathbb{R}$, for $p \in \mathbb{P} \cup \lbrace \infty \rbrace$, such that \[ \forall p \in \mathbb{P} \cup \lbrace \infty \rbrace, \, \forall g \in \Poly_{d}\left( \overline{\mathbb{Q}}_{p} \right), \, \max\left\lbrace M_{g}^{(1)}, M_{g}^{(2)} \right\rbrace \geq A^{\prime} \cdot M_{g} +B_{p}^{\prime} \, \text{.} \] Moreover, we can take $B_{p}^{\prime} = 0$ for all but finitely many $p \in \mathbb{P} \cup \lbrace \infty \rbrace$. Define \[ A = \frac{A^{\prime}}{d -1} \in \mathbb{R}_{> 0} \quad \text{and} \quad B = \sum_{p \in \mathbb{P} \cup \lbrace \infty \rbrace} B_{p}^{\prime} \in \mathbb{R} \, \text{.} \] Now, suppose that $f \in \Poly_{d}\left( \overline{\mathbb{Q}} \right)$. Fix a number field $\mathbb{K}$ containing the coefficients of $f$, its critical points and its multipliers at all its cycles with period $1$ or $2$. Then \[ \begin{split} \max\left\lbrace H_{f}^{(1)}, H_{f}^{(2)} \right\rbrace & = \frac{1}{[\mathbb{K} \colon \mathbb{Q}]} \sum_{p \in \mathbb{P} \cup \lbrace \infty \rbrace} \sum_{\sigma \mathpunct{:} \mathbb{K} \hookrightarrow \overline{\mathbb{Q}}_{p}} \max_{\substack{q \in \lbrace 1, 2 \rbrace\\ \lambda \in \Lambda_{f}^{(q)}}} \left( \frac{1}{q} \log^{+}\left\lvert \sigma(\lambda) \right\rvert_{p} \right)\\ & \geq \frac{1}{[\mathbb{K} \colon \mathbb{Q}]} \sum_{p \in \mathbb{P} \cup \lbrace \infty \rbrace} \sum_{\sigma \mathpunct{:} \mathbb{K} \hookrightarrow \overline{\mathbb{Q}}_{p}} \max\left\lbrace M_{\sigma(f)}^{(1)}, M_{\sigma(f)}^{(2)} \right\rbrace\\ & \geq \frac{1}{[\mathbb{K} \colon \mathbb{Q}]} \sum_{p \in \mathbb{P} \cup \lbrace \infty \rbrace} \sum_{\sigma \mathpunct{:} \mathbb{K} \hookrightarrow \overline{\mathbb{Q}}_{p}} \left( A^{\prime} \cdot M_{\sigma(f)} +B_{p}^{\prime} \right)\\ & \geq \frac{A}{[\mathbb{K} \colon \mathbb{Q}]} \sum_{p \in \mathbb{P} \cup \lbrace \infty \rbrace} \sum_{\sigma \mathpunct{:} \mathbb{K} \hookrightarrow \overline{\mathbb{Q}}_{p}} \left( \sum_{\gamma \in \Gamma_{f}} \rho_{\gamma} \cdot g_{\sigma(f)}\left( \sigma(\gamma) \right) \right) +B\\ & = A \cdot H_{f} +B \, \text{.} \end{split} \] Thus, the corollary is proved.
\end{proof}

\section{The non-Archimedean case}
\label{section:degenNonArch}

We shall adapt here the discussion of Section~\ref{section:degenArch} in order to prove Theorem~\ref{theorem:degenB} in the non-Archimedean case.

Throughout this section, we fix an integer $d \geq 2$ and an algebraically closed field $K$ of characteristic $0$ equipped with a non-Archimedean absolute value $\lvert . \rvert$. We also assume that the residue characteristic of $K$ either equals $0$ or is greater than $d$, so that $\lvert j \rvert = 1$ for all $j \in \lbrace 1, \dotsc, d \rbrace$. In addition, we assume that $\lvert . \rvert$ is not the trivial absolute value, as Theorem~\ref{theorem:degenB} would be immediate otherwise. Note that we do not assume $K$ to be complete here, although there is no gain in generality in not doing so since Theorem~\ref{theorem:degenB} clearly holds for $K$ if it holds for its completion $\widehat{K}$.

\subsection{A few preliminaries on non-Archimedean analysis}

First, let us recall some basic facts about disks and polynomial maps in the non-Archimedean setting. We omit proofs and refer to~\cite[Chapters~2 and~3]{B2019} for further information.

In this section, we shall only work with finite unions of disks. Given $w \in K$ and $r \in \mathbb{R}_{> 0}$, we denote by $D(w, r)$ and $\overline{D(w, r)}$ the open and closed disks of center $w$ and radius $r$, respectively, which are given by \[ D(w, r) = \left\lbrace z \in K : \lvert z -w \rvert < r \right\rbrace \quad \text{and} \quad \overline{D(w, r)} = \left\lbrace z \in K : \lvert z -w \rvert \leq r \right\rbrace \, \text{.} \] Note that a disk has a unique radius. In contrast, each point of a disk is a center. Although all disks are both open and closed topologically, we say here that a disk is \emph{open} if it is of the form $D(w, r)$, with $w \in K$ and $r \in \mathbb{R}_{> 0}$, and we say that it is \emph{closed} if it is of the form $\overline{D(w, r)}$, with $w \in K$ and $r \in \mathbb{R}_{> 0}$. Now, note that a disk is both open and closed if and only if its radius does not lie in $\left\lvert K^{*} \right\rvert$.

Suppose that $U$ is a finite union of disks in $K$. Then $U$ can be written uniquely as the union of finitely many pairwise disjoint disks $U_{1}, \dotsc, U_{N}$ in $K$, with $N \geq 0$. In addition, every disk contained in $U$ is contained in $U_{j}$ for some $j \in \lbrace 1, \dotsc, N \rbrace$. These disks $U_{1}, \dotsc, U_{N}$ are called the \emph{disk components} of $U$. The disk components of every finite union of open disks are all open. Similarly, the disk components of every finite union of closed disks are all closed.

Now, every nonconstant polynomial in $K[z]$ maps open disks to open disks and closed disks to closed disks. Given disks $U, V$ in $K$, we say that a polynomial map $f \colon U \rightarrow V$ has \emph{degree} $e \geq 1$ if every element of $V$ has exactly $e$ preimages under $f$ in $U$, counting multiplicities.

Given a polynomial $f \in K[z]$ of degree $D \geq 1$ with leading coefficient $a_{D} \in K^{*}$ and $w \in K$, it is not hard to show that, for all $r \in \mathbb{R}_{> 0}$ sufficiently large, we have \[ f\left( D(w, r) \right) = D\left( f(w), \left\lvert a_{D} \right\rvert r^{D} \right) \quad \text{and} \quad f\left( \overline{D(w, r)} \right) = \overline{D\left( f(w), \left\lvert a_{D} \right\rvert r^{D} \right)} \] and the maps $f \colon D(w, r) \rightarrow D\left( f(w), \left\lvert a_{D} \right\rvert r^{D} \right)$ and $f \colon \overline{D(w, r)} \rightarrow \overline{D\left( f(w), \left\lvert a_{D} \right\rvert r^{D} \right)}$ have degree $D$. More generally, the result below describes precisely the images of disks under nonconstant polynomial maps in the non-Archimedean setting.

\begin{lemma}
\label{lemma:image}
Suppose that $f \in K[z]$ has degree $D \geq 1$, $w \in K$ and $r \in \mathbb{R}_{> 0}$. Set \[ s = \max_{j \in \lbrace 1, \dotsc, D \rbrace} \left\lvert \frac{f^{(j)}(w)}{j!} \right\rvert r^{j} \in \mathbb{R}_{> 0} \, \text{,} \] and define $e_{\min}$ and $e_{\max}$ to be the smallest and largest integers $j \in \lbrace 1, \dotsc, D \rbrace$ such that $s = \left\lvert \frac{f^{(j)}(w)}{j!} \right\rvert r^{j}$, respectively. Then we have \[ f\left( D(w, r) \right) = D\left( f(w), s \right) \quad \text{and} \quad f\left( \overline{D(w, r)} \right) = \overline{D\left( f(w), s \right)} \, \text{.} \] Moreover, the maps $f \colon D(w, r) \rightarrow D\left( f(w), s \right)$ and $f \colon \overline{D(w, r)} \rightarrow \overline{D\left( f(w), s \right)}$ have degrees $e_{\min}$ and $e_{\max}$, respectively.
\end{lemma}

We can also describe the preimages of disks under polynomial maps in the non-Archimedean setting.

\begin{lemma}
\label{lemma:preimage}
Suppose that $U, V$ are disks in $K$, $f \colon U \rightarrow V$ is a polynomial map of degree $e \geq 1$ and $W$ is a disk contained in $V$. Then $f^{-1}(W)$ is a nonempty finite union of disks, and its disk components $U_{1}, \dotsc, U_{N}$, with $N \geq 1$, are all open if $W$ is open and all closed if $W$ is closed. Moreover, the map $f \colon U_{j} \rightarrow W$ has a degree $e_{j} \geq 1$ for each $j \in \lbrace 1, \dotsc, N \rbrace$, and we have $e = \sum\limits_{j = 1}^{N} e_{j}$.
\end{lemma}

\begin{remark}
Note that, if $f \in K[z]$ is a polynomial of degree $e \geq 1$ and $W$ is a disk in $K$, then the conclusion of Lemma~\ref{lemma:preimage} still holds. Indeed, in this case, there exist disks $U, V$ in $K$ such that $U = f^{-1}(V)$ and $W \subseteq V$, and $f \colon U \rightarrow V$ has degree $e$.
\end{remark}

Finally, we have a non-Archimedean analogue of the Riemann--Hurwitz formula for disks, which relates the degree of a map to the number of its critical points.

\begin{lemma}
\label{lemma:rhFormula}
Suppose that $U, V$ are disks in $K$ and $f \colon U \rightarrow V$ is a polynomial map of degree $e \geq 1$. Also assume that $e$ is less than the residue characteristic of $K$ if the latter is positive. Then $e = C +1$, where $C \geq 0$ is the number of critical points for $f$ in $U$, counting multiplicities.
\end{lemma}

\begin{remark}
The assumption on the degree of the map in Lemma~\ref{lemma:rhFormula} is necessary, as the following shows: Suppose that $p \geq 2$ is a prime number. Then the algebraic closure $\overline{\mathbb{Q}}_{p}$ of the field $\mathbb{Q}_{p}$ of $p$\nobreakdash-adic numbers is naturally a non-Archimedean field with residue characteristic $p$. The polynomial $f(z) = z^{p} -p z \in \overline{\mathbb{Q}}_{p}[z]$ maps $D(0, 1)$ onto itself with degree $p$, while the critical points for $f$ all lie outside $D(0, 1)$. The polynomial $g(z) = z^{p +1} -z^{p} \in \overline{\mathbb{Q}}_{p}[z]$ maps $D(0, 1)$ onto itself with degree $p$, while the critical points for $g$ all lie in $D(0, 1)$. In particular, the conclusion of Lemma~\ref{lemma:rhFormula} does not hold for the maps $f \colon D(0, 1) \rightarrow D(0, 1)$ and $g \colon D(0, 1) \rightarrow D(0, 1)$.
\end{remark}

\subsection{The Green function of a polynomial map in the non-Archimedean setting}

Now, let us adapt our discussion of Green functions and maximal escape rates for complex polynomial maps to the non-Archimedean setting.

Suppose that $f \in \Poly_{d}(K)$. Recall that the \emph{Green function} $g_{f} \colon K \rightarrow \mathbb{R}_{\geq 0}$ of $f$ is given by \[ g_{f}(z) = \lim_{n \rightarrow +\infty} \frac{1}{d^{n}} \log^{+}\left\lvert f^{\circ n}(z) \right\rvert \, \text{.} \] This map $g_{f}$ is well defined and satisfies $g_{f} \circ f = d \cdot g_{f}$. Moreover, for each $z \in K$, we have $g_{f}(z) = 0$ if and only if $\sup\limits_{n \geq 0} \left\lvert f^{\circ n}(z) \right\rvert < +\infty$. Also recall that the \emph{maximal escape rate} $M_{f}$ of $f$ is defined by \[ M_{f} = \max\left\lbrace g_{f}(c) : c \in K, \, f^{\prime}(c) = 0 \right\rbrace \, \text{.} \] For every $\phi \in \Aff(K)$, we have $g_{\phi \centerdot f} = g_{f} \circ \phi^{-1}$, and hence $M_{\phi \centerdot f} = M_{f}$.

Thus, using conjugation, we may first restrict our attention to polynomials in a particular form. For $\boldsymbol{c} = \left( c_{1}, \dotsc, c_{d -1} \right) \in K^{d -1}$, define \[ f_{\boldsymbol{c}}(z) = \frac{1}{d} z^{d} +\sum_{j = 1}^{d -1} \frac{(-1)^{j} \tau_{j}(\boldsymbol{c})}{d -j} z^{d -j} \in \Poly_{d}(K) \, \text{,} \] where $\tau_{1}(\boldsymbol{c}), \dotsc, \tau_{d -1}(\boldsymbol{c})$ denote the elementary symmetric functions of $c_{1}, \dotsc, c_{d -1}$, so that \[ f_{\boldsymbol{c}}(0) = 0 \quad \text{and} \quad f_{\boldsymbol{c}}^{\prime}(z) = \prod_{j = 1}^{d -1} \left( z -c_{j} \right) \, \text{.} \] These polynomials have already been studied by Ingram in~\cite{I2012}. Nevertheless, for completeness and to specify the values of certain constants in the present case, we include details.

For $\boldsymbol{c} = \left( c_{1}, \dotsc, c_{d -1} \right) \in K^{d -1}$, define \[ \lVert \boldsymbol{c} \rVert = \max_{j \in \lbrace 1, \dotsc, d -1 \rbrace} \left\lvert c_{j} \right\rvert \in \mathbb{R}_{\geq 0} \, \text{.} \] Using the ultrametric triangle inequality and our assumption on the residue characteristic of $K$, we obtain the following:

\begin{claim}
\label{claim:estimates2}
We have $g_{f_{\boldsymbol{c}}}(z) \leq \log^{+}\left( \max\left\lbrace \lVert \boldsymbol{c} \rVert, \lvert z \rvert \right\rbrace \right)$ for all $\boldsymbol{c} \in K^{d -1}$ and all $z \in K$. Moreover, we have $g_{f_{\boldsymbol{c}}}(z) = \log^{+}\lvert z \rvert$ for all $\boldsymbol{c} \in K^{d -1}$ and all $z \in K \setminus \overline{D\left( 0, \lVert \boldsymbol{c} \rVert \right)}$.
\end{claim}

\begin{proof}
Note that the first assertion is simply a particular case of Claim~\ref{claim:estimates1}. Thus, let us prove the second one. Suppose that $\boldsymbol{c} \in K^{d -1}$. For every $z \in K \setminus \overline{D\left( 0, \lVert \boldsymbol{c} \rVert \right)}$, we have \[ \max_{j \in \lbrace 1, \dotsc, d -1 \rbrace} \left\lvert \frac{(-1)^{j} \tau_{j}(\boldsymbol{c})}{d -j} z^{d -j} \right\rvert \leq \max_{j \in \lbrace 1, \dotsc, d -1 \rbrace} \lVert \boldsymbol{c} \rVert^{j} \lvert z \rvert^{d -j} < \lvert z \rvert^{d} = \left\lvert \frac{1}{d} z^{d} \right\rvert \, \text{,} \] which yields $\left\lvert f_{\boldsymbol{c}}(z) \right\rvert = \lvert z \rvert^{d}$ by the ultrametric triangle inequality. By induction, we deduce that $\left\lvert f_{\boldsymbol{c}}^{\circ n}(z) \right\rvert = \lvert z \rvert^{d^{n}}$ for all $z \in K \setminus \overline{D\left( 0, \max\left\lbrace 1, \lVert \boldsymbol{c} \rVert \right\rbrace \right)}$ and all $n \geq 0$. As a result, for every $z \in K \setminus \overline{D\left( 0, \max\left\lbrace 1, \lVert \boldsymbol{c} \rVert \right\rbrace \right)}$, we have $\frac{1}{d^{n}} \log^{+}\left\lvert f_{\boldsymbol{c}}^{\circ n}(z) \right\rvert = \log^{+}\lvert z \rvert$ for all $n \geq 0$, which yields $g_{f_{\boldsymbol{c}}}(z) = \log^{+}\lvert z \rvert$ by letting $n \rightarrow +\infty$. Finally, we also have $g_{f_{\boldsymbol{c}}}(z) = \log^{+}\lvert z \rvert$ for each $z \in \overline{D(0, 1)} \setminus \overline{D\left( 0, \lVert \boldsymbol{c} \rVert \right)}$ by the first assertion of the claim. Thus, the claim is proved.
\end{proof}

Now, let us determine the maximal escape rates of these polynomials. To do so, we shall use Macaulay resultants to have an effective version of the Nullstellensatz for $r$ homogeneous polynomials in $r$ variables over a commutative ring, with $r \geq 1$. Thus, let us start by recalling a few necessary facts about resultants.

Suppose that $R$ is a commutative ring and $P_{1}, \dotsc, P_{r} \in R\left[ T_{1}, \dotsc, T_{r} \right]$ are homogeneous polynomials of degrees $e_{1}, \dotsc, e_{r} \geq 1$, respectively, with $r \geq 1$. Then there exists an element $\res\left( P_{1}, \dotsc, P_{r} \right) \in R$, called the \emph{Macaulay resultant} of $P_{1}, \dotsc, P_{r}$, that satisfies the following:
\begin{itemize}
\item There are an integer $E \geq \max\limits_{k \in \lbrace 1, \dotsc, r \rbrace} e_{k}$ and some homogeneous polynomials $Q_{j, k} \in R\left[ T_{1}, \dotsc, T_{r} \right]$ of degrees $E -e_{k}$, with $j, k \in \lbrace 1, \dotsc, r \rbrace$, such that \[ \res\left( P_{1}, \dotsc, P_{r} \right) T_{j}^{E} = \sum_{k = 1}^{r} P_{k}\left( T_{1}, \dotsc, T_{r} \right) Q_{j, k}\left( T_{1}, \dotsc, T_{r} \right) \] for each $j \in \lbrace 1, \dotsc, r \rbrace$.
\item For any algebraically closed field $\Omega$ and any ring homomorphism $\varphi \colon R \rightarrow \Omega$, we have $\varphi\left( \res\left( P_{1}, \dotsc, P_{r} \right) \right) = 0$ if and only if the homogeneous polynomials $\varphi\left( P_{1} \right), \dotsc, \varphi\left( P_{r} \right) \in \Omega\left[ T_{1}, \dotsc, T_{r} \right]$ have a common zero in $\Omega^{r} \setminus \lbrace 0 \rbrace$, where $\varphi \colon R\left[ T_{1}, \dotsc, T_{r} \right] \rightarrow \Omega\left[ T_{1}, \dotsc, T_{r} \right]$ denotes the unique ring homomorphism that extends $\varphi \colon R \rightarrow \Omega$ and satisfies $\varphi\left( T_{j} \right) = T_{j}$ for all $j \in \lbrace 1, \dotsc, r \rbrace$.
\end{itemize}
We refer to~\cite[Chapter~IX, Section~3]{La2002} for further details about resultants.

We now return to the study of the polynomials $f_{\boldsymbol{c}} \in \Poly_{d}(K)$, with $\boldsymbol{c} \in K^{d -1}$. Consider the subring $A = \mathbb{Z}\left[ \frac{1}{2}, \dotsc, \frac{1}{d} \right]$ of $K$. For each $a \in A$, we have $\lvert a \rvert \leq 1$, with equality holding if and only if $a$ is not divisible by the residue characteristic of $K$. For $j \in \lbrace 1, \dotsc, d -1 \rbrace$, also consider the polynomial $F_{j} \in A\left[ T_{1}, \dotsc, T_{d -1} \right]$ given by $F_{j}(\boldsymbol{c}) = f_{\boldsymbol{c}}\left( c_{j} \right)$ for all $\boldsymbol{c} = \left( c_{1}, \dotsc, c_{d -1} \right) \in K^{d -1}$. For every $j \in \lbrace 1, \dotsc, d -1 \rbrace$, the polynomial $F_{j}$ is homogeneous of degree $d$. Furthermore, we have the following:

\begin{claim}
\label{claim:resultant}
We have $\left\lvert \res\left( F_{1}, \dotsc, F_{d -1} \right) \right\rvert = 1$.
\end{claim}

\begin{proof}
Suppose that $p > d$ is a prime number. Let us show that $\res\left( F_{1}, \dotsc, F_{d -1} \right)$ is not divisible by $p$ in $A$. Note that $A/p A$ is the field $\mathbb{F}_{p}$ with $p$ elements. Thus, denoting by $\overline{\mathbb{F}}_{p}$ the algebraic closure of $\mathbb{F}_{p}$, we have a natural ring homomorphism $\varphi \colon A \rightarrow \overline{\mathbb{F}}_{p}$. Then $\res\left( F_{1}, \dotsc, F_{d -1} \right) \in p A$ if and only if $\varphi\left( \res\left( F_{1}, \dotsc, F_{d -1} \right) \right) = 0$, which occurs if and only if $\varphi\left( F_{1} \right), \dotsc, \varphi\left( F_{d -1} \right)$ have a common zero in $\overline{\mathbb{F}}_{p}^{d -1} \setminus \lbrace 0 \rbrace$. Suppose that $\mathfrak{c} = \left( c_{1}, \dotsc, c_{d -1} \right) \in \overline{\mathbb{F}}_{p}^{d -1}$ is a common zero of $\varphi\left( F_{1} \right), \dotsc, \varphi\left( F_{d -1} \right)$. We shall prove that $\mathfrak{c} = 0$. Define \[ \mathfrak{f}(z) = \frac{1}{d} z^{d} +\sum_{j = 1}^{d -1} \frac{(-1)^{j} \tau_{j}(\mathfrak{c})}{d -j} z^{d -j} \in \overline{\mathbb{F}}_{p}[z] \, \text{,} \] where $\tau_{1}(\mathfrak{c}), \dotsc, \tau_{d -1}(\mathfrak{c})$ denote the elementary symmetric functions of $c_{1}, \dotsc, c_{d -1}$. We have $\mathfrak{f}^{\prime}(z) = \prod\limits_{j = 1}^{d -1} \left( z -c_{j} \right)$ and $\mathfrak{f}\left( c_{j} \right) = \varphi\left( F_{j} \right)(\mathfrak{c}) = 0$ for each $j \in \lbrace 1, \dotsc, d -1 \rbrace$. Now, define $\Gamma = \left\lbrace c_{1}, \dotsc, c_{d -1} \right\rbrace$ and, for $\gamma \in \Gamma$, denote by $\rho_{\gamma}$ the number of indices $j \in \lbrace 1, \dotsc, d -1 \rbrace$ such that $c_{j} = \gamma$. We have $\mathfrak{f}^{\prime}(z) = \prod\limits_{\gamma \in \Gamma} (z -\gamma)^{\rho_{\gamma}}$ and $\mathfrak{f}(\gamma) = 0$ for each $\gamma \in \Gamma$. As a result, as $\overline{\mathbb{F}}_{p}$ has characteristic $p > d$, each $\gamma \in \Gamma$ has multiplicity $\rho_{\gamma} +1$ as a preimage of $0$ under $\mathfrak{f}$ and every other preimage of $0$ has multiplicity $1$. It follows that \[ d = r +\sum_{\gamma \in \Gamma} \left( \rho_{\gamma} +1 \right) = r +d -1 +s \, \text{,} \] where $r \geq 0$ is the number of preimages of $0$ under $\mathfrak{f}$ that are not in $\Gamma$ and $s \geq 1$ is the cardinality of $\Gamma$, which yields $r = 0$ and $s = 1$. Therefore, as $\mathfrak{f}(0) = 0$, we have $\Gamma = \lbrace 0 \rbrace$, and hence $\mathfrak{c} = 0$. Thus, we have proved that $\res\left( F_{1}, \dotsc, F_{d -1} \right) \in A \setminus p A$ for each prime number $p > d$. In particular, $\res\left( F_{1}, \dotsc, F_{d -1} \right) \in A$ is not divisible by the residue characteristic of $K$. This completes the proof of the claim.
\end{proof}

This allows us to determine the maximal escape rate $M_{f_{\boldsymbol{c}}}$ of $f_{\boldsymbol{c}}$, with $\boldsymbol{c} \in K^{d -1}$.

\begin{claim}
\label{claim:maxEscape}
We have $M_{f_{\boldsymbol{c}}} = \log^{+}\lVert \boldsymbol{c} \rVert$ for all $\boldsymbol{c} \in K^{d -1}$.
\end{claim}

\begin{proof}
Suppose that $\boldsymbol{c} = \left( c_{1}, \dotsc, c_{d -1} \right) \in K^{d -1}$. We have $g_{f_{\boldsymbol{c}}}\left( c_{j} \right) \leq \log^{+}\lVert \boldsymbol{c} \rVert$ for all $j \in \lbrace 1, \dotsc, d -1 \rbrace$ by the first assertion of Claim~\ref{claim:estimates2}, and hence $M_{f_{\boldsymbol{c}}} \leq \log^{+}\lVert \boldsymbol{c} \rVert$. It remains to show that $M_{f_{\boldsymbol{c}}} \geq \log^{+}\lVert \boldsymbol{c} \rVert$. If $\lVert \boldsymbol{c} \rVert \leq 1$, this is immediate. Thus, assume now that $\lVert \boldsymbol{c} \rVert > 1$. There exist some integer $D \geq d$ and homogeneous polynomials $G_{j, k} \in A\left[ T_{1}, \dotsc, T_{d -1} \right]$ of degree $D -d$, with $j, k \in \lbrace 1, \dotsc, d -1 \rbrace$, such that \[ \res\left( F_{1}, \dotsc, F_{d -1} \right) c_{j}^{D} = \sum_{k = 1}^{d -1} f_{\boldsymbol{c}}\left( c_{k} \right) G_{j, k}(\boldsymbol{c}) \] for each $j \in \lbrace 1, \dotsc, d -1 \rbrace$. Therefore, by Claim~\ref{claim:resultant}, we have \[ \left\lvert c_{j} \right\rvert^{D} \leq \max_{k \in \lbrace 1, \dotsc, d -1 \rbrace} \left\lvert f_{\boldsymbol{c}}\left( c_{k} \right) G_{j, k}(\boldsymbol{c}) \right\rvert \leq \left( \max_{k \in \lbrace 1, \dotsc, d -1 \rbrace} \left\lvert f_{\boldsymbol{c}}\left( c_{k} \right) \right\rvert \right) \lVert \boldsymbol{c} \rVert^{D -d} \] for all $j \in \lbrace 1, \dotsc, d -1 \rbrace$, and hence $\max\limits_{k \in \lbrace 1, \dotsc, d -1 \rbrace} \left\lvert f_{\boldsymbol{c}}\left( c_{k} \right) \right\rvert \geq \lVert \boldsymbol{c} \rVert^{d}$. Thus, there exists $k \in \lbrace 1, \dotsc, d -1 \rbrace$ such that $\left\lvert f_{\boldsymbol{c}}\left( c_{k} \right) \right\rvert \geq \lVert \boldsymbol{c} \rVert^{d}$. By the second assertion of Claim~\ref{claim:estimates2}, as $\lVert \boldsymbol{c} \rVert > 1$, it follows that \[ d \cdot g_{f_{\boldsymbol{c}}}\left( c_{j} \right) = g_{f_{\boldsymbol{c}}}\left( f_{\boldsymbol{c}}\left( c_{j} \right) \right) = \log^{+}\left\lvert f_{\boldsymbol{c}}\left( c_{j} \right) \right\rvert \geq d \cdot \log^{+}\lVert \boldsymbol{c} \rVert \, \text{.} \] Thus, we have $M_{f_{\boldsymbol{c}}} \geq \log^{+}\lVert \boldsymbol{c} \rVert$, and the claim is proved.
\end{proof}

From the discussion above, we now derive results about the Green functions of arbitrary polynomials in $\Poly_{d}(K)$.

\begin{lemma}
\label{lemma:greenDisks1}
Suppose that $f \in \Poly_{d}(K)$ has leading coefficient $a_{d} \in K^{*}$. Then, for each $\eta \in \mathbb{R}_{> 0}$, the set $\left\lbrace g_{f} < \eta \right\rbrace$ is a nonempty finite union of open disks and the set $\left\lbrace g_{f} \leq \eta \right\rbrace$ is a nonempty finite union of closed disks. Moreover, $\left\lbrace g_{f} < \eta \right\rbrace$ is an open disk of radius $\left\lvert a_{d} \right\rvert^{\frac{-1}{d -1}} \exp(\eta)$ for all $\eta \in \left( M_{f}, +\infty \right)$. In addition, $\left\lbrace g_{f} \leq \eta \right\rbrace$ is a closed disk of radius $\left\lvert a_{d} \right\rvert^{\frac{-1}{d -1}} \exp(\eta)$ for all $\eta \in \left[ M_{f}, +\infty \right)$.
\end{lemma}

\begin{proof}
Choose a $(d -1)$th root $\alpha \in K^{*}$ of $d \cdot a_{d}$ and a fixed point $w \in K$ for $f$, and define $\phi(z) = \alpha (z -w) \in \Aff(K)$. Then $\phi \centerdot f \in \Poly_{d}(K)$ has leading coefficient $\frac{1}{d}$ and satisfies $\phi \centerdot f(0) = 0$, and therefore $\phi \centerdot f = f_{\boldsymbol{c}}$, where $c_{1}, \dotsc, c_{d -1} \in K$ are the critical points for $\phi \centerdot f$ and $\boldsymbol{c} = \left( c_{1}, \dotsc, c_{d -1} \right)$. Now, by Claims~\ref{claim:estimates2} and~\ref{claim:maxEscape}, we have $\left\lbrace g_{f_{\boldsymbol{c}}} < \eta \right\rbrace = D\left( 0, \exp(\eta) \right)$ for all $\eta \in \left( M_{f_{\boldsymbol{c}}}, +\infty \right)$. Moreover, we have $g_{f_{\boldsymbol{c}}} = g_{f} \circ \phi^{-1}$ and $M_{f_{\boldsymbol{c}}} = M_{f}$ by conjugation. Therefore, as $\lvert \alpha \rvert = \left\lvert a_{d} \right\rvert^{\frac{1}{d -1}}$, we have \[ \forall \eta \in \left( M_{f}, +\infty \right), \, \left\lbrace g_{f} < \eta \right\rbrace = \phi^{-1}\left( \left\lbrace g_{f_{\boldsymbol{c}}} < \eta \right\rbrace \right) = D\left( w, \left\lvert a_{d} \right\rvert^{\frac{-1}{d -1}} \exp(\eta) \right) \, \text{.} \] Similarly, we have $\left\lbrace g_{f_{\boldsymbol{c}}} \leq \eta \right\rbrace = \overline{D\left( 0, \exp(\eta) \right)}$ for each $\eta \in \left[ M_{f_{\boldsymbol{c}}}, +\infty \right)$ by Claims~\ref{claim:estimates2} and~\ref{claim:maxEscape}, and hence \[ \forall \eta \in \left[ M_{f}, +\infty \right), \, \left\lbrace g_{f} \leq \eta \right\rbrace = \phi^{-1}\left( \left\lbrace g_{f_{\boldsymbol{c}}} \leq \eta \right\rbrace \right) = \overline{D\left( w, \left\lvert a_{d} \right\rvert^{\frac{-1}{d -1}} \exp(\eta) \right)} \, \text{.} \] Finally, suppose that $\eta \in \mathbb{R}_{> 0}$. There exists an integer $k \geq 0$ such that $d^{k} \eta > M_{f}$. Then $\left\lbrace g_{f} < d^{k} \eta \right\rbrace$ is an open disk in $K$ by the previous discussion. By Lemma~\ref{lemma:preimage}, it follows that \[ \left\lbrace g_{f} < \eta \right\rbrace = \left( f^{\circ k} \right)^{-1}\left( \left\lbrace g_{f} < d^{k} \eta \right\rbrace \right) \] is a nonempty finite union of open disks. Similarly, $\left\lbrace g_{f} \leq d^{k} \eta \right\rbrace$ is a closed disk in $K$ by the previous discussion, and hence \[ \left\lbrace g_{f} \leq \eta \right\rbrace = \left( f^{\circ k} \right)^{-1}\left( \left\lbrace g_{f} \leq d^{k} \eta \right\rbrace \right) \] is a nonempty finite union of closed disks by Lemma~\ref{lemma:preimage}. This completes the proof of the lemma.
\end{proof}

Finally, we also have the following non-Archimedean analogue of Lemma~\ref{lemma:disconnected}:

\begin{lemma}
\label{lemma:greenDisks2}
Suppose that $f \in \Poly_{d}(K)$ satisfies $M_{f} > 0$, and denote by $a_{d} \in K^{*}$ its leading coefficient and by $C \geq 1$ the number of critical points $c \in K$ for $f$ such that $g_{f}(c) = M_{f}$, counting multiplicities. Then $\left\lbrace g_{f} < M_{f} \right\rbrace$ has exactly $C +1$ disk components, and these are all open disks of radius $\left\lvert a_{d} \right\rvert^{\frac{-1}{d -1}} \exp\left( M_{f} \right)$.
\end{lemma}

\begin{proof}
By Lemma~\ref{lemma:greenDisks1}, the set $\left\lbrace g_{f} < M_{f} \right\rbrace$ is a nonempty finite union of open disks. Denote here by $U_{1}, \dotsc, U_{N}$, with $N \geq 1$, its disk components. For $j \in \lbrace 1, \dotsc, N \rbrace$, denote by $C_{j} \geq 0$ the number of critical points for $f$ in $U_{j}$, counting multiplicities. Note that $\sum\limits_{j = 1}^{N} C_{j} = d -1 -C$. Now, as $\left\lbrace g_{f} < d \cdot M_{f} \right\rbrace$ is a disk by Lemma~\ref{lemma:greenDisks1} and \[ \left\lbrace g_{f} < M_{f} \right\rbrace = f^{-1}\left( \left\lbrace g_{f} < d \cdot M_{f} \right\rbrace \right) \, \text{,} \] the map $f \colon U_{j} \rightarrow \left\lbrace g_{f} < d \cdot M_{f} \right\rbrace$ has a degree $d_{j} \geq 1$ for all $j \in \lbrace 1, \dotsc, N \rbrace$, and we have $d = \sum\limits_{j = 1}^{N} d_{j}$, by Lemma~\ref{lemma:preimage}. In addition, $d_{j} = C_{j} +1$ for each $j \in \lbrace 1, \dotsc, N \rbrace$ by Lemma~\ref{lemma:rhFormula}. Therefore, we have \[ d = \sum_{j = 1}^{N} \left( C_{j} +1 \right) = d -1 -C +N \, \text{,} \] and hence $N = C +1$, as desired.

Finally, note that $\left\lbrace g_{f} \leq M_{f} \right\rbrace$ is a disk of radius $\left\lvert a_{d} \right\rvert^{\frac{-1}{d -1}} \exp\left( M_{f} \right)$ by Lemma~\ref{lemma:greenDisks1}. Moreover, for every $j \in \lbrace 1, \dotsc, C +1 \rbrace$, we have \[ f\left( U_{j} \right) = \left\lbrace g_{f} < d \cdot M_{f} \right\rbrace \quad \text{and} \quad f\left( \left\lbrace g_{f} \leq M_{f} \right\rbrace \right) = \left\lbrace g_{f} \leq d \cdot M_{f} \right\rbrace \, \text{,} \] and these are two disks of the same radius by Lemma~\ref{lemma:greenDisks1}. By Lemma~\ref{lemma:image}, it follows that $U_{j}$ also has radius $\left\lvert a_{d} \right\rvert^{\frac{-1}{d -1}} \exp\left( M_{f} \right)$ for all $j \in \lbrace 1, \dotsc, C +1 \rbrace$. This completes the proof of the lemma.
\end{proof}

\subsection{A combinatorial argument}

Now, let us obtain an analogue of Lemma~\ref{lemma:combinatArch}, which plays a key role in our proof of Theorem~\ref{theorem:degenB} in the Archimedean case.

To do this, we shall first prove the two-islands lemma below, which is the non-Archimedean counterpart of Lemma~\ref{lemma:islandsArch}.

\begin{lemma}
\label{lemma:islandsNonArch}
Suppose that $U, V$ are disks, $f \colon U \rightarrow V$ is a polynomial map of degree $e \geq 1$ and $V_{1}, V_{2}$ are disjoint disks contained in $V$. Also assume that $e$ is less than the residue characteristic of $K$ if the latter is positive. Then there exist $j \in \lbrace 1, 2 \rbrace$ and a disk component $U_{j}$ of $f^{-1}\left( V_{j} \right)$ such that $f$ induces a bijection from $U_{j}$ to $V_{j}$.
\end{lemma}

\begin{proof}
For $j \in \lbrace 1, 2 \rbrace$, denote here by $C_{j} \geq 0$ the number of critical points for $f$ in $f^{-1}\left( V_{j} \right)$, counting multiplicities. Then, by Lemma~\ref{lemma:rhFormula}, we have \[ e = C +1 \geq C_{1} +C_{2} +1 \geq 2 \min\left\lbrace C_{1}, C_{2} \right\rbrace +1 \, \text{,} \] where $C \geq 0$ is the number of critical points for $f$ in $U$, counting multiplicities. By Lemma~\ref{lemma:preimage}, the set $f^{-1}\left( V_{j} \right)$ is a nonempty finite union of disks for each $j \in \lbrace 1, 2 \rbrace$. For $j \in \lbrace 1, 2 \rbrace$, denote by $U_{j}^{(1)}, \dotsc, U_{j}^{\left( N_{j} \right)}$ its disk components, with $N_{j} \geq 1$. Then, by Lemma~\ref{lemma:preimage}, for each $j \in \lbrace 1, 2 \rbrace$, the map $f \colon U_{j}^{(\ell)} \rightarrow V_{j}$ has a degree $e_{j}^{(\ell)} \geq 1$ for all $\ell \in \left\lbrace 1, \dotsc, N_{j} \right\rbrace$, and we have $e = \sum\limits_{\ell = 1}^{N_{j}} e_{j}^{(\ell)}$. In addition, for all $j \in \lbrace 1, 2 \rbrace$ and all $\ell \in \left\lbrace 1, \dotsc, N_{j} \right\rbrace$, we have $e_{j}^{(\ell)} = C_{j}^{(\ell)} +1$ by Lemma~\ref{lemma:rhFormula}, where $C_{j}^{(\ell)} \geq 0$ denotes the number of critical points for $f$ in $U_{j}^{(\ell)}$, counting multiplicities. Thus, we have \[ \forall j \in \lbrace 1, 2 \rbrace, \, e = \sum_{\ell = 1}^{N_{j}} \left( C_{j}^{(\ell)} +1 \right) = C_{j} +N_{j} \, \text{.} \] Therefore, as $e \geq 2 \min\left\lbrace C_{1}, C_{2} \right\rbrace +1$, there exists $j \in \lbrace 1, 2 \rbrace$ such that $C_{j} < N_{j}$. As a result, $C_{j}^{(\ell)} = 0$ for some $\ell \in \left\lbrace 1, \dotsc, N_{j} \right\rbrace$, and we have $e_{j}^{(\ell)} = 1$. Thus, $f$ induces a bijection from $U_{j}^{(\ell)}$ to $V_{j}$, and the lemma is proved.
\end{proof}

\begin{remark}
The assumptions that $f \colon U \rightarrow V$ is surjective and that its degree $e$ is finite and satisfies a certain condition related to the residue characteristic of $K$ are essential in our proof of Lemma~\ref{lemma:islandsNonArch}. In~\cite{Be2003}, Benedetto proved a more involved two-islands theorem for holomorphic functions on a disk in an algebraically closed field that is complete with respect to a non-Archimedean and nontrivial absolute value. In~\cite{B2008}, Benedetto also proved a non-Archimedean four-islands theorem for meromorphic functions. We refer the reader to these articles for more details.
\end{remark}

Finally, we have the result below, which is completely analogous to Lemma~\ref{lemma:combinatArch}.

\begin{lemma}
\label{lemma:combinatNonArch}
Suppose that $f \in \Poly_{d}(K)$ satisfies $M_{f} > 0$. Then one of the following two conditions is satisfied:
\begin{enumerate}
\item\label{item:combinatNonArch1} there exists a disk component $U$ of $\left\lbrace g_{f} < \frac{M_{f}}{d} \right\rbrace$ such that $f$ induces a bijection from $U$ onto the disk component $V$ of $\left\lbrace g_{f} < M_{f} \right\rbrace$ containing $U$;
\item\label{item:combinatNonArch2} for all distinct disk components $V, V^{\prime}$ of $\left\lbrace g_{f} < M_{f} \right\rbrace$, there exists a disk component $U$ of $\left\lbrace g_{f} < \frac{M_{f}}{d} \right\rbrace$ contained in $V$ such that $f$ induces a bijection from $U$ onto $V^{\prime}$.
\end{enumerate}
In addition, if $d \in \lbrace 2, 3 \rbrace$, then there exists a disk component $V$ of $\left\lbrace g_{f} < M_{f} \right\rbrace$ such that $f$ induces a bijection from $V$ onto $\left\lbrace g_{f} < d \cdot M_{f} \right\rbrace$.
\end{lemma}

\begin{proof}
Assume here that the condition~\eqref{item:combinatNonArch1} does not hold, and let us show that the condition~\eqref{item:combinatNonArch2} is satisfied. By Lemma~\ref{lemma:greenDisks1}, $\left\lbrace g_{f} < \frac{M_{f}}{d} \right\rbrace$ and $\left\lbrace g_{f} < M_{f} \right\rbrace$ are nonempty finite unions of disks. Now, assume that $V, V^{\prime}$ are two distinct disk components of $\left\lbrace g_{f} < M_{f} \right\rbrace$. It follows from Lemma~\ref{lemma:preimage} that $f^{-1}(V)$ and $f^{-1}\left( V^{\prime} \right)$ are both unions of disk components of $\left\lbrace g_{f} < \frac{M_{f}}{d} \right\rbrace$. Now, as $\left\lbrace g_{f} < d \cdot M_{f} \right\rbrace$ is a disk by Lemma~\ref{lemma:greenDisks1}, the induced map $f \colon V \rightarrow \left\lbrace g_{f} < d \cdot M_{f} \right\rbrace$ has a degree $e \in \lbrace 1, \dotsc, d \rbrace$ by Lemma~\ref{lemma:preimage}. Moreover, no disk component of $\left\lbrace g_{f} < \frac{M_{f}}{d} \right\rbrace$ contained in $V$ is mapped bijectively onto $V$ by $f$ by hypothesis. Therefore, by Lemma~\ref{lemma:islandsNonArch}, $f$ maps a disk component of $\left\lbrace g_{f} < \frac{M_{f}}{d} \right\rbrace$ contained in $V$ bijectively onto $V^{\prime}$. Thus, the desired result is proved.

Finally, assume that $d \in \lbrace 2, 3 \rbrace$. By Lemma~\ref{lemma:greenDisks1}, $\left\lbrace g_{f} < M_{f} \right\rbrace$ is a nonempty finite union of disks and $\left\lbrace g_{f} < d \cdot M_{f} \right\rbrace$ is a disk. In fact, by Lemma~\ref{lemma:greenDisks2}, $\left\lbrace g_{f} < M_{f} \right\rbrace$ has several disk components $V_{1}, \dotsc, V_{N}$, with $N \geq 2$. By Lemma~\ref{lemma:preimage}, the induced map $f \colon V_{j} \rightarrow \left\lbrace g_{f} < d \cdot M_{f} \right\rbrace$ has a degree $d_{j} \geq 1$ for each $j \in \lbrace 1, \dotsc, N \rbrace$, and $d = \sum\limits_{j = 1}^{N} d_{j}$. Therefore, as $d \leq 3$, there exists $j \in \lbrace 1, \dotsc, N \rbrace$ such that $d_{j} = 1$, and $f$ induces a bijection from $V_{j}$ onto $\left\lbrace g_{f} < d \cdot M_{f} \right\rbrace$. This completes the proof of the lemma.
\end{proof}

\subsection{Multipliers and maximal escape rates}

Here, let us relate multipliers at periodic points to maximal escape rates under certain combinatorial assumptions. More precisely, let us obtain a non-Archimedean analogue of Lemma~\ref{lemma:ineqArch}.

In the non-Archimedean setting, ratios of radii of disks play the role of moduli of complex annuli. Thus, we have the well-known result below, which follows from Lemma~\ref{lemma:image} and is the non-Archimedean counterpart of Lemma~\ref{lemma:modulusArch}.

\begin{lemma}
\label{lemma:modulusNonArch}
Suppose that $U \subsetneq V$ are disks and $f \colon U \rightarrow V$ is a bijective polynomial map. Then $f$ has a unique fixed point $z_{0} \in U$ and $\left\lvert f^{\prime}\left( z_{0} \right) \right\rvert = \frac{s}{r}$, where $r, s \in \mathbb{R}_{> 0}$ are the radii of $U, V$, respectively.
\end{lemma}

\begin{proof}
Note that $s > r$ since $U \subsetneq V$ by hypothesis and $U$ and $V$ are both open or both closed by Lemma~\ref{lemma:image}. Now, choose $w \in f^{-1}(U)$. By Lemma~\ref{lemma:image}, as $w \in U$ and $f \colon U \rightarrow V$ is bijective, we have $\left\lvert f^{\prime}(w) \right\rvert = \frac{s}{r}$ and $s \geq \left\lvert \frac{f^{(j)}(w)}{j!} \right\rvert r^{j}$ for each $j \geq 2$, with strict inequality if $U$ is closed. Now, define $g(z) = f(z) -z \in K[z]$. Then we have $\left\lvert g^{\prime}(w) \right\rvert = \left\lvert f^{\prime}(w) -1 \right\rvert = \frac{s}{r}$ since $\left\lvert f^{\prime}(w) \right\rvert = \frac{s}{r} > 1$. Moreover, for each $j \geq 2$, we have $s \geq \left\lvert \frac{g^{(j)}(w)}{j!} \right\rvert r^{j}$, with strict inequality if $U$ is closed, since $g^{(j)}(w) = f^{(j)}(w)$. As a result, $g(U)$ is a disk of radius $s$ and the induced map $g \colon U \rightarrow g(U)$ is bijective by Lemma~\ref{lemma:image}. Moreover, we have $\left\lvert g(w) \right\rvert = \left\lvert f(w) -w \right\rvert \leq r$ as $w \in f^{-1}(U)$, and hence $0 \in g(U)$. Therefore, the map $f \colon U \rightarrow V$ has a unique fixed point $z_{0} \in U$. Finally, as $f \colon U \rightarrow V$ is bijective, we have $\left\lvert f^{\prime}\left( z_{0} \right) \right\rvert = \frac{s}{r}$ by Lemma~\ref{lemma:image}. This completes the proof of the lemma.
\end{proof}

We also have the well-known result below, which follows easily from Lemma~\ref{lemma:image} and is a non-Archimedean analogue of Gr\"{o}tzsch's inequality.

\begin{lemma}
\label{lemma:radii}
Suppose that $f \in K[z]$ is not constant and $U \subseteq V$ are disks. Then \[ \left( \frac{R}{r} \right)^{e} \leq \frac{S}{s} \leq \left( \frac{R}{r} \right)^{E} \, \text{,} \] where $r, R, s, S$ are the radii of $U, V, f(U), f(V)$, respectively, and $e$ and $E$ are the degrees of $f \colon U \rightarrow f(U)$ and $f \colon V \rightarrow f(V)$, respectively.
\end{lemma}

\begin{proof}
Choose $w \in U$. Then $U, V$ are disks of center $w$ and radii $r, R$, respectively. Therefore, by Lemma~\ref{lemma:image}, we have \[ s = \left\lvert \frac{f^{(e)}(w)}{e!} \right\rvert r^{e} \geq \left\lvert \frac{f^{(E)}(w)}{E!} \right\rvert r^{E} \quad \text{and} \quad S = \left\lvert \frac{f^{(E)}(w)}{E!} \right\rvert R^{E} \geq \left\lvert \frac{f^{(e)}(w)}{e!} \right\rvert R^{e} \, \text{.} \] This completes the proof of the lemma.
\end{proof}

Finally, we obtain the result below, which is the non-Archimedean counterpart of Lemma~\ref{lemma:ineqArch}. Note that, in the present context, we also have an upper bound on the absolute values of multipliers.

\begin{lemma}
\label{lemma:ineqNonArch}
Suppose that $f \in \Poly_{d}(K)$ satisfies $M_{f} > 0$, $\eta \geq M_{f}$, $U_{0}, \dotsc, U_{p -1}$ are disk components of $\left\lbrace g_{f} < \frac{\eta}{d^{k}} \right\rbrace$, with $k \geq 0$ and $p \geq 1$, $V_{0}, \dotsc, V_{p -1}$ are the disk components of $\left\lbrace g_{f} < \frac{\eta}{d^{k -1}} \right\rbrace$ containing $U_{0}, \dotsc, U_{p -1}$, respectively, and $f$ induces a bijection from $U_{j}$ to $V_{j +1 \pmod{p}}$ for all $j \in \lbrace 0, \dotsc, p -1 \rbrace$. Then $f^{\circ p}$ has a unique fixed point $z_{0} \in K$ such that $f^{\circ j}\left( z_{0} \right) \in U_{j}$ for all $j \in \lbrace 0, \dotsc, p -1 \rbrace$. Furthermore, we have \[ (d -1) \left( \sum_{j = 0}^{p -1} \frac{1}{e_{j}} \right) \eta \leq \log\left\lvert \left( f^{\circ p} \right)^{\prime}\left( z_{0} \right) \right\rvert \leq (d -1) \left( \sum_{j = 0}^{p -1} \frac{1}{d_{j}} \right) \eta \, \text{,} \] where $d_{j}$ and $e_{j}$ denote the degrees of the induced maps $f^{\circ k} \colon \overline{U}_{j} \rightarrow \left\lbrace g_{f} \leq \eta \right\rbrace$ and $f^{\circ k} \colon V_{j} \rightarrow \left\lbrace g_{f} < d \cdot \eta \right\rbrace$, respectively, and $\overline{U}_{j}$ is the disk component of $\left\lbrace g_{f} \leq \frac{\eta}{d^{k}} \right\rbrace$ containing $U_{j}$ for all $j \in \lbrace 0, \dotsc, p -1 \rbrace$.
\end{lemma}

\begin{proof}
For $j \in \lbrace 0, \dotsc, p -1 \rbrace$, define $f_{j} \colon U_{j} \rightarrow V_{j +1 \pmod{p}}$ to be the bijective map induced by $f$. For $j \in \lbrace 0, \dotsc, p \rbrace$, define \[ W_{j} = \left( f_{j -1} \circ \dotsb \circ f_{0} \right)^{-1}\left( V_{j \pmod{p}} \right) = \left( f_{j -2} \circ \dotsb \circ f_{0} \right)^{-1}\left( U_{j -1} \right) \, \text{,} \] where $W_{0} = V_{0}$ and $W_{1} = U_{0}$ by convention. It follows from Lemma~\ref{lemma:preimage} that $W_{j}$ is a disk for all $j \in \lbrace 0, \dotsc, p \rbrace$. In addition, for each $j \in \lbrace 0, \dotsc, p -1 \rbrace$, we have \[ f^{\circ k}\left( U_{j} \right) \subseteq \left\lbrace g_{f} \leq \eta \right\rbrace \subsetneq \left\lbrace g_{f} < d \cdot \eta \right\rbrace = f^{\circ k}\left( V_{j} \right) \] by Lemmas~\ref{lemma:preimage} and~\ref{lemma:greenDisks1}, which yields $U_{j} \subsetneq V_{j}$, and hence $W_{j +1} \subsetneq W_{j}$. As a result, by Lemma~\ref{lemma:modulusNonArch}, the map $f_{p -1} \circ \dotsb \circ f_{0} \colon W_{p} \rightarrow W_{0}$ has a unique fixed point $z_{0} \in W_{p}$ and we have \[ \left\lvert \left( f^{\circ p} \right)^{\prime}\left( z_{0} \right) \right\rvert = \frac{\rho_{0}}{\rho_{p}} = \prod_{j = 0}^{p -1} \frac{\rho_{j}}{\rho_{j +1}} \, \text{,} \] where $\rho_{0}, \dotsc, \rho_{p}$ denote the radii of $W_{0}, \dotsc, W_{p}$, respectively. Now, note that $z_{0}$ is the unique fixed point for $f^{\circ p}$ that satisfies $f^{\circ j}\left( z_{0} \right) \in U_{j}$ for all $j \in \lbrace 0, \dotsc, p -1 \rbrace$. Thus, it remains to prove the desired inequalities. For $j \in \lbrace 0, \dotsc, p -1 \rbrace$, define $r_{j}$ and $R_{j}$ to be the radii of $U_{j}$ and $V_{j}$, respectively. Then, for all $j \in \lbrace 0, \dotsc, p -1 \rbrace$, we have $\frac{\rho_{j}}{\rho_{j +1}} = \frac{R_{j}}{r_{j}}$ by Lemma~\ref{lemma:radii} since $f_{j -1} \circ \dotsb \circ f_{0}$ maps bijectively $W_{j}$ onto $V_{j}$ and $W_{j +1}$ onto $U_{j}$. Thus, it suffices to prove that $\left( \frac{d -1}{e_{j}} \right) \eta \leq \log\left( \frac{R_{j}}{r_{j}} \right) \leq \left( \frac{d -1}{d_{j}} \right) \eta$ for all $j \in \lbrace 0, \dotsc, p -1 \rbrace$. Suppose that $j \in \lbrace 0, \dotsc, p -1 \rbrace$. As $d \cdot \eta > M_{f}$, we have $f^{\circ (k +1)}\left( U_{j} \right) = \left\lbrace g_{f} < d \cdot \eta \right\rbrace$ and $f^{\circ (k +1)}\left( \overline{U}_{j} \right) = \left\lbrace g_{f} \leq d \cdot \eta \right\rbrace$ by Lemmas~\ref{lemma:preimage} and~\ref{lemma:greenDisks1}, and these are disks of the same radius by Lemma~\ref{lemma:greenDisks1}. As a result, the disk $\overline{U}_{j}$ also has radius $r_{j}$ by Lemma~\ref{lemma:image}. Therefore, as $f^{\circ k}$ maps $\overline{U}_{j}$ onto $\left\lbrace g_{f} \leq \eta \right\rbrace$ with degree $d_{j}$ and $V_{j}$ onto $\left\lbrace g_{f} < d \cdot \eta \right\rbrace$ with degree $e_{j}$, we have \[ \left( \frac{R_{j}}{r_{j}} \right)^{d_{j}} \leq \exp\left( (d -1) \eta \right) \leq \left( \frac{R_{j}}{r_{j}} \right)^{e_{j}} \] by Lemmas~\ref{lemma:greenDisks1} and~\ref{lemma:radii}. This completes the proof of the lemma.
\end{proof}

\begin{remark}
To prove Theorem~\ref{theorem:degenB} in the non-Archimedean case, we shall only use Lemma~\ref{lemma:ineqNonArch} with $\eta = M_{f}$, $k \in \lbrace 0, 1 \rbrace$ and $p \in \lbrace 1, 2 \rbrace$ to only derive lower bounds on the absolute values of multipliers at certain small cycles. Nonetheless, our general statement of Lemma~\ref{lemma:ineqNonArch} also allows us to show that the bounds in Theorem~\ref{theorem:degenB} are optimal (see Propositions~\ref{proposition:sharp1} and~\ref{proposition:sharp2}) and to obtain a lower bound on the absolute values of multipliers at all cycles in terms of the periods and the minimum of the Green function on the set of critical points (see Proposition~\ref{proposition:ineqMin}).
\end{remark}

\subsection{Proof of Theorem~\ref{theorem:degenB} in the non-Archimedean case}

Now, let us derive Theorem~\ref{theorem:degenB} in the non-Archimedean case from Lemmas~\ref{lemma:combinatNonArch} and~\ref{lemma:ineqNonArch}. This is similar to our proof of Theorem~\ref{theorem:degenB} in the Archimedean case.

\begin{proof}[Proof of Theorem~\ref{theorem:degenB} in the non-Archimedean case]
Assume here that $K$ is an algebraically closed field of characteristic $0$ that is equipped with a non-Archimedean absolute value $\lvert . \rvert$ with residue characteristic $0$ or greater than $d$ and $f \in \Poly_{d}(K)$. Note that the desired result is immediate if the absolute value $\lvert . \rvert$ is trivial. Thus, from now on, assume that $\lvert . \rvert$ is not trivial.

First, suppose that $M_{f} = 0$. Denote by $\lambda_{1}, \dotsc, \lambda_{d}$ the multipliers of $f$ at all its fixed points repeated according to their multiplicities. Define $\sigma_{1}, \dotsc, \sigma_{d}$ to be the elementary symmetric functions of $\lambda_{1}, \dotsc, \lambda_{d}$. We have $d +\sum\limits_{j = 1}^{d} (-1)^{j} (d -j) \sigma_{j} = 0$ by the holomorphic fixed-point formula. Therefore, since $\lvert d \rvert = 1$, we have $\left\lvert \sigma_{j} \right\rvert \geq 1$ for some $j \in \lbrace 1, \dotsc, d \rbrace$ by the ultrametric triangle inequality. As a result, we have $\left\lvert \lambda_{k} \right\rvert \geq 1$ for some $k \in \lbrace 1, \dotsc, d \rbrace$ by the ultrametric triangle inequality. This shows that $M_{f}^{(1)} \geq 0$, as desired.

Thus, from now on, assume that $M_{f} > 0$. By Lemma~\ref{lemma:greenDisks1}, $\left\lbrace g_{f} < M_{f} \right\rbrace$ is a finite union of disks and $\left\lbrace g_{f} < d \cdot M_{f} \right\rbrace$ is a disk. Moreover, $\left\lbrace g_{f} < M_{f} \right\rbrace$ has several disk components $V_{1}, \dotsc, V_{N}$, with $N \geq 2$, by Lemma~\ref{lemma:greenDisks2}. Now, by Lemma~\ref{lemma:preimage}, the map $f \colon V_{j} \rightarrow \left\lbrace g_{f} < d \cdot M_{f} \right\rbrace$ has some degree $d_{j} \geq 1$ for all $j \in \lbrace 1, \dotsc, N \rbrace$, and we have $d = \sum\limits_{j = 1}^{N} d_{j}$. To conclude, let us consider three cases.

Suppose that $d_{j} = 1$ for some $j \in \lbrace 1, \dotsc, N \rbrace$. Note that this holds if $d \in \lbrace 2, 3 \rbrace$. Then $f$ induces a bijection from $V_{j}$ onto $\left\lbrace g_{f} < d \cdot M_{f} \right\rbrace$. Therefore, by Lemma~\ref{lemma:ineqNonArch}, $f$ has a unique fixed point $z_{0} \in V_{j}$ and we have $\log\left\lvert f^{\prime}\left( z_{0} \right) \right\rvert \geq (d -1) M_{f}$. Thus, we have $M_{f}^{(1)} \geq (d -1) M_{f}$.

Now, suppose that the condition~\eqref{item:combinatNonArch1} of Lemma~\ref{lemma:combinatNonArch} is satisfied and $d_{j} \geq 2$ for all $j \in \lbrace 1, \dotsc, N \rbrace$. Then there exist some $j \in \lbrace 1, \dotsc, N \rbrace$ and a disk component $U_{j}$ of $\left\lbrace g_{f} < \frac{M_{f}}{d} \right\rbrace$ contained in $V_{j}$ such that $f$ maps bijectively $U_{j}$ to $V_{j}$. As a result, by Lemma~\ref{lemma:ineqNonArch}, $f$ has a unique fixed point $z_{0} \in U_{j}$ and we have $\log\left\lvert f^{\prime}\left( z_{0} \right) \right\rvert \geq \frac{d -1}{d_{j}} M_{f}$. Moreover, $d_{j} = d -\sum\limits_{k \neq j} d_{k} \leq d -2$. Thus, we have $M_{f}^{(1)} \geq \frac{d -1}{d -2} M_{f}$.

Finally, suppose that the condition~\eqref{item:combinatNonArch2} of Lemma~\ref{lemma:combinatNonArch} is satisfied. Then there are disk components $U_{1}, U_{2}$ of $\left\lbrace g_{f} < \frac{M_{f}}{d} \right\rbrace$ contained in $V_{1}, V_{2}$, respectively, such that $f$ maps bijectively $U_{1}$ to $V_{2}$ and $U_{2}$ to $V_{1}$. As a result, by Lemma~\ref{lemma:ineqNonArch}, the map $f^{\circ 2}$ has a unique fixed point $z_{0} \in K$ such that $z_{0} \in U_{1}$ and $f\left( z_{0} \right) \in U_{2}$ and we have \[ \log\left\lvert \left( f^{\circ 2} \right)^{\prime}\left( z_{0} \right) \right\rvert \geq (d -1) \left( \frac{1}{d_{1}} +\frac{1}{d_{2}} \right) M_{f} \geq (d -1) \left( \frac{1}{d_{1}} +\frac{1}{d -d_{1}} \right) M_{f} \, \text{.} \] Moreover, $\frac{d -1}{2} \left( \frac{1}{d_{1}} +\frac{1}{d -d_{1}} \right) \geq C_{d}$. Thus, we have $M_{f}^{(2)} \geq C_{d} \cdot M_{f}$. This completes the proof of the theorem.
\end{proof}

\begin{remark}
As in the complex setting, one can strengthen Lemmas~\ref{lemma:islandsNonArch} and~\ref{lemma:combinatNonArch} and deduce that $M_{f}^{(2)} \geq 2 M_{f}$ for all $f \in \Poly_{d}(K)$ such that $M_{f} > 0$.
\end{remark}

\subsection{Sharpness of the bounds}

To end this section, let us show that the bounds in Theorem~\ref{theorem:degenB} are sharp in some sense. We shall first study the non-Archimedean case and then deduce that our bounds are also optimal in the complex setting. To avoid making our discussion too long, we omit the simpler cases where $d \in \lbrace 2, 3 \rbrace$.

We shall use the following result from non-Archimedean dynamics:

\begin{lemma}
\label{lemma:nonRepelling}
Suppose that $U \subseteq V$ are disks and $f \colon U \rightarrow V$ is a polynomial map of degree $e \geq 2$ whose critical points all lie in $\bigcap\limits_{n \geq 0} \left( f^{\circ n} \right)^{-1}(V)$. Also assume that $e$ is less than the residue characteristic of $K$ if the latter is positive. Then we have $\left\lvert \left( f^{\circ p} \right)^{\prime}\left( z_{0} \right) \right\rvert \leq 1$ for each periodic point $z_{0} \in K$ for $f$ with period $p \geq 1$.
\end{lemma}

\begin{proof}
First, note that $f$ has exactly $e -1$ critical points in $U$, counting multiplicities, by Lemma~\ref{lemma:rhFormula}. Now, assume that $\left( f^{\circ n} \right)^{-1}(V)$ is a disk for some $n \geq 0$. Then, by Lemma~\ref{lemma:preimage}, $\left( f^{\circ (n +1)} \right)^{-1}(V)$ is the union of finitely many pairwise disjoint disks $W_{1}, \dotsc, W_{N}$, with $N \geq 1$, the induced map $f \colon W_{j} \rightarrow \left( f^{\circ n} \right)^{-1}(V)$ has some degree $e_{j} \geq 1$ for each $j \in \lbrace 1, \dotsc, N \rbrace$, and we have $e = \sum\limits_{j = 1}^{N} e_{j}$. For $j \in \lbrace 1, \dotsc, N \rbrace$, denote by $C_{j} \geq 0$ the number of critical points for $f$ in $W_{j}$, counting multiplicities. Then $\sum\limits_{j = 1}^{N} C_{j} = e -1$ by assumption. In addition, $e_{j} = C_{j} +1$ for each $j \in \lbrace 1, \dotsc, N \rbrace$ by Lemma~\ref{lemma:rhFormula}. As a result, we have $e = \sum\limits_{j = 1}^{N} \left( C_{j} +1 \right) = e -1 +N$, which yields $N = 1$ and $e_{1} = e$. Thus, $\left( \left( f^{\circ n} \right)^{-1}(V) \right)_{n \geq 0}$ is a decreasing sequence of disks and the map $f \colon \left( f^{\circ (n +1)} \right)^{-1}(V) \rightarrow \left( f^{\circ n} \right)^{-1}(V)$ has degree $e$ for each $n \geq 0$. For $n \geq 0$, denote by $r_{n} \in \mathbb{R}_{> 0}$ the radius of $\left( f^{\circ n} \right)^{-1}(V)$. We have $\frac{r_{n -1}}{r_{n}} = \left( \frac{r_{n}}{r_{n +1}} \right)^{e}$ for each $n \geq 1$ by Lemma~\ref{lemma:radii}. Now, suppose that $z_{0} \in K$ is a periodic point for $f$ with period $p \geq 1$. As $z_{0} \in \bigcap\limits_{n \geq 0} \left( f^{\circ n} \right)^{-1}(V)$, we have $\left\lvert \left( f^{\circ p} \right)^{\prime}\left( z_{0} \right) \right\rvert \leq \frac{r_{n}}{r_{n +p}}$ for each $n \geq 0$ by Lemma~\ref{lemma:image}. Thus, we have $\left\lvert \left( f^{\circ p} \right)^{\prime}\left( z_{0} \right) \right\rvert \leq \left( \frac{r_{0}}{r_{p}} \right)^{\frac{1}{e^{n}}}$ for each $n \geq 0$, which yields $\left\lvert \left( f^{\circ p} \right)^{\prime}\left( z_{0} \right) \right\rvert \leq 1$ by letting $n \rightarrow +\infty$. This completes the proof of the lemma.
\end{proof}

Exhibiting explicit examples, we obtain the two results below, which show that the bounds in Theorem~\ref{theorem:degenB} are optimal in the non-Archimedean case.

\begin{proposition}
\label{proposition:sharp1}
Assume here that $d \geq 4$. Then, for every $R \in \left\lvert K^{*} \right\rvert$, there exists $f \in \Poly_{d}(K)$ such that \[ M_{f} = \log^{+}(R) \, \text{,} \quad M_{f}^{(1)} = 0 \quad \text{and} \quad M_{f}^{(2)} = C_{d} \cdot M_{f} \, \text{,} \] where $C_{d} \in \mathbb{R}_{> 0}$ is defined in Theorem~\ref{theorem:degenB}.
\end{proposition}

\begin{proof}
Observe that, if $R \in (0, 1]$, then $f(z) = z^{d}$ satisfies the required conditions. Now, assume that $R \in \left\lvert K^{*} \right\rvert \cap (1, +\infty)$. Choose $\gamma \in K^{*}$ such that $\lvert \gamma \rvert = R$. Set \[ \left( d_{0}, d_{1} \right) = \begin{cases} \left( \frac{d}{2}, \frac{d}{2} \right) & \text{if } d \text{ is even}\\ \left( \frac{d -1}{2}, \frac{d +1}{2} \right) & \text{if } d \text{ is odd} \end{cases} \, \text{,} \quad \text{so that} \quad C_{d} = \frac{d -1}{2} \left( \frac{1}{d_{0}} +\frac{1}{d_{1}} \right) \, \text{.} \] Define \[ f(z) = f_{\boldsymbol{c}}(z) = \sum_{j = 0}^{d_{1} -1} b_{j} z^{d_{0} +j} (z -\gamma)^{d_{1} -1 -j} \left( d_{0} z -\left( d_{0} +1 +j \right) \omega \right) \, \text{,} \] with $b_{j} = \frac{(-1)^{j} \left( d_{0} -1 \right)! \left( d_{1} -1 \right)!}{\left( d_{0} +1 +j \right)! \left( d_{1} -1 -j \right)!}$ for all $j \in \left\lbrace 0, \dotsc, d_{1} -1 \right\rbrace$, where \[ \omega = \frac{d_{0}}{d} \gamma +\frac{(-1)^{d_{1}} (d -1)!}{\left( d_{0} -1 \right)! \left( d_{1} -1 \right)!} \gamma^{2 -d} \quad \text{and} \quad \boldsymbol{c} = \left( \underbrace{0, \dotsc, 0}_{d_{0} -1 \text{ entries}}, \underbrace{\gamma, \dotsc, \gamma}_{d_{1} -1 \text{ entries}}, \omega \right) \, \text{.} \] Thus, we have \[ f(0) = 0 \, \text{,} \quad f(\gamma) = \gamma \quad \text{and} \quad f^{\prime}(z) = z^{d_{0} -1} (z -\gamma)^{d_{1} -1} (z -\omega) \, \text{.} \] Note that $\lvert \omega \rvert = R$. As a result, we have $M_{f} = \log(R) > 0$ by Claim~\ref{claim:maxEscape}. Moreover, we have $g_{f}(0) = 0$ and $g_{f}(\gamma) = 0$, which yields $g_{f}(\omega) = M_{f}$. Therefore, $\left\lbrace g_{f} < M_{f} \right\rbrace$ is the union of two disjoint open disks $V_{0}, V_{1}$ of radius $R$ by Lemma~\ref{lemma:greenDisks2}. Without loss of generality, we can assume that $0 \in V_{0}$ and $\gamma \in V_{1}$. Thus, by Lemmas~\ref{lemma:preimage}, \ref{lemma:rhFormula} and~\ref{lemma:greenDisks1}, the maps $f_{0} \colon V_{0} \rightarrow \left\lbrace g_{f} < d \cdot M_{f} \right\rbrace$ and $f_{1} \colon V_{1} \rightarrow \left\lbrace g_{f} < d \cdot M_{f} \right\rbrace$ induced by $f$ have degrees $d_{0}$ and $d_{1}$, respectively, since $0$ and $\gamma$ are critical points for $f$ with multiplicities $d_{0} -1$ and $d_{1} -1$, respectively. Now, by Lemmas~\ref{lemma:preimage} and~\ref{lemma:rhFormula}, $f_{0}^{-1}\left( V_{1} \right)$ is the union of $d_{0}$ pairwise distinct disk components $U_{0}^{(1)}, \dotsc, U_{0}^{\left( d_{0} \right)}$ of $\left\lbrace g_{f} < \frac{M_{f}}{d} \right\rbrace$ and $f$ maps bijectively $U_{0}^{(k)}$ to $V_{1}$ for all $k \in \left\lbrace 1, \dotsc, d_{0} \right\rbrace$. Similarly, $f_{1}^{-1}\left( V_{0} \right)$ is the union of $d_{1}$ pairwise distinct disk components $U_{1}^{(1)}, \dotsc, U_{1}^{\left( d_{1} \right)}$ of $\left\lbrace g_{f} < \frac{M_{f}}{d} \right\rbrace$ and $f$ maps bijectively $U_{1}^{(\ell)}$ to $V_{0}$ for all $\ell \in \left\lbrace 1, \dotsc, d_{1} \right\rbrace$.

Now, suppose that $k \in \left\lbrace 1, \dotsc, d_{0} \right\rbrace$ and $\ell \in \left\lbrace 1, \dotsc, d_{1} \right\rbrace$. Then there is a unique periodic point $z_{0} \in K$ for $f$ with period $2$ such that $z_{0} \in U_{0}^{(k)}$ and $f\left( z_{0} \right) \in U_{1}^{(\ell)}$ by Lemma~\ref{lemma:ineqNonArch}. Now, define $\overline{U}_{0} = f_{0}^{-1}\left( \left\lbrace g_{f} \leq M_{f} \right\rbrace \right)$ and $\overline{U}_{1} = f_{1}^{-1}\left( \left\lbrace g_{f} \leq M_{f} \right\rbrace \right)$. Then $\overline{U}_{0}$ and $\overline{U}_{1}$ are disks and the maps $f \colon \overline{U}_{0} \rightarrow \left\lbrace g_{f} \leq M_{f} \right\rbrace$ and $f \colon \overline{U}_{1} \rightarrow \left\lbrace g_{f} \leq M_{f} \right\rbrace$ have degrees $d_{0}$ and $d_{1}$, respectively, by Lemmas~\ref{lemma:preimage}, \ref{lemma:rhFormula} and~\ref{lemma:greenDisks1}. Thus, $\overline{U}_{0}$ and $\overline{U}_{1}$ are also the disk components of $\left\lbrace g_{f} \leq \frac{M_{f}}{d} \right\rbrace$ containing $z_{0}$ and $f\left( z_{0} \right)$, respectively. Therefore, by Lemma~\ref{lemma:ineqNonArch}, we have \[ \frac{1}{2} \log\left\lvert \left( f^{\circ 2} \right)^{\prime}\left( z_{0} \right) \right\rvert = \frac{d -1}{2} \left( \frac{1}{d_{0}} +\frac{1}{d_{1}} \right) M_{f} = C_{d} \cdot M_{f} \, \text{.} \]

Let us conclude the proof of the proposition. If $z_{0} \in K$ is any fixed point for $f$, then the point $z_{0}$ is also fixed for $f_{j} \colon V_{j} \rightarrow \left\lbrace g_{f} < d \cdot M_{f} \right\rbrace$ for some $j \in \lbrace 0, 1 \rbrace$, and hence $\log\left\lvert f^{\prime}\left( z_{0} \right) \right\rvert \leq 0$ by Lemma~\ref{lemma:nonRepelling}. In addition, we always have $M_{f}^{(1)} \geq 0$ by the holomorphic fixed-point formula, as shown in the proof of Theorem~\ref{theorem:degenB}. This shows that $M_{f}^{(1)} = 0$. Now, suppose that $z_{0} \in K$ is a periodic point for $f$ with period $2$. If $z_{0}$ and $f\left( z_{0} \right)$ both lie in $V_{j}$ for some $j \in \lbrace 0, 1 \rbrace$, then the point $z_{0}$ is also periodic for $f_{j}$, and therefore $\frac{1}{2} \log\left\lvert \left( f^{\circ 2} \right)^{\prime}\left( z_{0} \right) \right\rvert \leq 0$ by Lemma~\ref{lemma:nonRepelling}. Otherwise, replacing $z_{0}$ by $f\left( z_{0} \right)$ if necessary, we have $z_{0} \in U_{0}^{(k)}$ and $f\left( z_{0} \right) \in U_{1}^{(\ell)}$ for some $k \in \left\lbrace 1, \dotsc, d_{0} \right\rbrace$ and some $\ell \in \left\lbrace 1, \dotsc, d_{1} \right\rbrace$, and hence $\frac{1}{2} \log\left\lvert \left( f^{\circ 2} \right)^{\prime}\left( z_{0} \right) \right\rvert = C_{d} \cdot M_{f}$ by the discussion above. Moreover, we also proved that the latter case occurs for some choices of $z_{0}$. Thus, we have $M_{f}^{(2)} = C_{d} \cdot M_{f}$. This completes the proof of the proposition.
\end{proof}

\begin{proposition}
\label{proposition:sharp2}
Assume here that $d \geq 4$. Then, for every $R \in \left\lvert K^{*} \right\rvert$, there exists $f \in \Poly_{d}(K)$ such that \[ M_{f} = \log^{+}(R) \, \text{,} \quad M_{f}^{(1)} = \frac{d -1}{d -2} M_{f} \quad \text{and} \quad M_{f}^{(2)} = \frac{d -1}{d -2} M_{f} \, \text{.} \]
\end{proposition}

\begin{proof}
Observe that, if $R \in (0, 1]$, then $f(z) = z^{d}$ satisfies the required conditions. Now, assume that $R \in \left\lvert K^{*} \right\rvert \cap (1, +\infty)$. Choose $\gamma \in K^{*}$ such that $\lvert \gamma \rvert = R$. Define \[ f(z) = f_{\boldsymbol{c}}(z) = \frac{1}{d} z^{2} (z -\gamma)^{d -2} \, \text{,} \quad \text{with} \quad \boldsymbol{c} = \left( 0, \underbrace{\gamma, \dotsc, \gamma}_{d -3 \text{ entries}}, \frac{2}{d} \gamma \right) \in K^{d -1} \, \text{.} \] Thus, we have \[ f(0) = 0 \, \text{,} \quad f(\gamma) = 0 \quad \text{and} \quad f^{\prime}(z) = z (z -\gamma)^{d -3} \left( z -\frac{2}{d} \gamma \right) \, \text{.} \] Note that $M_{f} = \log(R) > 0$ by Claim~\ref{claim:maxEscape}. Moreover, $g_{f}(0) = 0$ and $g_{f}(\gamma) = 0$, and hence $g_{f}\left( \frac{2}{d} \gamma \right) = M_{f}$. Therefore, $\left\lbrace g_{f} < M_{f} \right\rbrace$ is the union of two disjoint open disks $V_{0}, V_{1}$ of radius $R$ by Lemma~\ref{lemma:greenDisks2}. Without loss of generality, we may assume that $0 \in V_{0}$ and $\gamma \in V_{1}$. Since $0$ and $\gamma$ are critical points for $f$ with multiplicities $1$ and $d -3$, respectively, the maps $f_{0} \colon V_{0} \rightarrow \left\lbrace g_{f} < d \cdot M_{f} \right\rbrace$ and $f_{1} \colon V_{1} \rightarrow \left\lbrace g_{f} < d \cdot M_{f} \right\rbrace$ induced by $f$ have degrees $2$ and $d -2$, respectively, by Lemmas~\ref{lemma:preimage}, \ref{lemma:rhFormula} and~\ref{lemma:greenDisks1}. By similar arguments, the sets $\overline{U}_{0} = f_{0}^{-1}\left( \left\lbrace g_{f} \leq M_{f} \right\rbrace \right)$ and $\overline{U}_{1} = f_{1}^{-1}\left( \left\lbrace g_{f} \leq M_{f} \right\rbrace \right)$ are all the disk components of $\left\lbrace g_{f} \leq \frac{M_{f}}{d} \right\rbrace$ and the induced maps $f \colon \overline{U}_{0} \rightarrow \left\lbrace g_{f} \leq M_{f} \right\rbrace$ and $f \colon \overline{U}_{1} \rightarrow \left\lbrace g_{f} \leq M_{f} \right\rbrace$ have degrees $2$ and $d -2$, respectively. The set $f_{1}^{-1}\left( V_{1} \right)$ is the union of $d -2$ distinct disk components $U_{1}^{(1)}, \dotsc, U_{1}^{(d -2)}$ of $\left\lbrace g_{f} < \frac{M_{f}}{d} \right\rbrace$ and $f$ maps bijectively $U_{1}^{(k)}$ onto $V_{1}$ for each $k \in \lbrace 1, \dotsc, d -2 \rbrace$ by Lemmas~\ref{lemma:preimage} and~\ref{lemma:rhFormula}. Similarly, the set $f_{0}^{-1}\left( V_{1} \right)$ is the union of two distinct disk components $W_{0}^{(1)}, W_{0}^{(2)}$ of $\left\lbrace g_{f} < \frac{M_{f}}{d} \right\rbrace$ and $f$ maps bijectively $W_{0}^{(k)}$ onto $V_{1}$ for each $k \in \lbrace 1, 2 \rbrace$. Moreover, the set $W_{1} = f_{1}^{-1}\left( V_{0} \right)$ is a disk component of $\left\lbrace g_{f} < \frac{M_{f}}{d} \right\rbrace$ and $f$ maps $W_{1}$ onto $V_{0}$ with degree $d -2$.

Suppose that $k \in \lbrace 1, \dotsc, d -2 \rbrace$. Then, by the discussion above and Lemma~\ref{lemma:ineqNonArch}, there is a unique fixed point $z_{0} \in U_{1}^{(k)}$ for $f$ and we have $\log\left\lvert f^{\prime}\left( z_{0} \right) \right\rvert = \frac{d -1}{d -2} M_{f}$.

Now, suppose that $k, \ell \in \lbrace 1, \dotsc, d -2 \rbrace$ are distinct. Then, by Lemma~\ref{lemma:ineqNonArch}, there exists a unique periodic point $z_{0} \in K$ for $f$ with period $2$ such that $z_{0} \in U_{1}^{(k)}$ and $f\left( z_{0} \right) \in U_{1}^{(\ell)}$ and we have $\frac{1}{2} \log\left\lvert \left( f^{\circ 2} \right)^{\prime}\left( z_{0} \right) \right\rvert = \frac{d -1}{d -2} M_{f}$.

Finally, suppose that $z_{0} \in K$ is any periodic point for $f$ with period $2$ such that $z_{0} \in V_{0}$ and $f\left( z_{0} \right) \in V_{1}$. Note that $z_{0} \in W_{0}^{(k)}$ and $f\left( z_{0} \right) \in W_{1}$ for some $k \in \lbrace 1, 2 \rbrace$. Now, denote by $D_{0}$ and $D_{1}$ the disk components of $\left\lbrace g_{f} < \frac{M_{f}}{d^{2}} \right\rbrace$ containing $z_{0}$ and $f\left( z_{0} \right)$, respectively. Then $f$ induces bijections from $D_{0}$ onto $W_{1}$ and from $D_{1}$ onto $W_{0}^{(k)}$ by Lemmas~\ref{lemma:preimage} and~\ref{lemma:rhFormula}. Therefore, $z_{0}$ is the unique periodic point for $f$ with period $2$ such that $z_{0} \in D_{0}$ and $f\left( z_{0} \right) \in D_{1}$ by Lemma~\ref{lemma:ineqNonArch}. First, observe that the maps $f^{\circ 2} \colon W_{0}^{(k)} \rightarrow \left\lbrace g_{f} < d \cdot M_{f} \right\rbrace$ and $f^{\circ 2} \colon W_{1} \rightarrow \left\lbrace g_{f} < d \cdot M_{f} \right\rbrace$ have degrees $d -2$ and $2 (d -2)$, respectively. Next, denote here by $\overline{D}_{0}$ and $\overline{D}_{1}$ the disk components of $\left\lbrace g_{f} \leq \frac{M_{f}}{d^{2}} \right\rbrace$ containing $z_{0}$ and $f\left( z_{0} \right)$, respectively. Then, by Lemmas~\ref{lemma:preimage} and~\ref{lemma:rhFormula}, $f$ maps bijectively $\overline{D}_{0}$ onto $\overline{U}_{1}$, we have $\overline{D}_{1} = f_{1}^{-1}\left( \overline{U}_{0} \right)$ and the map $f \colon \overline{D}_{1} \rightarrow \overline{U}_{0}$ has degree $d -2$. It follows that $f^{\circ 2} \colon \overline{D}_{0} \rightarrow \left\lbrace g_{f} \leq M_{f} \right\rbrace$ and $f^{\circ 2} \colon \overline{D}_{1} \rightarrow \left\lbrace g_{f} \leq M_{f} \right\rbrace$ have degrees $d -2$ and $2 (d -2)$, respectively. Therefore, by Lemma~\ref{lemma:ineqNonArch}, we have \[ \frac{1}{2} \log\left\lvert \left( f^{\circ 2} \right)^{\prime}\left( z_{0} \right) \right\rvert = \frac{d -1}{2} \left( \frac{1}{d -2} +\frac{1}{2 (d -2)} \right) M_{f} = \frac{3 (d -1)}{4 (d -2)} M_{f} \, \text{.} \]

Let us conclude the proof of the proposition. Suppose that $z_{0} \in K$ is any fixed point for $f$. If $z_{0} \in V_{0}$, then the point $z_{0}$ is also fixed for $f_{0} \colon V_{0} \rightarrow \left\lbrace g_{f} < d \cdot M_{f} \right\rbrace$, which yields $\log\left\lvert f^{\prime}\left( z_{0} \right) \right\rvert \leq 0$ by Lemma~\ref{lemma:nonRepelling}. Otherwise, we have $z_{0} \in U_{1}^{(k)}$ for some $k \in \lbrace 1, \dotsc, d -2 \rbrace$, and hence $\log\left\lvert f^{\prime}\left( z_{0} \right) \right\rvert = \frac{d -1}{d -2} M_{f}$ by the previous discussion. In addition, we also proved that the latter case occurs for certain choices of $z_{0}$. This shows that $M_{f}^{(1)} = \frac{d -1}{d -2} M_{f}$. Now, suppose that $z_{0} \in K$ is any periodic point for $f$ with period $2$. If $z_{0}$ and $f\left( z_{0} \right)$ both lie in $V_{0}$, then the point $z_{0}$ is also periodic for $f_{0}$, and hence $\frac{1}{2} \log\left\lvert \left( f^{\circ 2} \right)^{\prime}\left( z_{0} \right) \right\rvert \leq 0$ by Lemma~\ref{lemma:nonRepelling}. If $z_{0}$ and $f\left( z_{0} \right)$ both lie in $V_{1}$, then $z_{0} \in U_{1}^{(k)}$ and $f\left( z_{0} \right) \in U_{1}^{(\ell)}$ for some distinct $k, \ell \in \lbrace 1, \dotsc, d -2 \rbrace$, and hence $\frac{1}{2} \log\left\lvert \left( f^{\circ 2} \right)^{\prime}\left( z_{0} \right) \right\rvert = \frac{d -1}{d -2} M_{f}$ by the previous discussion. Otherwise, replacing $z_{0}$ by $f\left( z_{0} \right)$ if necessary, we have $z_{0} \in V_{0}$ and $f\left( z_{0} \right) \in V_{1}$, and thus the discussion above yields $\frac{1}{2} \log\left\lvert \left( f^{\circ 2} \right)^{\prime}\left( z_{0} \right) \right\rvert = \frac{3 (d -1)}{4 (d -2)} M_{f}$. Furthermore, we also proved that the second case occurs for some choices of $z_{0}$. Thus, we have $M_{f}^{(2)} = \frac{d -1}{d -2} M_{f}$. This completes the proof of the proposition.
\end{proof}

Applying Propositions~\ref{proposition:sharp1} and~\ref{proposition:sharp2} with the non-Archimedean field $K = \mathbb{C}\left\lbrace \! \left\lbrace \frac{1}{t} \right\rbrace \! \right\rbrace$ of convergent complex Puiseux series in $\frac{1}{t}$, with $t$ an indeterminate, we shall show that the bounds in Theorem~\ref{theorem:degenB} are also optimal in the complex setting.

Assume here that $t$ is an indeterminate and $K = \mathbb{C}\left\lbrace \! \left\lbrace \frac{1}{t} \right\rbrace \! \right\rbrace$ is the field of Puiseux series in $\frac{1}{t}$ with coefficients in $\mathbb{C}$ that converge on some punctured neighborhood of $t = \infty$. Then $K$ is algebraically closed according to the Newton--Puiseux theorem (see~\cite{N2000}). We equip $K$ with its usual absolute value $\lvert . \rvert$, which is given by \[ \lvert a \rvert = \lim_{t \rightarrow \infty} \exp\left( \frac{\log\left\lvert a(t) \right\rvert_{\infty}}{\log\lvert t \rvert_{\infty}} \right) \, \text{,} \] where $\lvert . \rvert_{\infty}$ denotes the usual absolute value on $\mathbb{C}$. Thus, $K$ is a non-Archimedean valued field and its residue field is naturally isomorphic to $\mathbb{C}$. Now, note that any meromorphic function on a neighborhood of $t = \infty$ in $\widehat{\mathbb{C}}$ can be identified with an element of $K$ via its Laurent series expansion at $t = \infty$. In particular, denoting by $\overline{\mathbb{D}}$ the closed unit disk around the origin in $\mathbb{C}$, every holomorphic family $\left( f_{t} \right)_{t \in \mathbb{C} \setminus \overline{\mathbb{D}}}$ of elements of $\Poly_{d}(\mathbb{C})$ with a pole at $t = \infty$ induces some element $f \in \Poly_{d}(K)$. Specifically, if a holomorphic family $\left( f_{t} \right)_{t \in \mathbb{C} \setminus \overline{\mathbb{D}}}$ of elements of $\Poly_{d}(\mathbb{C})$ with a pole at $t = \infty$ is given by $f_{t}(z) = \sum\limits_{j = 0}^{d} a_{j}(t) z^{j}$, with $a_{0}, \dotsc, a_{d}$ holomorphic on $\mathbb{C} \setminus \overline{\mathbb{D}}$ and meromorphic on $\widehat{\mathbb{C}} \setminus \overline{\mathbb{D}}$, then the induced element $f \in \Poly_{d}(K)$ is $f(z) = \sum\limits_{j = 0}^{d} a_{j} z^{j}$.

To deduce results in the complex setting from analogous statements in the non-Archimedean case, we shall use the result below. It is a particular case of a result due to DeMarco.

\begin{lemma}[{\cite[Proposition~3.1]{DM2016}}]
\label{lemma:familyMaxEscape}
Suppose here that $K = \mathbb{C}\left\lbrace \! \left\lbrace \frac{1}{t} \right\rbrace \! \right\rbrace$ is the field of convergent complex Puiseux series in $\frac{1}{t}$, $\left( f_{t} \right)_{t \in \mathbb{C} \setminus \overline{\mathbb{D}}}$ is a holomorphic family of elements of $\Poly_{d}(\mathbb{C})$ that has a pole at $t = \infty$ and $f \in \Poly_{d}(K)$ is the element induced by $\left( f_{t} \right)_{t \in \mathbb{C} \setminus \overline{\mathbb{D}}}$. Then $M_{f_{t}} = M_{f} \cdot \log\lvert t \rvert_{\infty} +o\left( \log\lvert t \rvert_{\infty} \right)$ as $t \rightarrow \infty$.
\end{lemma}

We shall also use the following fact:

\begin{lemma}
\label{lemma:familyMultiplier}
Suppose here that $K = \mathbb{C}\left\lbrace \! \left\lbrace \frac{1}{t} \right\rbrace \! \right\rbrace$ is the field of convergent complex Puiseux series in $\frac{1}{t}$, $\left( f_{t} \right)_{t \in \mathbb{C} \setminus \overline{\mathbb{D}}}$ is a holomorphic family of elements of $\Poly_{d}(\mathbb{C})$ that has a pole at $t = \infty$ and $f \in \Poly_{d}(K)$ is the element induced by $\left( f_{t} \right)_{t \in \mathbb{C} \setminus \overline{\mathbb{D}}}$. Then, for every integer $p \geq 1$, we have $M_{f_{t}}^{(p)} = M_{f}^{(p)} \cdot \log\lvert t \rvert_{\infty} +O(1)$ as $t \rightarrow \infty$.
\end{lemma}

\begin{proof}
Suppose that $p \geq 1$ is an integer. Write \[ \Lambda_{f}^{(p)} = \left[ \lambda_{1}, \dotsc, \lambda_{N} \right] \in K^{N}/\mathfrak{S}_{N} \, \text{,} \quad \text{with} \quad N = N_{d}^{(p)} \, \text{.} \] There exists $R \in \mathbb{R}_{> 0}$ such that the complex Puiseux series $\lambda_{1}, \dotsc, \lambda_{N}$ all converge on $\mathbb{C} \setminus \overline{\mathbb{D}}_{R}$, where $\overline{\mathbb{D}}_{R}$ denotes the closed disk of center $0$ and radius $R$ in $\mathbb{C}$. Then, as the elementary symmetric functions of the multipliers of polynomials of degree $d$ at all their cycles with period $p$ define regular functions $\sigma_{d, j}^{(p)} \circ \pi_{d}$ on $\Poly_{d}$, with $j \in \left\lbrace 1, \dotsc, N \right\rbrace$, we have \[ \forall t \in \mathbb{C} \setminus \overline{\mathbb{D}}_{R}, \, \Lambda_{f_{t}}^{(p)} = \left[ \lambda_{1}(t), \dotsc, \lambda_{N}(t) \right] \in \mathbb{C}^{N}/\mathfrak{S}_{N} \, \text{.} \] Finally, note that $\log\left\lvert a(t) \right\rvert_{\infty} = \log\lvert a \rvert \cdot \log\lvert t \rvert_{\infty} +O(1)$ as $t \rightarrow \infty$ for each $a \in K$. In particular, for each $j \in \lbrace 1, \dotsc, N \rbrace$, we have $\log\left\lvert \lambda_{j}(t) \right\rvert_{\infty} = \log\left\lvert \lambda_{j} \right\rvert \cdot \log\lvert t \rvert_{\infty} +O(1)$ as $t \rightarrow \infty$. This completes the proof of the lemma.
\end{proof}

Finally, combining Propositions~\ref{proposition:sharp1} and~\ref{proposition:sharp2} with Lemmas~\ref{lemma:familyMaxEscape} and~\ref{lemma:familyMultiplier}, we directly deduce the two results below, which show that the bounds in Theorem~\ref{theorem:degenB} are also sharp in the complex case.

\begin{corollary}
Assume that $d \geq 4$. Then there exists a rational family $\left( f_{t} \right)_{t \in \mathbb{C}^{*}}$ of elements of $\Poly_{d}(\mathbb{C})$ such that $f_{t}$ degenerates in $\mathcal{P}_{d}(\mathbb{C})$ as $t \rightarrow \infty$ and \[ M_{f_{t}}^{(1)} = O(1) \quad \text{and} \quad M_{f_{t}}^{(2)} \sim C_{d} \cdot M_{f_{t}} \quad \text{as} \quad t \rightarrow \infty \, \text{.} \]
\end{corollary}

\begin{proof}
Consider the rational family $\left( f_{t} \right)_{t \in \mathbb{C}^{*}}$ of elements of $\Poly_{d}(\mathbb{C})$ defined by \[ f_{t}(z) = \sum_{j = 0}^{d_{1} -1} b_{j} z^{d_{0} +j} (z -t)^{d_{1} -1 -j} \left( d_{0} z -\left( d_{0} +1 +j \right) \omega_{t} \right) \, \text{,} \] with $b_{j} = \frac{(-1)^{j} \left( d_{0} -1 \right)! \left( d_{1} -1 \right)!}{\left( d_{0} +1 +j \right)! \left( d_{1} -1 -j \right)!}$ for all $j \in \left\lbrace 0, \dotsc, d_{1} -1 \right\rbrace$, where \[ \left( d_{0}, d_{1} \right) = \begin{cases} \left( \frac{d}{2}, \frac{d}{2} \right) & \text{if } d \text{ is even}\\ \left( \frac{d -1}{2}, \frac{d +1}{2} \right) & \text{if } d \text{ is odd} \end{cases} \quad \text{and} \quad \omega_{t} = \frac{d_{0}}{d} t +\frac{(-1)^{d_{1}} (d -1)!}{\left( d_{0} -1 \right)! \left( d_{1} -1 \right)!} t^{2 -d} \, \text{.} \] Now, assume that $K = \mathbb{C}\left\lbrace \! \left\lbrace \frac{1}{t} \right\rbrace \! \right\rbrace$, and denote by $f \in \Poly_{d}(K)$ the element induced by $\left( f_{t} \right)_{t \in \mathbb{C}^{*}}$. Then, as $\lvert t \rvert = \exp(1)$, the proof of Proposition~\ref{proposition:sharp1} shows that $M_{f} = 1$, $M_{f}^{(1)} = 0$ and $M_{f}^{(2)} = C_{d} \cdot M_{f}$. Thus, the desired result follows immediately from Lemmas~\ref{lemma:familyMaxEscape} and~\ref{lemma:familyMultiplier}.
\end{proof}

\begin{corollary}
Assume that $d \geq 4$. Then there exists a polynomial family $\left( f_{t} \right)_{t \in \mathbb{C}}$ of elements of $\Poly_{d}(\mathbb{C})$ such that $f_{t}$ degenerates in $\mathcal{P}_{d}(\mathbb{C})$ as $t \rightarrow \infty$ and \[ M_{f_{t}}^{(1)} \sim \frac{d -1}{d -2} M_{f_{t}} \quad \text{and} \quad M_{f_{t}}^{(2)} \sim \frac{d -1}{d -2} M_{f_{t}} \quad \text{as} \quad t \rightarrow \infty \, \text{.} \]
\end{corollary}

\begin{proof}
Consider the polynomial family $\left( f_{t} \right)_{t \in \mathbb{C}}$ of elements of $\Poly_{d}(\mathbb{C})$ defined by \[ f_{t}(z) = \frac{1}{d} z^{2} (z -t)^{d -2} \, \text{.} \] Now, assume that $K = \mathbb{C}\left\lbrace \! \left\lbrace \frac{1}{t} \right\rbrace \! \right\rbrace$, and denote by $f \in \Poly_{d}(K)$ the element induced by $\left( f_{t} \right)_{t \in \mathbb{C}}$. Then, as $\lvert t \rvert = \exp(1)$, the proof of Proposition~\ref{proposition:sharp2} shows that $M_{f} = 1$, $M_{f}^{(1)} = \frac{d -1}{d -2} M_{f}$ and $M_{f}^{(2)} = \frac{d -1}{d -2} M_{f}$. Thus, the desired result follows immediately from Lemmas~\ref{lemma:familyMaxEscape} and~\ref{lemma:familyMultiplier}.
\end{proof}

\section{Unique determination of a generic conjugacy class of polynomial maps by its multipliers at its small cycles}
\label{section:unique}

In this section, we shall prove Theorem~\ref{theorem:unique}. As before, we fix an integer $d \geq 2$.

\subsection{Some preliminaries}

First, let us present the ingredients that we use in our proof of Theorem~\ref{theorem:unique}.

Since $\mathcal{P}_{d}(\mathbb{C}) \cong \mathcal{P}_{d}^{\mc}(\mathbb{C})$, we can restrict our attention to monic centered complex polynomials. Recall here that \[ \Poly_{d}^{\mc}(\mathbb{C}) = \left\lbrace z^{d} +\sum_{j = 0}^{d -2} b_{j} z^{j} : b_{0}, \dotsc, b_{d -2} \in \mathbb{C} \right\rbrace \, \text{.} \] Now, define $\alpha = \exp\left( \frac{2 \pi i}{d -1} \right)$, so that \[ \mu_{d -1}(\mathbb{C}) = \langle \alpha \rangle = \left\lbrace \alpha^{k} : k \in \lbrace 0, \dotsc, d -2 \rbrace \right\rbrace \, \text{.} \] Also recall that the group $\mu_{d -1}(\mathbb{C})$ acts on $\Poly_{d}^{\mc}(\mathbb{C})$ via $\omega \centerdot f = \omega f\left( \frac{z}{\omega} \right)$ and that $\mathcal{P}_{d}(\mathbb{C})$ is biholomorphic to the quotient $\mathcal{P}_{d}^{\mc}(\mathbb{C})$ of $\Poly_{d}^{\mc}(\mathbb{C})$ by $\mu_{d -1}(\mathbb{C})$.

Our objective is to prove the result below, which directly implies Theorem~\ref{theorem:unique}.

\begin{lemma}
\label{lemma:unique}
There exists a nonempty open subset $U$ of $\Poly_{d}^{\mc}(\mathbb{C})$ such that, for every $f \in U$, \[ \left\lbrace g \in \Poly_{d}^{\mc}(\mathbb{C}) : \Lambda_{g}^{(1)} = \Lambda_{f}^{(1)} \text{ and } \Lambda_{g}^{(2)} = \Lambda_{f}^{(2)} \right\rbrace = \left\lbrace \alpha^{k} \centerdot f : k \in \lbrace 0, \dotsc, d -2 \rbrace \right\rbrace \, \text{.} \]
\end{lemma}

Define $\Xi$ to be the set of all elements $\left[ \lambda_{0}, \dotsc, \lambda_{d -1} \right] \in \mathbb{C}^{d}/\mathfrak{S}_{d}$ that satisfy $\lambda_{j} \neq 1$ for all $j \in \lbrace 0, \dotsc, d -1 \rbrace$ and \[ \forall J \subseteq \lbrace 0, \dotsc, d -1 \rbrace, \, \sum_{j \in J} \frac{1}{1 -\lambda_{j}} = 0 \Longleftrightarrow J = \varnothing \text{ or } \lbrace 0, \dotsc, d -1 \rbrace \, \text{.} \] To prove Lemma~\ref{lemma:unique}, we shall use the result below, which is due to Fujimura.

\begin{lemma}[{\cite[Theorem~6]{F2007}}]
\label{lemma:fujimura}
Suppose that $\Lambda \in \Xi$. Then there exist at most $(d -1)!$ elements $f \in \Poly_{d}^{\mc}(\mathbb{C})$ such that $\Lambda_{f}^{(1)} = \Lambda$.
\end{lemma}

Now, define \[ f_{0}(z) = z^{d} \in \Poly_{d}^{\mc}(\mathbb{C}) \, \text{.} \] We shall also use the explicit expressions for the differentials of multiplier maps at $f_{0}$, which are due to Gorbovickis.

\begin{lemma}[{\cite[Lemma~3.1]{G2016}}]
\label{lemma:gorbovickis}
Suppose that $z_{0} \in \mathbb{C}^{*}$ is a periodic point for $f_{0}$ with period $p \geq 1$, $U$ is an open neighborhood of $f_{0}$ in $\Poly_{d}^{\mc}(\mathbb{C})$, $\zeta_{0} \colon U \rightarrow \mathbb{C}$ is a holomorphic map such that $\zeta_{0}\left( f_{0} \right) = z_{0}$ and $f^{\circ p}\left( \zeta_{0}(f) \right) = \zeta_{0}(f)$ for all $f \in U$ and $\rho_{0} \colon U \rightarrow \mathbb{C}$ is the holomorphic map defined by $\rho_{0}(f) = \left( f^{\circ p} \right)^{\prime}\left( \zeta_{0}(f) \right)$. Then \[ \forall k \in \lbrace 0, \dotsc, d -2 \rbrace, \, \frac{\partial \rho_{0}}{\partial a_{k}}\left( f_{0} \right) = d^{p -1} (k -d) \sum_{j = 0}^{p -1} z_{0}^{d^{j} (k -d)} \, \text{.} \]
\end{lemma}

\subsection{Multipliers at fixed points}

Here, let us parametrize some open neighborhood $U_{1}$ of $f_{0}$ in $\Poly_{d}^{\mc}(\mathbb{C})$ by the multipliers at the fixed points and describe, for a generic $f \in U_{1}$, all the elements $g \in \Poly_{d}^{\mc}(\mathbb{C})$ such that $\Lambda_{g}^{(1)} = \Lambda_{f}^{(1)}$.

The fixed points for $f_{0}$ are precisely $0$ and the points $\alpha^{j}$, with $j \in \lbrace 0, \dotsc, d -2 \rbrace$. In addition, we have $f_{0}^{\prime}(0) = 0$ and $f_{0}^{\prime}\left( \alpha^{j} \right) = d$ for all $j \in \lbrace 0, \dotsc, d -2 \rbrace$. It follows from the implicit function theorem that there exist an open neighborhood $U_{1}$ of $f_{0}$ in $\Poly_{d}^{\mc}(\mathbb{C})$ and holomorphic maps $\zeta_{j}^{(1)} \colon U_{1} \rightarrow \mathbb{C}$, with $j \in \lbrace 0, \dotsc, d -2 \rbrace \cup \lbrace \lozenge \rbrace$, such that \[ \zeta_{j}^{(1)}\left( f_{0} \right) = \begin{cases} \alpha^{j} & \text{if } j \in \lbrace 0, \dotsc, d -2 \rbrace\\ 0 & \text{if } j = \lozenge \end{cases} \quad \text{and} \quad \forall f \in U_{1}, \, f\left( \zeta_{j}^{(1)}(f) \right) = \zeta_{j}^{(1)}(f) \] for all $j \in \lbrace 0, \dotsc, d -2 \rbrace \cup \lbrace \lozenge \rbrace$. Shrinking $U_{1}$ if necessary, we may assume that the points $\zeta_{j}^{(1)}(f)$, with $j \in \lbrace 0, \dotsc, d -2 \rbrace \cup \lbrace \lozenge \rbrace$, are pairwise distinct for each $f \in U_{1}$. For every $f \in U_{1}$, we have \[ \Phi_{f}^{(1)}(z) = \left( z -\zeta_{\lozenge}^{(1)}(f) \right) \prod_{j = 0}^{d -2} \left( z -\zeta_{j}^{(1)}(f) \right) \, \text{.} \]

For $j \in \lbrace 0, \dotsc, d -2 \rbrace \cup \lbrace \lozenge \rbrace$, define the holomorphic map $\rho_{j}^{(1)} \colon U_{1} \rightarrow \mathbb{C} \setminus \lbrace 1 \rbrace$ by \[ \rho_{j}^{(1)}(f) = f^{\prime}\left( \zeta_{j}^{(1)}(f) \right) \, \text{.} \] For every $f \in U_{1}$, we have \[ \Lambda_{f}^{(1)} = \left[ \rho_{0}^{(1)}(f), \dotsc, \rho_{d -2}^{(1)}(f), \rho_{\lozenge}^{(1)}(f) \right] \quad \text{and} \quad \frac{1}{1 -\rho_{\lozenge}^{(1)}(f)} +\sum_{j = 0}^{d -2} \frac{1}{1 -\rho_{j}^{(1)}(f)} = 0 \, \text{.} \] Now, define the holomorphic map \[ \boldsymbol{\rho}_{1} = \left( \rho_{0}^{(1)}, \dotsc, \rho_{d -2}^{(1)} \right) \colon U_{1} \rightarrow \mathbb{C}^{d -1} \, \text{.} \] Denote by $\cdot$ the natural action of $\mathfrak{S}_{d -1}$ on $\mathbb{C}^{d -1}$, which is given by \[ \sigma \cdot \left( \lambda_{0}, \dotsc, \lambda_{d -2} \right) = \left( \lambda_{\sigma^{-1}(0)}, \dotsc, \lambda_{\sigma^{-1}(d -2)} \right) \, \text{.} \] By the formulas above,
\begin{equation}
\label{equation:permSpec0}
\forall f, g \in U_{1}, \, \left( \exists \sigma \in \mathfrak{S}_{d -1}, \, \boldsymbol{\rho}_{1}(g) = \sigma \cdot \boldsymbol{\rho}_{1}(f) \right) \Longrightarrow \Lambda_{f}^{(1)} = \Lambda_{g}^{(1)} \, \text{.}
\end{equation}

Now, define \[ A_{1} = \left( \frac{\partial \rho_{j}^{(1)}}{\partial a_{k}}\left( f_{0} \right) \right)_{0 \leq j, k \leq d -2} \in M_{(d -1) \times (d -1)}(\mathbb{C}) \] to be the Jacobian matrix of $\boldsymbol{\rho}_{1}$ at $f_{0}$. By Lemma~\ref{lemma:gorbovickis}, we have \[ A_{1} = \left( (k -d) \alpha^{j (k -1)} \right)_{0 \leq j, k \leq d -2} \, \text{.} \] As a consequence, we have the key result below. For $k \in \lbrace 0, \dotsc, d -2 \rbrace$, define the polynomial \[ P_{k}(T) = \alpha^{k} \prod_{\substack{0 \leq j \leq d -2\\ j \neq k}} \left( \frac{T -\alpha^{j}}{\alpha^{k} -\alpha^{j}} \right) \in \mathbb{C}[T] \, \text{.} \]

\begin{claim}
\label{claim:inverse}
The matrix $A_{1}$ is invertible and $A_{1}^{-1} = \left( B_{j, k} \right)_{0 \leq j, k \leq d -2}$, where \[ \forall k \in \lbrace 0, \dotsc, d -2 \rbrace, \, P_{k}(T) = \sum_{j = 0}^{d -2} (j -d) B_{j, k} T^{j} \, \text{.} \]
\end{claim}

\begin{proof}
Suppose that \[ B = \left( B_{j, k} \right)_{0 \leq j, k \leq d -2} \in M_{(d -1) \times (d -1)}(\mathbb{C}) \, \text{.} \] For $k \in \lbrace 0, \dotsc, d -2 \rbrace$, define the polynomial \[ Q_{k}(T) = \sum_{j = 0}^{d -2} (j -d) B_{j, k} T^{j} \in \mathbb{C}[T] \, \text{.} \] Also denote by $\delta$ the Kronecker delta. Then \[ \begin{split} A_{1} B = I_{d -2} & \Longleftrightarrow \forall j, k \in \lbrace 0, \dotsc, d -2 \rbrace, \, \sum_{\ell = 0}^{d -2} (\ell -d) B_{\ell, k} \alpha^{j (\ell -1)} = \delta_{j k}\\ & \Longleftrightarrow \forall j, k \in \lbrace 0, \dotsc, d -2 \rbrace, \, Q_{k}\left( \alpha^{j} \right) = \alpha^{j} \delta_{j k}\\ & \Longleftrightarrow \forall k \in \lbrace 0, \dotsc, d -2 \rbrace, \, Q_{k} = P_{k} \, \text{.} \end{split} \] This completes the proof of the claim.
\end{proof}

By Claim~\ref{claim:inverse} and the inverse function theorem, shrinking $U_{1}$ if necessary, we can assume that $\boldsymbol{\rho}_{1}$ induces a biholomorphism from $U_{1}$ onto an open neighborhood $V_{1}$ of $(d, \dotsc, d)$ in $\mathbb{C}^{d -1}$. In addition, shrinking further $U_{1}$ if necessary, we can assume that $U_{1}$ is connected, $\Lambda_{f}^{(1)} \in \Xi$ for all $f \in U_{1}$ and $V_{1}$ is invariant under the natural action of $\mathfrak{S}_{d -1}$.

Now, define the action $*$ of $\mathfrak{S}_{d -1}$ on $U_{1}$ by \[ \sigma * f = \boldsymbol{\rho}_{1}^{-1}\left( \sigma \cdot \boldsymbol{\rho}_{1}(f) \right) \, \text{.} \] Denote by $\Delta$ the fat diagonal of $\mathbb{C}^{d -1}$, which is given by \[ \Delta = \bigcup_{0 \leq j < k \leq d -2} \left\lbrace \left( \lambda_{0}, \dotsc, \lambda_{d -2} \right) \in \mathbb{C}^{d -1} : \lambda_{j} = \lambda_{k} \right\rbrace \, \text{.} \] By~\eqref{equation:permSpec0}, we have $\Lambda_{\sigma * f}^{(1)} = \Lambda_{f}^{(1)} \in \Xi$ for all $f \in U_{1}$ and all $\sigma \in \mathfrak{S}_{d -1}$. Moreover, for each $f \in U_{1} \setminus \boldsymbol{\rho}_{1}^{-1}(\Delta)$, the elements $\sigma * f$, with $\sigma \in \mathfrak{S}_{d -1}$, are pairwise distinct. It follows from Lemma~\ref{lemma:fujimura} that
\begin{equation}
\label{equation:permSpec1}
\forall f \in U_{1} \setminus \boldsymbol{\rho}_{1}^{-1}(\Delta), \, \left\lbrace g \in \Poly_{d}^{\mc}(\mathbb{C}) : \Lambda_{g}^{(1)} = \Lambda_{f}^{(1)} \right\rbrace = \left\lbrace \sigma * f : \sigma \in \mathfrak{S}_{d -1} \right\rbrace \, \text{.}
\end{equation}
In addition, we can describe conjugation in terms of this action. Define the cyclic permutation \[ \sigma_{0} = (0 \, \dotso \, d -2) \in \mathfrak{S}_{d -1} \, \text{.} \]

\begin{claim}
\label{claim:permConj1}
We have $\alpha^{k} \centerdot f = \sigma_{0}^{k} * f$ for all $f \in U_{1}$ and all $k \in \lbrace 0, \dotsc, d -2 \rbrace$.
\end{claim}

\begin{proof}
The set $U_{1} \setminus \boldsymbol{\rho}_{1}^{-1}(\Delta)$ is a connected open subset of $\Poly_{d}^{\mc}(\mathbb{C})$. Moreover, if $f \in U_{1} \setminus \boldsymbol{\rho}_{1}^{-1}(\Delta)$ and $k \in \lbrace 0, \dotsc, d -2 \rbrace$, then there exists a permutation $\sigma \in \mathfrak{S}_{d -1}$ such that $\alpha^{k} \centerdot f = \sigma * f$ by~\eqref{equation:permSpec1}, as $\Lambda_{\alpha^{k} \centerdot f}^{(1)} = \Lambda_{f}^{(1)}$, and hence $\alpha^{k} \centerdot f \in U_{1} \setminus \boldsymbol{\rho}_{1}^{-1}(\Delta)$. Thus, $U_{1} \setminus \boldsymbol{\rho}_{1}^{-1}(\Delta)$ is invariant under the action of $\mu_{d -1}(\mathbb{C}) = \langle \alpha \rangle$ by conjugation. As a result, if $j, k \in \lbrace 0, \dotsc, d -2 \rbrace$, then there exists $\ell \in \lbrace 0, \dotsc, d -2 \rbrace \cup \lbrace \lozenge \rbrace$ such that $\zeta_{j}^{(1)}\left( \alpha^{k} \centerdot f \right) = \alpha^{k} \zeta_{\ell}^{(1)}(f)$ for all $f \in U_{1} \setminus \boldsymbol{\rho}_{1}^{-1}(\Delta)$, and we obtain $\ell = \sigma_{0}^{-k}(j)$ by letting $f \rightarrow f_{0}$. Thus, we have $\zeta_{j}^{(1)}\left( \alpha^{k} \centerdot f \right) = \alpha^{k} \zeta_{\sigma_{0}^{-k}(j)}^{(1)}(f)$ for all $f \in U_{1} \setminus \boldsymbol{\rho}_{1}^{-1}(\Delta)$ and all $j, k \in \lbrace 0, \dotsc, d -2 \rbrace$. As a consequence, for each $f \in U_{1} \setminus \boldsymbol{\rho}_{1}^{-1}(\Delta)$ and each $k \in \lbrace 0, \dotsc, d -2 \rbrace$, we have $\boldsymbol{\rho}_{1}\left( \alpha^{k} \centerdot f \right) = \sigma_{0}^{k} \cdot \boldsymbol{\rho}_{1}(f)$ since the multiplier is invariant under conjugation, which yields $\alpha^{k} \centerdot f = \sigma_{0}^{k} * f$. Thus, $f \mapsto \alpha^{k} \centerdot f$ and $f \mapsto \sigma_{0}^{k} * f$ coincide on $U_{1} \setminus \boldsymbol{\rho}_{1}^{-1}(\Delta)$, and hence they coincide on all of $U_{1}$ since $U_{1} \setminus \boldsymbol{\rho}_{1}^{-1}(\Delta)$ is dense in $U_{1}$. This completes the proof of the claim.
\end{proof}

\subsection{Multipliers at cycles with period $2$}

Here, let us examine the variations of the multipliers at the cycles with period $2$ on some open neighborhood $U_{2}$ of $f_{0}$ in $\Poly_{d}^{\mc}(\mathbb{C})$.

First, the periodic points for $f_{0}$ with period $2$ are the $\left( d^{2} -1 \right)$th roots of unity that are not $(d -1)$th roots of unity. Choose representatives $w_{0}, \dotsc, w_{\frac{d (d -1)}{2} -1}$ for the cycles for $f_{0}$ with period $2$. Setting $\beta = \exp\left( \frac{2 \pi i}{d^{2} -1} \right)$, we can take $w_{j} = \alpha^{j} \beta$ for all $j \in \lbrace 0, \dotsc, d -2 \rbrace$ since these lie in pairwise distinct cycles for $f_{0}$ with period $2$. We have $\left( f_{0}^{\circ 2} \right)^{\prime}\left( w_{j} \right) = d^{2}$ for all $j \in \left\lbrace 0, \dotsc, \frac{d (d -1)}{2} -1 \right\rbrace$. By the implicit function theorem, there exist an open neighborhood $U_{2}$ of $f_{0}$ in $\Poly_{d}^{\mc}(\mathbb{C})$ and holomorphic maps $\zeta_{j}^{(2)} \colon U_{2} \rightarrow \mathbb{C}$, with $j \in \left\lbrace 0, \dotsc, \frac{d (d -1)}{2} -1 \right\rbrace$, such that \[ \zeta_{j}^{(2)}\left( f_{0} \right) = w_{j} \quad \text{and} \quad \forall f \in U_{2}, \, f^{\circ 2}\left( \zeta_{j}^{(2)}(f) \right) = \zeta_{j}^{(2)}(f) \] for all $j \in \left\lbrace 0, \dotsc, \frac{d (d -1)}{2} -1 \right\rbrace$. Shrinking $U_{2}$ if necessary, we may assume that $U_{2}$ is connected and invariant under the action of $\mu_{d -1}(\mathbb{C}) = \langle \alpha \rangle$ by conjugation. For every $f \in U_{2}$, we have \[ \Phi_{f}^{(2)}(z) = \prod_{j = 0}^{\frac{d (d -1)}{2} -1} \left( z -\zeta_{j}^{(2)}(f) \right) \left( z -f\left( \zeta_{j}^{(2)}(f) \right) \right) \, \text{.} \]

Now, for $j \in \left\lbrace 0, \dotsc, \frac{d (d -1)}{2} -1 \right\rbrace$, define the holomorphic map $\rho_{j}^{(2)} \colon U_{2} \rightarrow \mathbb{C}$ by \[ \rho_{j}^{(2)}(f) = \left( f^{\circ 2} \right)^{\prime}\left( \zeta_{j}^{(2)}(f) \right) \, \text{.} \] For every $f \in U_{2}$, we have \[ \Lambda_{f}^{(2)} = \left[ \rho_{0}^{(2)}(f), \dotsc, \rho_{\frac{d (d -1)}{2} -1}^{(2)}(f) \right] \, \text{.} \] Define the holomorphic map \[ \boldsymbol{\rho}_{2} = \left( \rho_{0}^{(2)}, \dotsc, \rho_{\frac{d (d -1)}{2} -1}^{(2)} \right) \colon U_{2} \rightarrow \mathbb{C}^{\frac{d (d -1)}{2}} \, \text{.} \] Also denote by $\cdot$ the natural action of $\mathfrak{S}_{\frac{d (d -1)}{2}}$ on $\mathbb{C}^{\frac{d (d -1)}{2}}$, which is given by \[ \tau \cdot \left( \lambda_{0}, \dotsc, \lambda_{\frac{d (d -1)}{2} -1} \right) = \left( \lambda_{\tau^{-1}(0)}, \dotsc, \lambda_{\tau^{-1}\left( \frac{d (d -1)}{2} -1 \right)} \right) \, \text{.} \] By the formula above,
\begin{equation}
\label{equation:permSpec2}
\forall f, g \in U_{2}, \, \Lambda_{f}^{(2)} = \Lambda_{g}^{(2)} \Longleftrightarrow \left( \exists \tau \in \mathfrak{S}_{\frac{d (d -1)}{2}}, \, \boldsymbol{\rho}_{2}(g) = \tau \cdot \boldsymbol{\rho}_{2}(f) \right) \, \text{.}
\end{equation}

We may describe the behavior of $\boldsymbol{\rho}_{2}$ under conjugation. Define $\tau_{0} \in \mathfrak{S}_{\frac{d (d -1)}{2}}$ to be the permutation such that, for each $j \in \left\lbrace 0, \dotsc, \frac{d (d -1)}{2} -1 \right\rbrace$, the point $\alpha w_{j}$ lies in the cycle for $f_{0}$ containing $w_{\tau_{0}(j)}$. By the choice of $w_{0}, \dotsc, w_{d -2}$, we have
\begin{equation}
\label{equation:permCycle}
\forall j \in \lbrace 0, \dotsc, d -2 \rbrace, \, \tau_{0}(j) = j +1 \pmod{d -1} \in \lbrace 0, \dotsc, d -2 \rbrace \, \text{.}
\end{equation}

\begin{claim}
\label{claim:permConj2}
We have $\boldsymbol{\rho}_{2}\left( \alpha^{k} \centerdot f \right) = \tau_{0}^{k} \cdot \boldsymbol{\rho}_{2}(f)$ for all $f \in U_{2}$ and all $k \in \lbrace 0, \dotsc, d -2 \rbrace$.
\end{claim}

\begin{proof}
For every $j \in \left\lbrace 0, \dotsc, \frac{d (d -1)}{2} -1 \right\rbrace$ and every $k \in \lbrace 0, \dotsc, d -2 \rbrace$, there exists $\ell \in \left\lbrace 0, \dotsc, \frac{d (d -1)}{2} -1 \right\rbrace$ such that $\alpha^{-k} \zeta_{j}^{(2)}\left( \alpha^{k} \centerdot f \right)$ lies in the cycle for $f$ containing $\zeta_{\ell}^{(2)}(f)$ for all $f \in U_{2}$, and we have $\ell = \tau_{0}^{-k}(j)$ since $\alpha^{-k} w_{j}$ lies in the cycle for $f_{0}$ containing $w_{\ell}$ by taking $f = f_{0}$. Thus, for each $f \in U_{2}$ and each $k \in \lbrace 0, \dotsc, d -2 \rbrace$, the point $\alpha^{-k} \zeta_{j}^{(2)}\left( \alpha^{k} \centerdot f \right)$ belongs to the cycle for $f$ containing $\zeta_{\tau_{0}^{-k}(j)}^{(2)}(f)$ for each $j \in \left\lbrace 0, \dotsc, \frac{d (d -1)}{2} -1 \right\rbrace$, and hence $\boldsymbol{\rho}_{2}\left( \alpha^{k} \centerdot f \right) = \tau_{0}^{k} \cdot \boldsymbol{\rho}_{2}(f)$ because the multiplier is invariant under conjugation. This completes the proof of the claim.
\end{proof}

Finally, define \[ A_{2} = \left( \frac{\partial \rho_{j}^{(2)}}{\partial a_{k}}\left( f_{0} \right) \right)_{\substack{0 \leq j \leq \frac{d (d -1)}{2} -1\\ 0 \leq k \leq d -2}} \in M_{\frac{d (d -1)}{2} \times (d -1)}(\mathbb{C}) \] to be the Jacobian matrix of $\boldsymbol{\rho}_{2}$ at $f_{0}$. By Lemma~\ref{lemma:gorbovickis}, we have
\begin{equation}
\label{equation:matrix}
A_{2} = \left( d (k -d) \left( w_{j}^{k -d} +w_{j}^{d (k -d)} \right) \right)_{\substack{0 \leq j \leq \frac{d (d -1)}{2} -1\\ 0 \leq k \leq d -2}} \, \text{.}
\end{equation}

\subsection{Proof of Theorem~\ref{theorem:unique}}

Finally, let us conclude here our proof of Lemma~\ref{lemma:unique} and deduce Theorem~\ref{theorem:unique}.

By~\eqref{equation:permSpec1} and Claim~\ref{claim:permConj1}, to prove Lemma~\ref{lemma:unique}, it suffices to show that there exists a nonempty open subset $U \subseteq U_{1} \setminus \boldsymbol{\rho}_{1}^{-1}(\Delta)$ of $\Poly_{d}^{\mc}(\mathbb{C})$ such that
\begin{equation}
\label{equation:goal1}
\forall f \in U, \, \left\lbrace \sigma \in \mathfrak{S}_{d -1} : \Lambda_{\sigma * f}^{(2)} = \Lambda_{f}^{(2)} \right\rbrace = \left\langle \sigma_{0} \right\rangle \, \text{.}
\end{equation}
Choose a connected open neighborhood $U_{0} \subseteq U_{1} \cap U_{2}$ of $f_{0}$ in $\Poly_{d}^{\mc}(\mathbb{C})$ such that $V_{0} = \boldsymbol{\rho}_{1}\left( U_{0} \right) \subseteq \mathbb{C}^{d -1}$ is invariant under the natural action of $\mathfrak{S}_{d -1}$, and define the holomorphic map \[ \boldsymbol{\rho} = \boldsymbol{\rho}_{2} \circ \boldsymbol{\rho}_{1}^{-1} \colon V_{0} \rightarrow \mathbb{C}^{\frac{d (d -1)}{2}} \, \text{.} \] By~\eqref{equation:permSpec2}, for every $f \in U_{0}$ and every $\sigma \in \mathfrak{S}_{d -1}$, \[ \begin{split} \Lambda_{\sigma * f}^{(2)} = \Lambda_{f}^{(2)} & \Longleftrightarrow \exists \tau \in \mathfrak{S}_{\frac{d (d -1)}{2}}, \, \boldsymbol{\rho}_{2}(\sigma * f) = \tau \cdot \boldsymbol{\rho}_{2}(f)\\ & \Longleftrightarrow \exists \tau \in \mathfrak{S}_{\frac{d (d -1)}{2}}, \, \boldsymbol{\rho}\left( \sigma \cdot \boldsymbol{\rho}_{1}(f) \right) = \tau \cdot \boldsymbol{\rho}\left( \boldsymbol{\rho}_{1}(f) \right) \, \text{.} \end{split} \] Thus, to prove Lemma~\ref{lemma:unique}, it suffices to show that there is a nonempty open subset $V \subseteq V_{0}$ of $\mathbb{C}^{d -1}$ such that
\begin{equation}
\label{equation:goal2}
\forall \boldsymbol{\lambda} \in V, \, \left\lbrace \sigma \in \mathfrak{S}_{d -1} : \exists \tau \in \mathfrak{S}_{\frac{d (d -1)}{2}}, \, \boldsymbol{\rho}(\sigma \cdot \boldsymbol{\lambda}) = \tau \cdot \boldsymbol{\rho}(\boldsymbol{\lambda}) \right\rbrace = \left\langle \sigma_{0} \right\rangle \, \text{,}
\end{equation}
since $U = \boldsymbol{\rho}_{1}^{-1}(V \setminus \Delta)$ would then satisfy~\eqref{equation:goal1}.

By Claims~\ref{claim:permConj1} and~\ref{claim:permConj2}, for every $f \in U_{0}$ and every $k \in \lbrace 0, \dotsc, d -2 \rbrace$, we have \[ \boldsymbol{\rho}\left( \sigma_{0}^{k} \cdot \boldsymbol{\rho}_{1}(f) \right) = \boldsymbol{\rho}_{2}\left( \sigma_{0}^{k} * f \right) = \boldsymbol{\rho}_{2}\left( \alpha^{k} \centerdot f \right) = \tau_{0}^{k} \cdot \boldsymbol{\rho}_{2}(f) = \tau_{0}^{k} \cdot \boldsymbol{\rho}\left( \boldsymbol{\rho}_{1}(f) \right) \, \text{.} \] Therefore, we have
\begin{equation}
\label{equation:equivar}
\forall \boldsymbol{\lambda} \in V_{0}, \, \forall k \in \lbrace 0, \dotsc, d -2 \rbrace, \, \boldsymbol{\rho}\left( \sigma_{0}^{k} \cdot \boldsymbol{\lambda} \right) = \tau_{0}^{k} \cdot \boldsymbol{\rho}(\boldsymbol{\lambda}) \, \text{.}
\end{equation}
Thus, in view of~\eqref{equation:goal2}, it is enough to prove that there is a nonempty open subset $V \subseteq V_{0}$ of $\mathbb{C}^{d -1}$ such that $\boldsymbol{\rho}(\sigma \cdot \boldsymbol{\lambda}) \neq \tau \cdot \boldsymbol{\rho}(\boldsymbol{\lambda})$ for all $\boldsymbol{\lambda} \in V$, all $\sigma \in \mathfrak{S}_{d -1} \setminus \left\langle \sigma_{0} \right\rangle$ and all $\tau \in \mathfrak{S}_{\frac{d (d -1)}{2}}$.

Now, identify elements of $\mathbb{C}^{d -1}$ and $\mathbb{C}^{\frac{d (d -1)}{2}}$ with column vectors. Define \[ A = A_{2} A_{1}^{-1} \in M_{\frac{d (d -1)}{2} \times (d -1)}(\mathbb{C}) \] to be the Jacobian matrix of $\boldsymbol{\rho}$ at $(d, \dotsc, d) \in \mathbb{C}^{d -1}$, so that \[ \boldsymbol{\rho}\left( (d, \dotsc, d) +h \right) = \left( d^{2}, \dotsc, d^{2} \right) +A h +o(h) \quad \text{as} \quad h \rightarrow 0 \, \text{.} \] Denote by $\times_{\C}$ the action of $\mathfrak{S}_{d -1}$ on $M_{\frac{d (d -1)}{2} \times (d -1)}(\mathbb{C})$ that permutes the columns of matrices, which is given by \[ \sigma \times_{\C} \left( \begin{array}{@{}c|c|c@{}} C_{0} & \dotso & C_{d -2} \end{array} \right) = \left( \begin{array}{@{}c|c|c@{}} C_{\sigma^{-1}(0)} & \dotso & C_{\sigma^{-1}(d -2)} \end{array} \right) \, \text{.} \] Then $M (\sigma \cdot h) = \left( \sigma^{-1} \times_{\C} M \right) h$ for all $\sigma \in \mathfrak{S}_{d -1}$, all $M \in M_{\frac{d (d -1)}{2} \times (d -1)}(\mathbb{C})$ and all $h \in \mathbb{C}^{d -1}$. Therefore, for every $\sigma \in \mathfrak{S}_{d -1}$, we have \[ \boldsymbol{\rho}\left( (d, \dotsc, d) +\sigma \cdot h \right) = \left( d^{2}, \dotsc, d^{2} \right) +\left( \sigma^{-1} \times_{\C} A \right) h +o(h) \quad \text{as} \quad h \rightarrow 0 \, \text{,} \] and hence the Jacobian matrix of $\boldsymbol{\lambda} \mapsto \boldsymbol{\rho}(\sigma \cdot \boldsymbol{\lambda})$ at $(d, \dotsc, d)$ equals $\sigma^{-1} \times_{\C} A$. Also denote by $\times_{\R}$ the action of $\mathfrak{S}_{\frac{d (d -1)}{2}}$ on $M_{\frac{d (d -1)}{2} \times (d -1)}(\mathbb{C})$ that permutes the rows of matrices, which is given by \[ \tau \times_{\R} \left( \begin{array}{@{}c@{}} R_{0}\\ \hline \vdots\\ \hline R_{\frac{d (d -1)}{2} -1} \end{array} \right) = \left( \begin{array}{@{}c@{}} R_{\tau^{-1}(0)}\\ \hline \vdots\\ \hline R_{\tau^{-1}\left( \frac{d (d -1)}{2} -1 \right)} \end{array} \right) \, \text{.} \] Then $\tau \cdot (M h) = \left( \tau \times_{\R} M \right) h$ for all $\tau \in \mathfrak{S}_{\frac{d (d -1)}{2}}$, all $M \in M_{\frac{d (d -1)}{2} \times (d -1)}(\mathbb{C})$ and all $h \in \mathbb{C}^{d -1}$. Therefore, for every $\tau \in \mathfrak{S}_{\frac{d (d -1)}{2}}$, we have \[ \tau \cdot \boldsymbol{\rho}\left( (d, \dotsc, d) +h \right) = \left( d^{2}, \dotsc, d^{2} \right) +\left( \tau \times_{\R} A \right) h +o(h) \quad \text{as} \quad h \rightarrow 0 \, \text{,} \] and hence the Jacobian matrix of $\boldsymbol{\lambda} \mapsto \tau \cdot \boldsymbol{\rho}(\boldsymbol{\lambda})$ at $(d, \dotsc, d)$ equals $\tau \times_{\R} A$.

Thus, Lemma~\ref{lemma:unique} follows easily from the result below.

\begin{lemma}
\label{lemma:permMatrix}
We have \[ \left\lbrace \sigma \in \mathfrak{S}_{d -1} : \exists \tau \in \mathfrak{S}_{\frac{d (d -1)}{2}}, \, \sigma \times_{\C} A = \tau \times_{\R} A \right\rbrace = \left\langle \sigma_{0} \right\rangle \, \text{.} \]
\end{lemma}

Let us postpone the proof of Lemma~\ref{lemma:permMatrix} and finish our proof of Lemma~\ref{lemma:unique} first.

\begin{proof}[Proof of Lemma~\ref{lemma:unique}]
For each $\sigma \in \mathfrak{S}_{d -1} \setminus \left\langle \sigma_{0} \right\rangle$ and each $\tau \in \mathfrak{S}_{\frac{d (d -1)}{2}}$, the Jacobian matrices of $\boldsymbol{\lambda} \mapsto \boldsymbol{\rho}(\sigma \cdot \boldsymbol{\lambda})$ and $\boldsymbol{\lambda} \mapsto \tau \cdot \boldsymbol{\rho}(\boldsymbol{\lambda})$ at $(d, \dotsc, d)$ equal $\sigma^{-1} \times_{\C} A$ and $\tau \times_{\R} A$, respectively, and these two matrices are different by Lemma~\ref{lemma:permMatrix}. In particular, for every $\sigma \in \mathfrak{S}_{d -1} \setminus \left\langle \sigma_{0} \right\rangle$ and every $\tau \in \mathfrak{S}_{\frac{d (d -1)}{2}}$, the holomorphic maps $\boldsymbol{\lambda} \mapsto \boldsymbol{\rho}(\sigma \cdot \boldsymbol{\lambda})$ and $\boldsymbol{\lambda} \mapsto \tau \cdot \boldsymbol{\rho}(\boldsymbol{\lambda})$ are different, and therefore $\left\lbrace \boldsymbol{\lambda} \in V_{0} : \boldsymbol{\rho}(\sigma \cdot \boldsymbol{\lambda}) \neq \tau \cdot \boldsymbol{\rho}(\boldsymbol{\lambda}) \right\rbrace$ forms a dense open subset of $V_{0}$. As a result, \[ V = \bigcap_{\sigma \in \mathfrak{S}_{d -1} \setminus \left\langle \sigma_{0} \right\rangle} \bigcap_{\tau \in \mathfrak{S}_{\frac{d (d -1)}{2}}} \left\lbrace \boldsymbol{\lambda} \in V_{0} : \boldsymbol{\rho}(\sigma \cdot \boldsymbol{\lambda}) \neq \tau \cdot \boldsymbol{\rho}(\boldsymbol{\lambda}) \right\rbrace \] is a nonempty open subset of $\mathbb{C}^{d -1}$, which is contained in $V_{0}$. Set $U = \boldsymbol{\rho}_{1}^{-1}(V \setminus \Delta)$. Then $U$ is a nonempty open subset of $\Poly_{d}^{\mc}(\mathbb{C})$ and we have \[ \left\lbrace g \in \Poly_{d}^{\mc}(\mathbb{C}) : \Lambda_{g}^{(1)} = \Lambda_{f}^{(1)} \text{ and } \Lambda_{g}^{(2)} = \Lambda_{f}^{(2)} \right\rbrace = \left\lbrace \alpha^{k} \centerdot f : k \in \lbrace 0, \dotsc, d -2 \rbrace \right\rbrace \] for all $f \in U$ by the previous discussion. Thus, the lemma is proved.
\end{proof}

To complete our proof of Lemma~\ref{lemma:unique}, it remains to prove Lemma~\ref{lemma:permMatrix}. To do this, we shall first show the result below. From now on, write \[ A = \left( A_{j, k} \right)_{\substack{0 \leq j \leq \frac{d (d -1)}{2} -1\\ 0 \leq k \leq d -2}} \, \text{.} \]

\begin{claim}
\label{claim:distinct}
The entries $A_{j, 0}$, with $j \in \left\lbrace 0, \dotsc, \frac{d (d -1)}{2} -1 \right\rbrace$, of the first column of $A$ are pairwise distinct.
\end{claim}

\begin{proof}
By Claim~\ref{claim:inverse} and~\eqref{equation:matrix}, we have \[ A_{j, k} = d \left( \frac{P_{k}\left( w_{j} \right)}{w_{j}^{d}} +\frac{P_{k}\left( w_{j}^{d} \right)}{w_{j}^{d^{2}}} \right) \] for all $j \in \left\lbrace 0, \dotsc, \frac{d (d -1)}{2} -1 \right\rbrace$ and all $k \in \lbrace 0, \dotsc, d -2 \rbrace$. In particular, we have \[ A_{j, 0} = d \left( \frac{P_{0}\left( w_{j} \right)}{w_{j}^{d}} +\frac{P_{0}\left( w_{j}^{d} \right)}{w_{j}^{d^{2}}} \right) = \left( \frac{d}{d -1} \right) F\left( w_{j} \right) \] for all $j \in \left\lbrace 0, \dotsc, \frac{d (d -1)}{2} -1 \right\rbrace$, where \[ F(T) = (d -1) \left( \frac{P_{0}(T)}{T^{d}} +\frac{P_{0}\left( T^{d} \right)}{T^{d^{2}}} \right) \in \mathbb{C}(T) \, \text{.} \] Recall that the points $w_{j}$, with $j \in \left\lbrace 0, \dotsc, \frac{d (d -1)}{2} -1 \right\rbrace$, are representatives for the cycles for $f_{0}$ with period $2$. Moreover, the periodic points for $f_{0}$ with period $2$ are given by $\beta^{j}$, with $j \in \mathbb{Z} \setminus (d +1) \mathbb{Z}$, where $\beta = \exp\left( \frac{2 \pi i}{d^{2} -1} \right)$ as before. Therefore, it suffices to show that \[ \forall j, k \in \mathbb{Z} \setminus (d +1) \mathbb{Z}, \, F\left( \beta^{j} \right) = F\left( \beta^{k} \right) \Longrightarrow k \equiv j \text{ or } j d \pmod{d^{2} -1} \, \text{.} \] Now, note that $P_{0}(T) = \frac{1}{d -1} \left( \frac{T^{d -1} -1}{T -1} \right)$, which yields \[ F(T) = \frac{T^{d -1} -1}{T^{d} (T -1)} +\frac{T^{d (d -1)} -1}{T^{d^{2}} \left( T^{d} -1 \right)} \, \text{.} \] Therefore, for every $j \in \mathbb{Z} \setminus (d +1) \mathbb{Z}$, we have \[ F\left( \beta^{j} \right) = \frac{\beta^{j (d -1)} -1}{\beta^{j d} \left( \beta^{j} -1 \right)} +\frac{1 -\beta^{j (d -1)}}{\beta^{j d} \left( \beta^{j d} -1 \right)} = \frac{\left( \beta^{j (d -1)} -1 \right)^{2}}{\beta^{j (d -1)} \left( \beta^{j} -1 \right) \left( \beta^{j d} -1 \right)} \] since $\beta^{d^{2}} = \beta$, and hence \[ F\left( \beta^{j} \right) = \frac{\left( \beta^{\frac{j (d -1)}{2}} -\beta^{\frac{-j (d -1)}{2}} \right)^{2}}{\beta^{\frac{j (d +1)}{2}} \left( \beta^{\frac{j}{2}} -\beta^{\frac{-j}{2}} \right) \left( \beta^{\frac{j d}{2}} -\beta^{\frac{-j d}{2}} \right)} = \frac{\sin\left( \frac{j \pi}{d +1} \right)^{2} \exp\left( \frac{-j \pi i}{d -1} \right)}{\sin\left( \frac{j \pi}{d^{2} -1} \right) \sin\left( \frac{j d \pi}{d^{2} -1} \right)} \, \text{.} \] Now, suppose that $j, k \in \mathbb{Z} \setminus (d +1) \mathbb{Z}$ are such that $F\left( \beta^{j} \right) = F\left( \beta^{k} \right)$. Then there exists $\ell \in \mathbb{Z}$ such that \[ \frac{-j \pi}{d -1} = \frac{-k \pi}{d -1} +\ell \pi \quad \text{and} \quad \frac{\sin\left( \frac{j \pi}{d +1} \right)^{2}}{\sin\left( \frac{j \pi}{d^{2} -1} \right) \sin\left( \frac{j d \pi}{d^{2} -1} \right)} = \frac{(-1)^{\ell} \sin\left( \frac{k \pi}{d +1} \right)^{2}}{\sin\left( \frac{k \pi}{d^{2} -1} \right) \sin\left( \frac{k d \pi}{d^{2} -1} \right)} \, \text{.} \] Using basic trigonometric identities, it follows that \[ k = j +\ell (d -1) \quad \text{and} \quad \frac{1 -\cos\left( \frac{j \pi}{d +1} \right)^{2}}{\cos\left( \frac{j \pi}{d +1} \right) -\cos\left( \frac{j \pi}{d -1} \right)} = \frac{1 -\cos\left( \frac{(j -2 \ell) \pi}{d +1} \right)^{2}}{\cos\left( \frac{(j -2 \ell) \pi}{d +1} \right) -\cos\left( \frac{j \pi}{d -1} \right)} \, \text{.} \] Now, setting $a = \cos\left( \frac{j \pi}{d -1} \right)$, note that the function $\varphi \colon (-1, 1) \setminus \lbrace a \rbrace \rightarrow \mathbb{R}$ given by $\varphi(x) = \frac{1 -x^{2}}{x -a}$ is injective because $a \in [-1, 1]$. Therefore, $\cos\left( \frac{j \pi}{d +1} \right) = \cos\left( \frac{(j -2 \ell) \pi}{d +1} \right)$, which yields $\frac{(j -2 \ell) \pi}{d +1} \equiv \pm \frac{j \pi}{d +1} \pmod{2 \pi}$, and hence $\ell \equiv 0 \text{ or } j \pmod{d +1}$. Thus, $k \equiv j \text{ or } j d \pmod{d^{2} -1}$, and the claim is proved.
\end{proof}

We shall now prove Lemma~\ref{lemma:permMatrix}.

\begin{proof}[Proof of Lemma~\ref{lemma:permMatrix}]
For each $k \in \lbrace 0, \dotsc, d -2 \rbrace$, the two maps $\boldsymbol{\lambda} \mapsto \boldsymbol{\rho}\left( \sigma_{0}^{k} \cdot \boldsymbol{\lambda} \right)$ and $\boldsymbol{\lambda} \mapsto \tau_{0}^{k} \cdot \boldsymbol{\rho}(\boldsymbol{\lambda})$ coincide according to~\eqref{equation:equivar}, and hence they have the same Jacobian matrix at $(d, \dotsc, d)$. Thus, by the previous discussion,
\begin{equation}
\label{equation:circulant}
\forall k \in \lbrace 0, \dotsc, d -2 \rbrace, \, \sigma_{0}^{-k} \times_{\C} A = \tau_{0}^{k} \times_{\R} A \, \text{.}
\end{equation}
Now, suppose that $\sigma \in \mathfrak{S}_{d -1}$ and $\tau \in \mathfrak{S}_{\frac{d (d -1)}{2}}$ satisfy $\sigma \times_{\C} A = \tau \times_{\R} A$. Write \[ \sigma \times_{\C} A = \left( M_{j, k} \right)_{\substack{0 \leq j \leq \frac{d (d -1)}{2} -1\\ 0 \leq k \leq d -2}} = \tau \times_{\R} A \, \text{.} \] Set $\ell = \sigma(0)$, and let us prove that $\sigma = \sigma_{0}^{\ell}$. By~\eqref{equation:circulant}, we have \[ A_{0, 0} = M_{0, \ell} = A_{\tau^{-1}(0), \ell} = A_{\tau^{-1}(0), \sigma_{0}^{\ell}(0)} = A_{\tau_{0}^{-\ell} \tau^{-1}(0), 0} \, \text{,} \] which yields $\tau_{0}^{-\ell} \tau^{-1}(0) = 0$ by Claim~\ref{claim:distinct}, and hence $\tau^{-1}(0) = \ell$ according to~\eqref{equation:permCycle}. As a result, for every $k \in \lbrace 0, \dotsc, d -2 \rbrace$, we have \[ A_{0, \sigma^{-1}(k)} = M_{0, k} = A_{\ell, k} = A_{\tau_{0}^{\ell}(0), k} = A_{0, \sigma_{0}^{-\ell}(k)} \] by~\eqref{equation:permCycle} and~\eqref{equation:circulant}. Moreover, $A_{0, k} = A_{\tau_{0}^{-k}(0), 0}$ for all $k \in \lbrace 0, \dotsc, d -2 \rbrace$ by~\eqref{equation:circulant}, and hence $A_{0, 0}, \dotsc, A_{0, d -2}$ are pairwise distinct by~\eqref{equation:permCycle} and Claim~\ref{claim:distinct}. Therefore, $\sigma^{-1}(k) = \sigma_{0}^{-\ell}(k)$ for all $k \in \lbrace 0, \dotsc, d -2 \rbrace$, and hence $\sigma = \sigma_{0}^{\ell}$. Thus, the lemma is proved.
\end{proof}

Thus, our proof of Lemma~\ref{lemma:unique} is complete. To conclude, we simply observe that Theorem~\ref{theorem:unique} follows immediately. For completeness, we include details.

\begin{proof}[Proof of Theorem~\ref{theorem:unique}]
For $P \geq 1$, denote again by $\Sigma_{d}^{(P)}$ the scheme-theoretic image of $\Mult_{d}^{(P)}$. Then we have \[ \dim\left( \Sigma_{d}^{(2)} \right) \geq \dim\left( \Sigma_{d}^{(1)} \right) = d -1 = \dim\left( \mathcal{P}_{d} \right) \] because $\Sigma_{d}^{(1)}$ is the image of $\Sigma_{d}^{(2)}$ under the projection $p_{1} \colon \mathbb{A}^{d} \times \mathbb{A}^{\frac{d (d -1)}{2}} \rightarrow \mathbb{A}^{d}$ onto the first factor. Consequently, the induced morphism $\Mult_{d}^{(2)} \colon \mathcal{P}_{d} \rightarrow \Sigma_{d}^{(2)}$ has some finite degree $D \geq 1$. There is a nonempty Zariski-open subset $W$ of $\Sigma_{d}^{(2)}$ such that every $\boldsymbol{\sigma} \in W(\mathbb{C})$ has exactly $D$ preimages under $\Mult_{d}^{(2)} \colon \mathcal{P}_{d}(\mathbb{C}) \rightarrow \Sigma_{d}^{(2)}(\mathbb{C})$. Now, consider the preimage $V$ of $W$ under $\Mult_{d}^{(2)}$. Then $V$ is a nonempty Zariski-open subset of $\mathcal{P}_{d}$. For every $[f] \in V(\mathbb{C})$, there are exactly $D$ elements $[g] \in \mathcal{P}_{d}(\mathbb{C})$ such that $\Lambda_{g}^{(1)} = \Lambda_{f}^{(1)}$ and $\Lambda_{g}^{(2)} = \Lambda_{f}^{(2)}$. By Lemma~\ref{lemma:unique}, there is a nonempty open subset $U$ of $\Poly_{d}^{\mc}(\mathbb{C})$ such that, for every $f \in U$, the elements $g \in \Poly_{d}^{\mc}(\mathbb{C})$ such that $\Lambda_{g}^{(1)} = \Lambda_{f}^{(1)}$ and $\Lambda_{g}^{(2)} = \Lambda_{f}^{(2)}$ are precisely the $\alpha^{k} \centerdot f$, with $k \in \lbrace 0, \dotsc, d -2 \rbrace$. Now, consider the image $V^{\prime}$ of $U$ under the holomorphic map $\pi_{d}^{\mc} \colon \Poly_{d}^{\mc}(\mathbb{C}) \rightarrow \mathcal{P}_{d}(\mathbb{C})$, via the biholomorphism $\mathcal{P}_{d}(\mathbb{C}) \cong \mathcal{P}_{d}^{\mc}(\mathbb{C})$. Then $V^{\prime}$ is a nonempty open subset of $\mathcal{P}_{d}(\mathbb{C})$. Furthermore, each $[f] \in V^{\prime}$ is the unique $[g] \in \mathcal{P}_{d}(\mathbb{C})$ such that $\Lambda_{g}^{(1)} = \Lambda_{f}^{(1)}$ and $\Lambda_{g}^{(2)} = \Lambda_{f}^{(2)}$. Finally, note that $V(\mathbb{C}) \cap V^{\prime} \neq \varnothing$, as $V(\mathbb{C})$ is dense in $\mathcal{P}_{d}(\mathbb{C})$ for the complex topology. Thus, we have $D = 1$, and the theorem is proved.
\end{proof}

\appendix

\section{Additional estimates on multipliers of polynomial maps}
\label{appendix:bounds}

In this section, we again fix an integer $d \geq 2$.

\subsection{Upper bounds on absolute values of multipliers}

As mentioned in the introduction, it is not difficult to obtain upper bounds on the characteristic exponents of polynomial maps at periodic points in terms of the maximal escape rate. For completeness, let us give here some details.

Given a valued field $K$, we denote by $\lvert . \rvert_{K}$ its absolute value and by $\lVert . \rVert_{K^{d -1}}$ the norm on $K^{d -1}$ given by \[ \lVert \boldsymbol{c} \rVert_{K^{d -1}} = \max_{j \in \lbrace 1, \dotsc, d -1 \rbrace} \left\lvert c_{j} \right\rvert_{K} \quad \text{for} \quad \boldsymbol{c} = \left( c_{1}, \dotsc, c_{d -1} \right) \in K^{d -1} \, \text{.} \]

We shall also work again with the normal form introduced by Ingram in~\cite{I2012}. Given a field $K$ of characteristic $0$ and $\boldsymbol{c} = \left( c_{1}, \dotsc, c_{d -1} \right) \in K^{d -1}$, we define \[ f_{\boldsymbol{c}}(z) = \frac{1}{d} z^{d} +\sum_{j = 1}^{d -1} \frac{(-1)^{j} \tau_{j}(\boldsymbol{c})}{d -j} z^{d -j} \in \Poly_{d}(K) \, \text{,} \] where $\tau_{1}(\boldsymbol{c}), \dotsc, \tau_{d -1}(\boldsymbol{c})$ are the elementary symmetric functions of $c_{1}, \dotsc, c_{d -1}$.

For our purposes, we shall use the generalization of Claims~\ref{claim:estimates2} and~\ref{claim:maxEscape} below. As it can be derived in a similar way, by using the triangle inequality and elimination theory, we omit the proof.

\begin{claim}
\label{claim:estimates3}
Assume that $K$ is an algebraically closed valued field of characteristic $0$. Then there exist $\delta_{K} \in \mathbb{R}_{> 0}$ and $\Delta_{K} \in \mathbb{R}$ both depending only on the restriction of $\lvert . \rvert_{K}$ to $\mathbb{Q}$ such that $g_{f_{\boldsymbol{c}}}(z) \geq \log^{+}\lvert z \rvert_{K} +\Delta_{K}$ for all $\boldsymbol{c} \in K^{d -1}$ and all $z \in K$ such that $\lvert z \rvert_{K} > \delta_{K} \lVert \boldsymbol{c} \rVert_{K^{d -1}}$. In addition, there exists some $E_{K} \in \mathbb{R}$ depending only on the restriction of $\lvert . \rvert_{K}$ to $\mathbb{Q}$ such that $M_{f_{\boldsymbol{c}}} \geq \log^{+}\lVert \boldsymbol{c} \rVert_{K^{d -1}} +E_{K}$ for all $\boldsymbol{c} \in K^{d -1}$. Furthermore, we can take $\delta_{K} = 1$, $\Delta_{K} = 0$ and $E_{K} = 0$ if $K$ is non-Archimedean with residue characteristic $0$ or greater than $d$.
\end{claim}

Finally, we obtain the following result:

\begin{proposition}
\label{proposition:ineqMax}
Assume that $K$ is an algebraically closed valued field of characteristic $0$. Then there exists $B_{K} \in \mathbb{R}$ depending only on the restriction of $\lvert . \rvert_{K}$ to $\mathbb{Q}$ such that $M_{f}^{(p)} \leq (d -1) M_{f} +B_{K}$ for all $f \in \Poly_{d}(K)$ and all $p \geq 1$. Moreover, we can take $B_{K} = 0$ if $K$ is non-Archimedean with residue characteristic either $0$ or greater than $d$.
\end{proposition}

\begin{proof}
By conjugation, we may restrict our attention to the polynomials $f_{\boldsymbol{c}}$, with $\boldsymbol{c} \in K^{d -1}$. Define \[ \varepsilon_{K} = \begin{cases} 1 +\max\left\lbrace \delta_{K}, \exp\left( -\Delta_{K} \right) \right\rbrace & \text{if } K \text{ is Archimedean}\\ \max\left\lbrace 1, \delta_{K}, \exp\left( -\Delta_{K} \right) \right\rbrace & \text{if } K \text{ is non-Archimedean} \end{cases} \] and \[ B_{K} = (d -1) \left( \log\left( \varepsilon_{K} \right) -E_{K} \right) \in \mathbb{R} \, \text{,} \] with $\delta_{K} \in \mathbb{R}_{> 0}$, $\Delta_{K} \in \mathbb{R}$ and $E_{K} \in \mathbb{R}$ as in Claim~\ref{claim:estimates3}. Note that $B_{K}$ depends only on the restriction of $\lvert . \rvert_{K}$ to $\mathbb{Q}$ and we have $B_{K} = 0$ if $K$ is non-Archimedean with residue characteristic $0$ or greater than $d$. Now, suppose that $\boldsymbol{c} \in K^{d -1}$, $p \geq 1$ and $z_{0} \in K$ is a fixed point for $f_{\boldsymbol{c}}^{\circ p}$. For every $j \in \lbrace 0, \dotsc, p -1 \rbrace$, as $g_{f_{\boldsymbol{c}}}\left( f_{\boldsymbol{c}}^{\circ j}\left( z_{0} \right) \right) = 0$, we have \[ \begin{split} \left\lvert f_{\boldsymbol{c}}^{\circ j}\left( z_{0} \right) \right\rvert_{K} & \leq \max\left\lbrace \delta_{K} \lVert \boldsymbol{c} \rVert_{K^{d -1}}, \exp\left( -\Delta_{K} \right) \right\rbrace\\ & \leq \max\left\lbrace \delta_{K}, \exp\left( -\Delta_{K} \right) \right\rbrace \cdot \max\left\lbrace 1, \lVert \boldsymbol{c} \rVert_{K^{d -1}} \right\rbrace \end{split} \] by Claim~\ref{claim:estimates3}. Therefore, for every $j \in \lbrace 0, \dotsc, p -1 \rbrace$, we have \[ \left\lvert f_{\boldsymbol{c}}^{\prime}\left( f_{\boldsymbol{c}}^{\circ j}\left( z_{0} \right) \right) \right\rvert_{K} = \prod_{\ell = 1}^{d -1} \left\lvert f_{\boldsymbol{c}}^{\circ j}\left( z_{0} \right) -c_{\ell} \right\rvert_{K} \leq \varepsilon_{K}^{d -1} \max\left\lbrace 1, \lVert \boldsymbol{c} \rVert_{K^{d -1}} \right\rbrace^{d -1} \] by the triangle inequality. As a result, we have \[ \begin{split} \frac{1}{p} \log\left\lvert \left( f_{\boldsymbol{c}}^{\circ p} \right)^{\prime}\left( z_{0} \right) \right\rvert_{K} & = \frac{1}{p} \sum_{j = 0}^{p -1} \log\left\lvert f_{\boldsymbol{c}}^{\prime}\left( f_{\boldsymbol{c}}^{\circ j}\left( z_{0} \right) \right) \right\rvert_{K}\\ & \leq (d -1) \log^{+}\lVert \boldsymbol{c} \rVert_{K^{d -1}} +(d -1) \log\left( \varepsilon_{K} \right)\\ & \leq (d -1) M_{f_{\boldsymbol{c}}} +B_{K} \end{split} \] by Claim~\ref{claim:estimates3}. This completes the proof of the proposition.
\end{proof}

\begin{remark}
In~\cite{Bu2003}, Buff used de Branges's theorem to show that we can take $B_{\mathbb{C}} = 2 \log(d)$ in the statement of Proposition~\ref{proposition:ineqMax}.
\end{remark}

As an immediate consequence of Proposition~\ref{proposition:ineqMax}, we obtain an upper bound on the heights of multipliers of polynomial maps in terms of the critical height.

\begin{corollary}
There exists some $B \in \mathbb{R}$ such that $H_{f}^{(p)} \leq (d -1) H_{f} +B$ for all $f \in \Poly_{d}\left( \overline{\mathbb{Q}} \right)$ and all $p \geq 1$.
\end{corollary}

\subsection{Lower bounds on absolute values of multipliers}

Finally, let us obtain here a lower bound on the absolute values of multipliers of polynomial maps whose critical points all escape.

Suppose that $f \in \Poly_{d}(K)$, with $K$ an algebraically closed valued field of characteristic $0$. We define the \emph{minimal escape rate} $m_{f}$ of $f$ by \[ m_{f} = \min\left\lbrace g_{f}(c) : c \in K, \, f^{\prime}(c) = 0 \right\rbrace \, \text{.} \] Also, for $p \geq 1$, we define \[ m_{f}^{(p)} = \min_{\lambda \in \Lambda_{f}^{(p)}} \left( \frac{1}{p} \log^{+}\lvert \lambda \rvert \right) \, \text{,} \] where $\lvert . \rvert$ denotes the absolute value on $K$.

As a consequence of Lemma~\ref{lemma:ineqArch} in the Archimedean case and Lemma~\ref{lemma:ineqNonArch} in the non-Archimedean case, we obtain the result below. In the complex setting, it is a slightly weaker version of~\cite[Theorem~1.6]{EL1992}.

\begin{proposition}
\label{proposition:ineqMin}
Assume that $K$ is an algebraically closed valued field of characteristic $0$ that is either Archimedean or non-Archimedean with residue characteristic $0$ or greater than $d$ and $f \in \Poly_{d}(K)$. Then $m_{f}^{(p)} \geq (d -1) m_{f}$ for all $p \geq 1$.
\end{proposition}

\begin{proof}
First, assume that $K$ is endowed with an Archimedean absolute value $\lvert . \rvert_{\infty}$. Then, by Ostrowski's theorem, there exist an embedding $\sigma \colon K \hookrightarrow \mathbb{C}$ and $s \in (0, 1]$ such that $\lvert z \rvert_{\infty} = \left\lvert \sigma(z) \right\rvert^{s}$ for each $z \in K$, where $\lvert . \rvert$ is the usual absolute value on $\mathbb{C}$. We have $m_{f} = s \cdot m_{\sigma(f)}$ and $m_{f}^{(p)} = s \cdot m_{\sigma(f)}^{(p)}$ for each $p \geq 1$. Thus, replacing $f$ by $\sigma(f)$ if necessary, we may assume that $f \in \Poly_{d}(\mathbb{C})$. Note that the desired result is immediate if $m_{f} = 0$. Now, suppose that $m_{f} > 0$. Choose an integer $k \geq 0$ such that $d^{k} m_{f} \geq M_{f}$. Suppose that $z_{0} \in \mathbb{C}$ is a periodic point for $f$ with period $p \geq 1$. For $j \in \lbrace 0, \dotsc, p -1 \rbrace$, denote by $U_{j}$ and $V_{j}$ the respective connected components of $\left\lbrace g_{f} < m_{f} \right\rbrace$ and $\left\lbrace g_{f} < d \cdot m_{f} \right\rbrace$ containing $f^{\circ j}\left( z_{0} \right)$. Then $U_{j}$ contains no critical point for $f$ for all $j \in \lbrace 0, \dotsc, p -1 \rbrace$. By the Riemann--Hurwitz formula, it follows that $f$ induces a biholomorphism from $U_{j}$ to $V_{j +1 \pmod{p}}$ for all $j \in \lbrace 0, \dotsc, p -1 \rbrace$. Therefore, by Lemma~\ref{lemma:ineqArch}, we have \[ \frac{1}{p} \log\left\lvert \left( f^{\circ p} \right)^{\prime}\left( z_{0} \right) \right\rvert \geq \frac{d -1}{p} \left( \sum_{j = 0}^{p -1} \frac{1}{d_{j}} \right) d^{k} m_{f} \geq (d -1) m_{f} \, \text{,} \] where $d_{j} \leq d^{k}$ is the degree of $f^{\circ k} \colon V_{j} \rightarrow \left\lbrace g_{f} < d^{k +1} m_{f} \right\rbrace$ for all $j \in \lbrace 0, \dotsc, p -1 \rbrace$. This completes the proof of the proposition in the Archimedean case.

Now, assume that $K$ is endowed with a non-Archimedean absolute value $\lvert . \rvert$ and the associated residue characteristic either equals $0$ or is greater than $d$. Note that the desired inequality is immediate if the absolute value $\lvert . \rvert$ is trivial or if $m_{f} = 0$. From now on, suppose that $\lvert . \rvert$ is not trivial and $m_{f} > 0$. Choose an integer $k \geq 0$ such that $d^{k} m_{f} \geq M_{f}$. Suppose that $z_{0} \in K$ is a periodic point for $f$ with period $p \geq 1$. We have $g_{f}\left( z_{0} \right) = 0$. By Lemma~\ref{lemma:greenDisks1}, the sets $\left\lbrace g_{f} < m_{f} \right\rbrace$ and $\left\lbrace g_{f} < d \cdot m_{f} \right\rbrace$ are finite unions of disks. Thus, for $j \in \lbrace 0, \dotsc, p -1 \rbrace$, define $U_{j}$ and $V_{j}$ to be the respective disk components of $\left\lbrace g_{f} < m_{f} \right\rbrace$ and $\left\lbrace g_{f} < d \cdot m_{f} \right\rbrace$ that contain $f^{\circ j}\left( z_{0} \right)$. Then $U_{j}$ contains no critical point for $f$ for each $j \in \lbrace 0, \dotsc, p -1 \rbrace$. As a result, it follows from Lemmas~\ref{lemma:preimage} and~\ref{lemma:rhFormula} that $f$ induces a bijection from $U_{j}$ to $V_{j +1 \pmod{p}}$ for each $j \in \lbrace 0, \dotsc, p -1 \rbrace$. Therefore, by Lemma~\ref{lemma:ineqNonArch}, we have \[ \frac{1}{p} \log\left\lvert \left( f^{\circ p} \right)^{\prime}\left( z_{0} \right) \right\rvert \geq \frac{d -1}{p} \left( \sum_{j = 0}^{p -1} \frac{1}{d_{j}} \right) d^{k} m_{f} \geq (d -1) m_{f} \, \text{,} \] where $d_{j} \leq d^{k}$ is the degree of $f^{\circ k} \colon V_{j} \rightarrow \left\lbrace g_{f} < d^{k +1} m_{f} \right\rbrace$ for all $j \in \lbrace 0, \dotsc, p -1 \rbrace$. This completes the proof of the proposition in the non-Archimedean case.
\end{proof}

\section{About isospectral polynomial maps}
\label{appendix:isospec}

\subsection{Examples of isospectral polynomial maps of composite degrees}

As mentioned in the introduction, the multiplier spectrum morphisms are not always isomorphisms onto their images. In fact, there are nonconjugate polynomial maps of any composite degree that have the same multiset of multipliers for each period. Although it is already known, let us detail this here for the reader's convenience.

To exhibit isospectral polynomial maps, one can use the following result:

\begin{proposition}[{\cite[Lemma~2.1]{P2019b}}]
\label{proposition:isospec}
Assume that $K$ is an algebraically closed field of characteristic $0$ and $f \in \Poly_{d}(K)$ and $g \in \Poly_{e}(K)$, with $d, e \geq 2$. Then $\Lambda_{f \circ g}^{(p)} = \Lambda_{g \circ f}^{(p)}$ for all $p \geq 1$.
\end{proposition}

\begin{proof}
For each $p \geq 1$, we have $g \circ (f \circ g)^{\circ p} = (g \circ f)^{\circ p} \circ g$, and hence $g$ sends any periodic point for $f \circ g$ in $K$ with period $p$ to a periodic point for $g \circ f$ in $K$ with period dividing $p$. Similarly, $f$ sends any periodic point for $g \circ f$ in $K$ with period $p \geq 1$ to a periodic point for $f \circ g$ in $K$ with period dividing $p$. Moreover, the map $f \circ g$ induces a permutation of its set of periodic points in $K$, which preserves the periods. Therefore, $g$ induces an injection from the set of periodic points for $f \circ g$ in $K$ into the set of periodic points for $g \circ f$ in $K$, which preserves the periods. In fact, this map induced by $g$ is a bijection as $g \circ f$ also permutes its periodic points in $K$. Finally, for each periodic point $z_{0} \in K$ for $f \circ g$ with period $p \geq 1$, we have \[ \begin{split} \left( (g \circ f)^{\circ p} \right)^{\prime}\left( g\left( z_{0} \right) \right) & = \left( \prod_{j = 0}^{p -1} f^{\prime}\left( (g \circ f)^{\circ j} \circ g\left( z_{0} \right) \right) \right) \left( \prod_{j = 0}^{p -1} g^{\prime}\left( f \circ (g \circ f)^{\circ j} \circ g\left( z_{0} \right) \right) \right)\\ & = \left( \prod_{j = 0}^{p -1} f^{\prime}\left( g \circ (f \circ g)^{\circ j}\left( z_{0} \right) \right) \right) \left( \prod_{j = 0}^{p -1} g^{\prime}\left( (f \circ g)^{\circ j}\left( z_{0} \right) \right) \right)\\ & = \left( (f \circ g)^{\circ p} \right)^{\prime}\left( z_{0} \right) \, \text{.} \end{split} \] Thus, the map $g$ induces a bijection from the set of periodic points for $f \circ g$ in $K$ onto the set of periodic points for $g \circ f$ in $K$, which preserves the periods and the multipliers. This completes the proof of the proposition.
\end{proof}

As a direct consequence of Proposition~\ref{proposition:isospec}, we obtain the following:

\begin{corollary}
Suppose that $d \geq 2$ is not a prime number. Then, for each $P \geq 1$, the morphism $\Mult_{d}^{(P)}$ is not injective.
\end{corollary}

\begin{proof}
Assume that $d = d_{1} d_{2}$, with $d_{1}, d_{2} \geq 2$, and define \[ f(z) = z^{d} +1 \in \Poly_{d}(\mathbb{Q}) \quad \text{and} \quad g(z) = \left( z^{d_{1}} +1 \right)^{d_{2}} \in \Poly_{d}(\mathbb{Q}) \, \text{.} \] Then we have $f = h_{1} \circ h_{2}$ and $g = h_{2} \circ h_{1}$, where $h_{1}(z) = z^{d_{1}} +1 \in \Poly_{d_{1}}(\mathbb{Q})$ and $h_{2}(z) = z^{d_{2}} \in \Poly_{d_{2}}(\mathbb{Q})$, and therefore $\Mult_{d}^{(P)}\left( [f] \right) = \Mult_{d}^{(P)}\left( [g] \right)$ for all $P \geq 1$ by Proposition~\ref{proposition:isospec}. However, $[f] \neq [g]$ in $\mathcal{P}_{d}(\mathbb{Q})$ because $f$ has only $1$ critical point in $\mathbb{C}$ whereas $g$ has exactly $d_{1} +1$ critical points in $\mathbb{C}$. This completes the proof of the corollary.
\end{proof}

Actually, it is suspected that Proposition~\ref{proposition:isospec} is the unique source of examples of nonconjugate isospectral polynomial maps (see~\cite[Problem~3.1]{P2019a}).

\subsection{The case of quartic polynomial maps}

Finally, as isospectral polynomial maps of degree $2$ or $3$ are necessarily conjugate, let us investigate the situation for polynomial maps of degree $4$.

Using explicit expressions for the multiplier spectrum morphisms, we show here that the pairs of nonconjugate isospectral quartic polynomial maps all come from Proposition~\ref{proposition:isospec}. More precisely, we obtain the following result:

\begin{proposition}
Assume that $K$ is an algebraically closed field of characteristic $0$ and $f, g \in \Poly_{4}(K)$ satisfy $\Lambda_{f}^{(1)} = \Lambda_{g}^{(1)}$ and $\Lambda_{f}^{(2)} = \Lambda_{g}^{(2)}$. Then $[f] = [g]$ in $\mathcal{P}_{4}(K)$ or there exist $h_{1}, h_{2} \in \Poly_{2}(K)$ such that $f = h_{1} \circ h_{2}$ and $g = h_{2} \circ h_{1}$.
\end{proposition}

\begin{proof}
We shall first work with monic centered quartic polynomials. Recall that \[ \Poly_{4}^{\mc} = \left\lbrace z^{4} +a_{2} z^{2} +a_{1} z +a_{0} \right\rbrace \] and that the algebraic group $\mu_{3} = \left\lbrace \omega : \omega^{3} = 1 \right\rbrace$ acts on $\Poly_{4}^{\mc}$ by \[ \omega \centerdot \left( z^{4} +a_{2} z^{2} +a_{1} z +a_{0} \right) = z^{4} +\omega^{-1} a_{2} z^{2} +a_{1} z +\omega a_{0} \, \text{.} \] Therefore, as $\mathcal{P}_{4}^{\mc} = \Poly_{4}^{\mc}/\mu_{3}$, we have \[ \mathbb{Q}\left[ \mathcal{P}_{4}^{\mc} \right] = \mathbb{Q}\left[ \Poly_{4}^{\mc} \right]^{\mu_{3}} = \mathbb{Q}[\alpha, \beta, \gamma, \delta] \, \text{,} \] where \[ \alpha = a_{1} \, \text{,} \quad \beta = a_{0}^{3} \, \text{,} \quad \gamma = a_{2}^{3} \quad \text{and} \quad \delta = a_{0} a_{2} \, \text{.} \] Now, for $h_{1}(z) = z^{2} +c_{1} \in \Poly_{2}^{\mc}$ and $h_{2}(z) = z^{2} +c_{2} \in \Poly_{2}^{\mc}$, we have \[ h_{1} \circ h_{2}(z) = z^{4} +a_{2} z^{2} +a_{1} z +a_{0} \in \Poly_{4}^{\mc} \, \text{,} \quad \text{with} \quad \left\lbrace \begin{array}{l} a_{0} = c_{2}^{2} +c_{1}\\ a_{1} = 0\\ a_{2} = 2 c_{2} \end{array} \right. \, \text{,} \] and $\left[ h_{1} \circ h_{2} \right] = \left[ h_{2} \circ h_{1} \right]$ in $\mathcal{P}_{4}^{\mc}$ if and only if $c_{1}^{3} = c_{2}^{3}$. Thus, setting \[ \mathcal{S}_{4}^{\mc} = \left\lbrace \left[ h_{1} \circ h_{2} \right] : h_{1}, h_{2} \in \Poly_{2}^{\mc} \right\rbrace \subseteq \mathcal{P}_{4}^{\mc} \] and \[ \mathcal{L}_{4}^{\mc} = \left\lbrace \left[ h_{1} \circ h_{2} \right] : h_{1}, h_{2} \in \Poly_{2}^{\mc}, \, \left[ h_{1} \circ h_{2} \right] = \left[ h_{2} \circ h_{1} \right] \right\rbrace \subseteq \mathcal{S}_{4}^{\mc} \, \text{,} \] we have \[ \mathcal{S}_{4}^{\mc} = \lbrace \alpha = 0 \rbrace \quad \text{and} \quad \mathcal{L}_{4}^{\mc} = \lbrace \alpha = 0 \rbrace \cap \left\lbrace 64 \beta^{3} -\gamma^{2} +12 \gamma \delta -48 \delta^{2} -8 \gamma = 0 \right\rbrace \, \text{.} \] For simplicity, write $s_{j} = \sigma_{4, j}^{(1)}$ for $j \in \lbrace 1, \dotsc, 4 \rbrace$ and $t_{j} = \sigma_{4, j}^{(2)}$ for $j \in \lbrace 1, \dotsc, 6 \rbrace$, so that \[ \Mult_{4}^{(2)} = \left( s_{1}, \dotsc, s_{4}, t_{1}, \dotsc, t_{6} \right) \colon \mathcal{P}_{4}^{\mc} \rightarrow \mathbb{A}^{10} \] via the natural isomorphism $\mathcal{P}_{4} \cong \mathcal{P}_{4}^{\mc}$. Using the software SageMath, we obtain
\begin{align*}
s_{1} & = -8 \alpha +12 \, \text{,}\\
s_{2} & = 18 \alpha^{2} -60 \alpha +4 \gamma -16 \delta +48 \, \text{,}\\
\begin{split}
s_{4} & = -27 \alpha^{4} +108 \alpha^{3} -4 \alpha^{2} \gamma +144 \alpha^{2} \delta -144 \alpha^{2} +8 \alpha \gamma -288 \alpha \delta\\
& \quad +16 \gamma \delta -128 \delta^{2} +64 \alpha +256 \beta +128 \delta \, \text{,}
\end{split}\\
\begin{split}
t_{2} & = 27 \alpha^{4} +324 \alpha^{3} +4 \alpha^{2} \gamma -144 \alpha^{2} \delta +1440 \alpha^{2} +24 \alpha \gamma -864 \alpha \delta\\
& \quad -16 \gamma \delta +128 \delta^{2} +2880 \alpha -256 \beta +96 \gamma -512 \delta +3840 \, \text{.}
\end{split}
\end{align*}
Then, using elimination with the software SageMath, we obtain
\begin{itemize}
\item an expression for $\alpha$ as an element of $\mathbb{Q}\left[ s_{1} \right]$:
\begin{equation}
\label{equation:alpha}
\alpha = \frac{-1}{8} s_{1} +\frac{3}{2} \, \text{,}
\end{equation}
\item a polynomial equation in $\delta$ with coefficients in $\mathbb{Q}\left[ s_{1}, s_{2}, s_{4}, t_{2} \right]$, degree $1$ and leading coefficient a constant multiple of $s_{1} -12$:
\begin{multline}
\label{equation:delta1}
2048 \left( s_{1} -12 \right) \delta -9 s_{1}^{3} +660 s_{1}^{2} -16 s_{1} s_{2} -19952 s_{1} +576 s_{2} -16 s_{4} -16 t_{2}\\ +202944 = 0 \, \text{,}
\end{multline}
\item a polynomial equation in $\delta$ with coefficients in $\mathbb{Q}\left[ s_{1}, s_{2}, s_{4}, t_{2} \right]$, degree $2$ and constant leading coefficient:
\begin{multline}
\label{equation:delta2}
1048576 \delta^{2} +512 \left( 315 s_{1}^{2} +4 s_{2}^{2} -3624 s_{1} -352 s_{2} -12 s_{4} +4 t_{2} -9552 \right) \delta\\ -9 s_{1}^{2} s_{2}^{2} +3672 s_{1}^{2} s_{2} +81 s_{1}^{2} s_{4} -63 s_{1}^{2} t_{2} -24 s_{1} s_{2}^{2} -344240 s_{1}^{2} -124416 s_{1} s_{2}\\ +1168 s_{1} s_{4} +784 s_{1} t_{2} +4848 s_{2}^{2} -672 s_{2} s_{4} -160 s_{2} t_{2} +s_{4}^{2} +2 s_{4} t_{2} +t_{2}^{2}\\ +7144320 s_{1} +1537152 s_{2} -12432 s_{4} -11664 t_{2} -12936960 = 0 \, \text{,}
\end{multline}
\item an expression for $\beta$ as an element of $\mathbb{Q}\left[ \delta, s_{1}, s_{2}, s_{4}, t_{2} \right]$:
\begin{equation}
\label{equation:beta}
\begin{split}
\beta & = \frac{-3}{2048} \delta s_{1}^{2} +\frac{3}{65536} s_{1}^{2} s_{2} +\frac{1}{4} \delta^{2} +\frac{31}{256} \delta s_{1} -\frac{1}{64} \delta s_{2}\\
& \quad +\frac{103}{16384} s_{1}^{2} -\frac{1}{2048} s_{1} s_{2} -\frac{1}{65536} s_{1} s_{4} -\frac{1}{65536} s_{1} t_{2} -\frac{127}{128} \delta\\
& \quad -\frac{1007}{2048} s_{1} +\frac{69}{4096} s_{2} +\frac{55}{16384} s_{4} -\frac{9}{16384} t_{2} +\frac{7131}{1024} \, \text{,}
\end{split}
\end{equation}
\item an expression for $\gamma$ as an element of $\mathbb{Q}\left[ \delta, s_{1}, s_{2}, s_{4}, t_{2} \right]$:
\begin{equation}
\label{equation:gamma}
\gamma = \frac{-9}{128} s_{1}^{2} +4 \delta -\frac{3}{16} s_{1} +\frac{1}{4} s_{2} +\frac{3}{8} \, \text{.}
\end{equation}
\end{itemize}
In particular, by~\eqref{equation:alpha} and the discussion above, we have $\mathcal{S}_{4}^{\mc} = \left\lbrace s_{1} = 12 \right\rbrace$. Thus, by the relations~\eqref{equation:alpha}, \eqref{equation:delta1}, \eqref{equation:beta} and~\eqref{equation:gamma}, each element $[f] \in \mathcal{P}_{4}^{\mc}(K) \setminus \mathcal{S}_{4}^{\mc}(K)$ is the unique $[g] \in \mathcal{P}_{4}^{\mc}(K)$ such that $\Mult_{4}^{(2)}\left( [g] \right) = \Mult_{4}^{(2)}\left( [f] \right)$. Now, note that every element of $K^{10}$ has at most two preimages in $\mathcal{P}_{4}^{\mc}(K)$ under $\Mult_{4}^{(2)}$ by the relations~\eqref{equation:alpha}, \eqref{equation:delta2}, \eqref{equation:beta} and~\eqref{equation:gamma}. Moreover, for all $[f] = \left[ h_{1} \circ h_{2} \right] \in \mathcal{S}_{4}^{\mc}(K)$, with $h_{1}, h_{2} \in \Poly_{2}^{\mc}(K)$, the elements $\left[ h_{1} \circ h_{2} \right]$ and $\left[ h_{2} \circ h_{1} \right]$ of $\mathcal{P}_{4}^{\mc}(K)$ are both preimages of $\Mult_{4}^{(2)}\left( [f] \right)$ under $\Mult_{4}^{(2)}$ by Proposition~\ref{proposition:isospec}, and these elements are distinct if $[f] \in \mathcal{S}_{4}^{\mc}(K) \setminus \mathcal{L}_{4}^{\mc}(K)$. As a result, for every $[f] = \left[ h_{1} \circ h_{2} \right] \in \mathcal{S}_{4}^{\mc}(K)$, with $h_{1}, h_{2} \in \Poly_{2}^{\mc}(K)$, we have \[ \forall [g] \in \mathcal{P}_{4}^{\mc}(K), \, \Mult_{4}^{(2)}\left( [g] \right) = \Mult_{4}^{(2)}\left( [f] \right) \Longleftrightarrow [g] = \left[ h_{1} \circ h_{2} \right] \text{ or } \left[ h_{2} \circ h_{1} \right] \] since $\mathcal{L}_{4}^{\mc}$ is a proper Zariski-closed subset of $\mathcal{S}_{4}^{\mc}$.

Finally, assume that $f, g \in \Poly_{4}(K)$ satisfy $\Lambda_{f}^{(1)} = \Lambda_{g}^{(1)}$ and $\Lambda_{f}^{(2)} = \Lambda_{g}^{(2)}$. Then, using the natural isomorphism $\mathcal{P}_{4} \cong \mathcal{P}_{4}^{\mc}$, it follows from the discussion above that $[f] = [g]$ in $\mathcal{P}_{4}(K)$ or there exist $h_{1}, h_{2} \in \Poly_{2}^{\mc}(K)$ such that $[f] = \left[ h_{1} \circ h_{2} \right]$ and $[g] = \left[ h_{2} \circ h_{1} \right]$ in $\mathcal{P}_{4}(K)$. In the latter situation, there are $\phi, \psi \in \Aff(K)$ such that $f = \phi \centerdot \left( h_{1} \circ h_{2} \right)$ and $g = \psi \centerdot \left( h_{2} \circ h_{1} \right)$ in $\Poly_{4}(K)$, and this yields $f = k_{1} \circ k_{2}$ and $g = k_{2} \circ k_{1}$, where $k_{1} = \phi \circ h_{1} \circ \psi^{-1}$ and $k_{2} = \psi \circ h_{2} \circ \phi^{-1}$. Thus, the proposition is proved.
\end{proof}

\providecommand{\bysame}{\leavevmode\hbox to3em{\hrulefill}\thinspace}
\providecommand{\MR}{\relax\ifhmode\unskip\space\fi MR }
\providecommand{\MRhref}[2]{%
	\href{http://www.ams.org/mathscinet-getitem?mr=#1}{#2}
}
\providecommand{\href}[2]{#2}


\begin{thebibliography}{MFK94}

\bibitem[Bea91]{B1991}
Alan~F. Beardon, \emph{Iteration of rational functions}, Graduate Texts in
Mathematics, vol. 132, Springer-Verlag, New York, 1991, Complex analytic
dynamical systems. \MR{1128089}

\bibitem[Ben03]{Be2003}
Robert~L. Benedetto, \emph{Non-{A}rchimedean holomorphic maps and the {A}hlfors
	{I}slands theorem}, Amer. J. Math. \textbf{125} (2003), no.~3, 581--622.
\MR{1981035}

\bibitem[Ben08]{B2008}
\bysame, \emph{An {A}hlfors islands theorem for non-{A}rchimedean meromorphic
	functions}, Trans. Amer. Math. Soc. \textbf{360} (2008), no.~8, 4099--4124.
\MR{2395165}

\bibitem[Ben19]{B2019}
\bysame, \emph{Dynamics in one non-archimedean variable}, Graduate Studies in
Mathematics, vol. 198, American Mathematical Society, Providence, RI, 2019.
\MR{3890051}

\bibitem[Ber00]{B2000}
Walter Bergweiler, \emph{The role of the {A}hlfors five islands theorem in
	complex dynamics}, Conform. Geom. Dyn. \textbf{4} (2000), 22--34.
\MR{1741773}

\bibitem[BH88]{BH1988}
Bodil Branner and John~H. Hubbard, \emph{The iteration of cubic polynomials.
	{I}. {T}he global topology of parameter space}, Acta Math. \textbf{160}
(1988), no.~3-4, 143--206. \MR{945011}

\bibitem[Buf03]{Bu2003}
Xavier Buff, \emph{On the {B}ieberbach conjecture and holomorphic dynamics},
Proc. Amer. Math. Soc. \textbf{131} (2003), no.~3, 755--759. \MR{1937413}

\bibitem[Car22]{C2022}
Francisco~C. Caramello, Jr., \emph{Introduction to orbifolds},
arXiv:1909.08699v6, 2022.

\bibitem[CG93]{CG1993}
Lennart Carleson and Theodore~W. Gamelin, \emph{Complex dynamics},
Universitext: Tracts in Mathematics, Springer-Verlag, New York, 1993.
\MR{1230383}

\bibitem[CK97]{CK1997}
Jacek Ch\k{a}dzy\'{n}ski and Tadeusz Krasi\'{n}ski, \emph{A set on which the
	{{\L}}ojasiewicz exponent at infinity is attained}, Ann. Polon. Math.
\textbf{67} (1997), no.~2, 191--197. \MR{1460600}

\bibitem[DeM16]{DM2016}
Laura DeMarco, \emph{Bifurcations, intersections, and heights}, Algebra Number
Theory \textbf{10} (2016), no.~5, 1031--1056. \MR{3531361}

\bibitem[DH84]{DH1984}
A.~Douady and J.~H. Hubbard, \emph{\'{E}tude dynamique des polyn\^{o}mes
	complexes. {P}artie {I}}, Publications Math\'{e}matiques d'Orsay
[Mathematical Publications of Orsay], vol. 84-2, Universit\'{e} de Paris-Sud,
D\'{e}partement de Math\'{e}matiques, Orsay, 1984. \MR{762431}

\bibitem[DM08]{DMMM2008}
Laura~G. DeMarco and Curtis~T. McMullen, \emph{Trees and the dynamics of
	polynomials}, Ann. Sci. \'{E}c. Norm. Sup\'{e}r. (4) \textbf{41} (2008),
no.~3, 337--382. \MR{2482442}

\bibitem[Dom97]{D1997}
P.~Dom\'{\i}nguez, \emph{Connectedness properties of {J}ulia sets of
	transcendental entire functions}, Complex Variables Theory Appl. \textbf{32}
(1997), no.~3, 199--215. \MR{1457686}

\bibitem[Dom98]{D1998}
\bysame, \emph{Dynamics of transcendental meromorphic functions}, Ann. Acad.
Sci. Fenn. Math. \textbf{23} (1998), no.~1, 225--250. \MR{1601879}

\bibitem[EL92]{EL1992}
A.~\`{E}. Er\"{e}menko and G.~M. Levin, \emph{Estimation of the characteristic
	exponents of a polynomial}, Teor. Funktsi\u{\i} Funktsional. Anal. i
Prilozhen. (1992), no.~58, 30--40. \MR{1258059}

\bibitem[Fav24]{F2024}
Charles Favre, \emph{Blow-up of multipliers in meromorphic families of rational
	maps}, In preparation, 2024.

\bibitem[FG22]{FG2022}
Charles Favre and Thomas Gauthier, \emph{The arithmetic of polynomial dynamical
	pairs}, Annals of Mathematics Studies, vol. 214, Princeton University Press,
Princeton, NJ, 2022. \MR{4529887}

\bibitem[FT08]{FT2008}
Masayo Fujimura and Masahiko Taniguchi, \emph{A compactification of the moduli
	space of polynomials}, Proc. Amer. Math. Soc. \textbf{136} (2008), no.~10,
3601--3609. \MR{2415044}

\bibitem[Fuj07]{F2007}
Masayo Fujimura, \emph{The moduli space of rational maps and surjectivity of
	multiplier representation}, Comput. Methods Funct. Theory \textbf{7} (2007),
no.~2, 345--360. \MR{2376676}

\bibitem[Gon24]{G2024}
Chen Gong, \emph{Multipliers of rational maps and rescaling limits}, In
preparation, 2024.

\bibitem[Gor16]{G2016}
Igors Gorbovickis, \emph{Algebraic independence of multipliers of periodic
	orbits in the space of polynomial maps of one variable}, Ergodic Theory
Dynam. Systems \textbf{36} (2016), no.~4, 1156--1166. \MR{3492973}

\bibitem[Got23]{G2023}
Rin Gotou, \emph{Dynamical systems of correspondences on the projective line
	{II}: degrees of multiplier maps}, arXiv:2309.15404v1, 2023.

\bibitem[GOV20]{GOV2020}
Thomas Gauthier, Y\^{u}suke Okuyama, and Gabriel Vigny, \emph{Approximation of
	non-archimedean {L}yapunov exponents and applications over global fields},
Trans. Amer. Math. Soc. \textbf{373} (2020), no.~12, 8963--9011. \MR{4177282}

\bibitem[Gro65]{G1965}
A.~Grothendieck, \emph{\'{E}l\'{e}ments de g\'{e}om\'{e}trie alg\'{e}brique.
	{IV}. \'{E}tude locale des sch\'{e}mas et des morphismes de sch\'{e}mas.
	{II}}, Inst. Hautes \'{E}tudes Sci. Publ. Math. (1965), no.~24, 231.
\MR{199181}

\bibitem[HT13]{HT2013}
Benjamin Hutz and Michael Tepper, \emph{Multiplier spectra and the moduli space
	of degree 3 morphisms on {$\mathbb{P}^{1}$}}, JP J. Algebra Number Theory
Appl. \textbf{29} (2013), no.~2, 189--206. \MR{3136590}

\bibitem[Hub06]{H2006}
John~Hamal Hubbard, \emph{Teichm\"{u}ller theory and applications to geometry,
	topology, and dynamics. {V}ol. 1}, Matrix Editions, Ithaca, NY, 2006,
Teichm\"{u}ller theory, With contributions by Adrien Douady, William Dunbar,
Roland Roeder, Sylvain Bonnot, David Brown, Allen Hatcher, Chris Hruska and
Sudeb Mitra, With forewords by William Thurston and Clifford Earle.
\MR{2245223}

\bibitem[Ing12]{I2012}
Patrick Ingram, \emph{A finiteness result for post-critically finite
	polynomials}, Int. Math. Res. Not. IMRN (2012), no.~3, 524--543. \MR{2885981}

\bibitem[JX23]{JX2023}
Zhuchao Ji and Junyi Xie, \emph{Homoclinic orbits, multiplier spectrum and
	rigidity theorems in complex dynamics}, Forum Math. Pi \textbf{11} (2023),
Paper No. e11, 37. \MR{4585467}

\bibitem[JX24]{JX2024}
\bysame, \emph{The multiplier spectrum morphism is generically injective},
arXiv:2309.15382v2, 2024.

\bibitem[Kiw15]{K2015}
Jan Kiwi, \emph{Rescaling limits of complex rational maps}, Duke Math. J.
\textbf{164} (2015), no.~7, 1437--1470. \MR{3347319}

\bibitem[Lan83]{L1983}
Serge Lang, \emph{Fundamentals of {D}iophantine geometry}, Springer-Verlag, New
York, 1983. \MR{715605}

\bibitem[Lan02]{La2002}
\bysame, \emph{Algebra}, third ed., Graduate Texts in Mathematics, vol. 211,
Springer-Verlag, New York, 2002. \MR{1878556}

\bibitem[Liu02]{Li2002}
Qing Liu, \emph{Algebraic geometry and arithmetic curves}, Oxford Graduate
Texts in Mathematics, vol.~6, Oxford University Press, Oxford, 2002,
Translated from the French by Reinie Ern\'{e}, Oxford Science Publications.
\MR{1917232}

\bibitem[Luo22]{L2022}
Yusheng Luo, \emph{Trees, length spectra for rational maps via barycentric
	extensions, and {B}erkovich spaces}, Duke Math. J. \textbf{171} (2022),
no.~14, 2943--3001. \MR{4491710}

\bibitem[Mat86]{M1986}
Hideyuki Matsumura, \emph{Commutative ring theory}, Cambridge Studies in
Advanced Mathematics, vol.~8, Cambridge University Press, Cambridge, 1986,
Translated from the Japanese by M. Reid. \MR{879273}

\bibitem[McM87]{MM1987}
Curt McMullen, \emph{Families of rational maps and iterative root-finding
	algorithms}, Ann. of Math. (2) \textbf{125} (1987), no.~3, 467--493.
\MR{890160}

\bibitem[McM88]{MM1988}
\bysame, \emph{Automorphisms of rational maps}, Holomorphic functions and
moduli, {V}ol. {I} ({B}erkeley, {CA}, 1986), Math. Sci. Res. Inst. Publ.,
vol.~10, Springer, New York, 1988, pp.~31--60. \MR{955807}

\bibitem[MFK94]{MFK1994}
D.~Mumford, J.~Fogarty, and F.~Kirwan, \emph{Geometric invariant theory}, third
ed., Ergebnisse der Mathematik und ihrer Grenzgebiete (2) [Results in
Mathematics and Related Areas (2)], vol.~34, Springer-Verlag, Berlin, 1994.
\MR{1304906}

\bibitem[Mil92]{M1992}
John Milnor, \emph{Remarks on iterated cubic maps}, Experiment. Math.
\textbf{1} (1992), no.~1, 5--24. \MR{1181083}

\bibitem[Mil93]{M1993}
\bysame, \emph{Geometry and dynamics of quadratic rational maps}, Experiment.
Math. \textbf{2} (1993), no.~1, 37--83, With an appendix by the author and
Lei Tan. \MR{1246482}

\bibitem[Mil06]{M2006}
\bysame, \emph{On {L}att\`{e}s maps}, Dynamics on the {R}iemann sphere, Eur.
Math. Soc., Z\"{u}rich, 2006, pp.~9--43. \MR{2348953}

\bibitem[Mil17]{M2017}
J.~S. Milne, \emph{Algebraic groups}, Cambridge Studies in Advanced
Mathematics, vol. 170, Cambridge University Press, Cambridge, 2017, The
theory of group schemes of finite type over a field. \MR{3729270}

\bibitem[MP94]{MP1994}
Patrick Morton and Pratiksha Patel, \emph{The {G}alois theory of periodic
	points of polynomial maps}, Proc. London Math. Soc. (3) \textbf{68} (1994),
no.~2, 225--263. \MR{1253503}

\bibitem[MS95]{MS1995}
Patrick Morton and Joseph~H. Silverman, \emph{Periodic points, multiplicities,
	and dynamical units}, J. Reine Angew. Math. \textbf{461} (1995), 81--122.
\MR{1324210}

\bibitem[Neu99]{N1999}
J\"{u}rgen Neukirch, \emph{Algebraic number theory}, Grundlehren der
mathematischen Wissenschaften [Fundamental Principles of Mathematical
Sciences], vol. 322, Springer-Verlag, Berlin, 1999, Translated from the 1992
German original and with a note by Norbert Schappacher, With a foreword by G.
Harder. \MR{1697859}

\bibitem[Now00]{N2000}
Krzysztof~Jan Nowak, \emph{Some elementary proofs of {P}uiseux's theorems},
Univ. Iagel. Acta Math. (2000), no.~38, 279--282. \MR{1812118}

\bibitem[Oku12]{O2012}
Y\^{u}suke Okuyama, \emph{Repelling periodic points and logarithmic
	equidistribution in non-archimedean dynamics}, Acta Arith. \textbf{152}
(2012), no.~3, 267--277. \MR{2885787}

\bibitem[Pak19a]{P2019a}
F.~Pakovich, \emph{On mutually semiconjugate rational functions}, Arnold Math.
J. \textbf{5} (2019), no.~2-3, 339--354. \MR{4031360}

\bibitem[Pak19b]{P2019b}
Fedor Pakovich, \emph{Recomposing rational functions}, Int. Math. Res. Not.
IMRN (2019), no.~7, 1921--1935. \MR{3938311}

\bibitem[Poo17]{P2017}
Bjorn Poonen, \emph{Rational points on varieties}, Graduate Studies in
Mathematics, vol. 186, American Mathematical Society, Providence, RI, 2017.
\MR{3729254}

\bibitem[QY09]{QY2009}
WeiYuan Qiu and YongCheng Yin, \emph{Proof of the {B}ranner-{H}ubbard
	conjecture on {C}antor {J}ulia sets}, Sci. China Ser. A \textbf{52} (2009),
no.~1, 45--65. \MR{2471515}

\bibitem[SGA70]{SGA1970}
\emph{Sch\'{e}mas en groupes. {I}: {P}ropri\'{e}t\'{e}s g\'{e}n\'{e}rales des
	sch\'{e}mas en groupes}, Lecture Notes in Mathematics, vol. Vol. 151,
Springer-Verlag, Berlin-New York, 1970, S\'{e}minaire de G\'{e}om\'{e}trie
Alg\'{e}brique du Bois Marie 1962/64 (SGA 3), Dirig\'{e} par M. Demazure et
A. Grothendieck. \MR{274458}

\bibitem[SGA71]{SGA1971}
\emph{Rev\^{e}tements \'{e}tales et groupe fondamental}, Lecture Notes in
Mathematics, vol. Vol. 224, Springer-Verlag, Berlin-New York, 1971,
S\'{e}minaire de G\'{e}om\'{e}trie Alg\'{e}brique du Bois Marie 1960--1961
(SGA 1), Dirig\'{e} par Alexandre Grothendieck. Augment\'{e} de deux
expos\'{e}s de M. Raynaud. \MR{354651}

\bibitem[Sil98]{S1998}
Joseph~H. Silverman, \emph{The space of rational maps on
	{$\boldsymbol{P}^{1}$}}, Duke Math. J. \textbf{94} (1998), no.~1, 41--77.
\MR{1635900}

\bibitem[Sil07]{S2007}
\bysame, \emph{The arithmetic of dynamical systems}, Graduate Texts in
Mathematics, vol. 241, Springer, New York, 2007. \MR{2316407}

\bibitem[Sil12]{S2012}
\bysame, \emph{Moduli spaces and arithmetic dynamics}, CRM Monograph Series,
vol.~30, American Mathematical Society, Providence, RI, 2012. \MR{2884382}

\bibitem[Sug17]{S2017}
Toshi Sugiyama, \emph{The moduli space of polynomial maps and their fixed-point
	multipliers}, Adv. Math. \textbf{322} (2017), 132--185. \MR{3720796}

\bibitem[Sug23]{S2023}
\bysame, \emph{The moduli space of polynomial maps and their fixed-point
	multipliers: {II}. {I}mprovement to the algorithm and monic centered
	polynomials}, Ergodic Theory Dynam. Systems \textbf{43} (2023), no.~11,
3777--3795. \MR{4651585}

\bibitem[VH92]{VH1992}
Franco Vivaldi and Spyros Hatjispyros, \emph{Galois theory of periodic orbits
	of rational maps}, Nonlinearity \textbf{5} (1992), no.~4, 961--978.
\MR{1174226}

\end{thebibliography}
\end{document}